\documentclass[a4paper, bibliography=totoc]{scrartcl}
\usepackage{latexsym,amssymb}
\usepackage{amsmath,amsthm,mathtools}
\usepackage{thmtools} 
\usepackage{thm-restate} 
\usepackage[UKenglish]{babel}
\usepackage[T1]{fontenc}
\usepackage[utf8]{inputenc}
\usepackage{hyperref}
\usepackage{enumerate}
\usepackage[mathscr]{euscript}
\usepackage{color}
\usepackage{xcolor}

\let\eps\varepsilon
\let\phi\varphi
\def\dist{\mathrm{dist}}
\def\P{\mathcal{P}}
\def\R{\mathbb{R}}
\def\N{\mathbb{N}}

\def\1{\mathbbm{1}}
\def\Lip{\mathrm{Lip}}

\def\spr#1#2{\langle{#1,#2}\rangle}
\def\D{\mathcal{D}} 
\def\leb{\mathcal{L}} 
   
\newcommand{\mytau}{\tau}
\newcommand{\myeps}{\theta}
\newcommand{\myepszero}{\theta}
\newcommand{\myepsone}{\theta'}
\newcommand{\tauzero}{\tau}
\newcommand{\tauone}{\tau'}
\newcommand{\epszero}{\eps}
\newcommand{\epsone}{\eps'}

\newcommand{\Nzero}{N}
\newcommand{\None}{N'}
\newcommand{\myseq}{\eps}
\newcommand{\myeta}{\theta}
\newcommand{\myr}{\rho}
\newcommand{\myI}{R}

\newcommand{\myconst}{250}
\newcommand{\mycons}{240}
\newcommand{\comp}{^\text{comp}}
\newcommand{\lebm}[1]{\mathcal L(#1)}
\newcommand{\lebmi}[1]{\mathcal L(#1)}
\newcommand{\lip}{\operatorname{Lip}}
\newcommand{\abs}[1]{\left|#1\right|}
\newcommand{\norm}[1]{\left\|#1\right\|}
\newcommand{\lnorm}[2]{\left\|#2\right\|_#1}
\newcommand{\inorm}[1]{\left\|#1\right\|_\infty}
\newcommand{\set}[1]{\left\{#1\right\}}
\newcommand{\br}[1]{\left(#1\right)}
\newcommand{\sqbr}[1]{\left[#1\right]}
\newcommand{\sk}[2]{\langle{#1},{#2}\rangle}

\newcommand{\osci}[2]{\operatorname{osc}_{#1}(#2)}

\newcommand{\diam}{\operatorname{diam}}

\newcommand{\interior}{\operatorname{Int}}

\newcommand{\image}{\operatorname{Im}}
\newcommand{\Diff}{\operatorname{Diff}}

\newcommand{\barf}{\overline{f}}
\newcommand{\barg}{\overline{g}}
\newcommand{\metric}{\zeta}
\newcommand{\A}{\mathscr{A}}
\newcommand{\nullA}{M}
\newcommand{\Fset}{X}
\newcommand{\ess}{\operatorname{ess}}
\newcommand{\lipd}{\lip_{1}([0,1]^{d},\R)}
\newcommand{\lipdl}{\lip_{1}([0,1]^{d},\R^{l})}
\newcommand{\lipint}{\lip_{1}([0,1],\R)}
\newcommand{\piset}{G}
\newcommand{\piiset}{H}
\newcommand{\good}{\mathbb G}
\newcommand{\bad}{\mathbb B}
\newcommand{\ggood}{G}
\newcommand{\bbad}{B}
\newcommand{\myd}{\rho}
\newcommand{\Sphere}{\mathbb{S}}
\newcommand{\Fonedim}{F}
\newcommand{\Fhidim}{\mathscr{F}}
\newcommand{\Finf}{F_\infty}
\newcommand{\Einf}{E_\infty}

\newtheorem{theorem}{Theorem}[section]
\newtheorem{prp}[theorem]{Proposition}
\newtheorem{lemma}[theorem]{Lemma}
\newtheorem{cor}[theorem]{Corollary}

\theoremstyle{definition}
\newtheorem{defn}[theorem]{Definition}
\newtheorem{remark}[theorem]{Remark}
\newtheorem*{remark*}{Remark}
\numberwithin{equation}{section}
\title{
	A dichotomy of sets via typical~differentiability.
}
\author{Michael Dymond\thanks{The first named author acknowledges the support of Austrian Science Fund (FWF): P 30902-N35 and of the EPSRC grant EP/N027531/1.}, Olga Maleva\thanks{The second named author acknowledges the support of the EPSRC grant EP/N027531/1.}}
\begin{document}
	\maketitle
	\begin{abstract}
		We obtain a criterion for an analytic subset of a Euclidean space to contain points of differentiability of a typical Lipschitz function, namely, that it cannot be covered by countably many sets, each of which is closed and purely unrectifiable (has zero length intersection with every $C^1$ curve). 
		Surprisingly, we establish 
		that any set failing this criterion witnesses the opposite extreme of typical behaviour: In any such coverable set a typical Lipschitz function is everywhere severely non-differentiable.
	\end{abstract}

\section{Introduction}
Whilst the classical
Rademacher Theorem guarantees that 
every set of positive (outer) Lebesgue measure in a Euclidean space $\R^d$ contains points of differentiability of every Lipschitz function on $\R^d$,
a major direction in geometric measure theory research of the last two decades 
was to explore to what extent this is true for Lebesgue null subsets of $\R^d$. 
It was shown in the 1940s~\cite{choquet,zahorski} that for any null set $N\subseteq\R$ there is a Lipschitz function $f\colon\R\to\R$ nowhere differentiable in $N$.
In contrast, 
for any $d\ge2$ there are Lebesgue null sets in which every Lipschitz function $\R^{d}\to\R$ has points of differentiability, see~\cite{Preiss_1990,Dore_Maleva3,dymond_maleva2016}. 
	Sets with the latter property are called universal differentiability sets (UDS). 
	
But if there is a Lipschitz function nowhere differentiable on a given set $N$, 
one naturally wonders what happens with a typical (in the sense of Baire category -- see exact definition below) Lipschitz function on $N$.
Classical results suggest that typical functions exhibit the worst possible differentiability behaviour, e.g.\ a typical continuous function on an interval is nowhere differentiable, see~\cite{Banach1931}.
Surprisingly, the complete opposite may be true in spaces of Lipschitz functions, even in spaces of Lipschitz functions restricted to some non-UDS $N$.
In dimension one,~\cite{preiss_tiser94}
shows that 
$N\subseteq\R$ can be covered by a countable union of closed null sets if and only if  a typical $1$-Lipschitz function $\R\to\R$ has no points of differentiability in $N$.
It can be seen from the proof in~\cite{preiss_tiser94} that for all other analytic sets, a typical $1$-Lipschitz function will be differentiable at a point inside the set. 

In the present paper we settle the question of differentiability of a typical Lipschitz function inside a given analytic subset $N$ of $\R^d$, $d\ge2$. We
give a complete characterisation of the subsets $N$ of $\R^d$ in which a typical $1$-Lipschitz function has points of differentiability:
they cannot be covered by an $F_\sigma$ purely unrectifiable set; we refer to such sets as \emph{typical differentiability sets} (a simple example is a $C^1$-curve in $\R^d$). 
We also show that for all remaining sets $N$
a typical $1$-Lipschitz function is nowhere differentiable, even directionally, inside~$N$.

We formally state our main results in the next section; see Theorems~\ref{thm:typ_ds} and~\ref{thm:typ_nds}, which imply a dichotomy between typical differentiability and typical non-differentiability sets for every dimension $d\ge1$, see Corollary~\ref{cor:typ_dich}.

Note that universal differentiability sets form a subclass of typical differentiability sets. 
Although to date there is no geometric-measure criterion for a set to be a UDS,
it has been established that UDS may be extremely small, e.g.\ compact and have Minkowski dimension $1$, see~\cite{dymond_maleva2016}. This demonstrates the extent to which the $F_\sigma$-null criterion from~\cite{preiss_tiser94} fails in higher dimensions: in dimension one countable unions of closed null sets are typical non-differentiability sets, but in all higher dimensions they may actually capture differentiability points of every Lipschitz function.	
We expect that, in the same spirit as for UDS, typical differentiability sets 
will be explored further, in particular, providing insight into typical behaviour of Lipschitz functions on non-Euclidean spaces; in this context one should mention recent research into UDS in
Heisenberg and, more generally, Carnot groups~\cite{UDS1,UDS2,UDS3}. 

Let us be more precise about the terminology we use. The present paper will not be excessively concerned with the measurability of subsets of Euclidean spaces, and so we will use the term \emph{measure} in the sense of Hausdorff measure, as in~\cite{mattila_1995}. This includes both the Lebesgue and outer Lebesgue one-dimensional measure, which we denote by $\leb$. 
A Lipschitz mapping with Lipschitz constant less than or equal to one is referred to as $1$-Lipschitz; let $\lip_{1}([0,1]^{d})$ denote the set of all $1$-Lipschitz functions $f:[0,1]^d\to\R$, viewed as a complete metric space when equipped with the metric
$\rho(f,g)=\lnorm{\infty}{g-f}$. For any Lipschitz mapping $f$ let $\Diff(f)$ denote the set of $t$ such that $f$ is differentiable at $t$. We say that a typical $1$-Lipschitz function has a certain property, if the set of those $f\in\lip_{1}([0,1]^{d})$ with this property is a residual subset of $\lip_{1}([0,1]^{d})$, i.e.\ its complement is meagre (in other words, is of first category). 

We refer to a set $S\subseteq(0,1)^d$ as \emph{typical differentiability set} if a typical $1$-Lipschitz function has points of differentiability in $S$, i.e.\ $\Diff(f)\cap S\ne\emptyset$. Let us also refer to subsets of $(0,1)^d$ in which a typical $1$-Lipschitz function has no points of differentiability as \emph{typical non-differentiability sets}.
A priori, a set $S\subseteq(0,1)^d$ may have exactly one of these two properties, or none;
we show in Corollary~\ref{cor:typ_dich} that for analytic $S$ `none' is impossible.

We would like to add that a very recent advance in this area, primarily for vector-valued Lipschitz mappings to Euclidean spaces of at least the same dimension, was made by Merlo~\cite{merlo}. 

It is worth mentioning further specific details of the aforementioned works~\cite{preiss_tiser94} and~\cite{merlo} which are of relevance to the present paper. Recall that~\cite{preiss_tiser94} characterises typical non-differentiability sets in $[0,1]$ as those sets which can be covered by countably many closed sets of measure zero. It also gives a sufficient condition for a set to be a typical differentiability set, via the property of having `every portion of positive measure'.
We now give a definition of this notion and its higher dimensional analogue.
\begin{defn}\label{def:every_portion}
\begin{enumerate}[(i)]
	\item\label{ev_prt_pos_ms} We say that a set $\Fonedim\subseteq \R$ has \emph{every portion of positive measure} if for every open set $U\subseteq \R$ with $U\cap \Fonedim\neq \emptyset$ we have that $\leb(U\cap \Fonedim)$ is positive.
	
	\item\label{ev_prt_pos_cw} 	We say that a set $\Fhidim\subseteq \R^{d}$ has \emph{every portion of positive cone width} if for every open set $U\subseteq \R^{d}$ with $U\cap \Fhidim\ne\emptyset$ there exists a $C^1$-smooth curve  $\nu_U\colon [0,1]\to\R^{d}$ with nowhere zero derivative such that $\leb(\nu_U^{-1}(U\cap \Fhidim))$ is positive.
	\end{enumerate}
\end{defn}
\begin{remark}\label{rem:portion}
If a set $\Fhidim$ has every portion of positive cone width and $a>0$, then the curve $\nu_U$ may always be chosen so that it additionally satisfies $\norm{\nu_U'(t)}=a$ for all $t\in(0,1)$.

Also, in Section~\ref{sec:construction} we introduce the notation $\Gamma_\Fhidim(U)$, to denote the collection of all $C^{1}$-smooth curves $\gamma$ with codomain $U$ and $\leb(\gamma^{-1}(\Fhidim))>0$. We may note here that if
$\Fhidim$ has every portion of positive cone width, the set $U$ is open with $U\cap \Fhidim\ne\emptyset$ and $a>0$, 
then there exists $\nu_U\in\Gamma_\Fhidim(U)$ such that
$\norm{\nu_U'(t)}=a$ for all~$t$.
\end{remark}
Note that the two notions~\eqref{ev_prt_pos_ms} and~\eqref{ev_prt_pos_cw} coincide in dimension $d=1$. Petruska~\cite[Theorem~1]{petruska93} proves that analytic subsets of $[0,1]$ not coverable by a union of countably many closed, measure zero sets can be characterised as those sets $E\subseteq [0,1]$ for which there exists a closed set $\Fonedim\subseteq [0,1]$ having every portion of positive measure such that $E\cap \Fonedim$ is relatively residual in $\Fonedim$. 

\begin{defn} We will use the term \emph{Lipschitz curve} to refer to a Lipschitz mapping $\gamma\colon I\to\R^{d}$, where $I\subseteq \R$ is a closed interval, with the property that the derivative $\gamma'$ is bounded away in magnitude from zero almost everywhere.
	
A set $P\subseteq \R^{d}$ is said to be \emph{purely unrectifiable} if for every Lipschitz curve $\gamma\colon [0,1]\to\R^{d}$ the set $\gamma^{-1}(P)$ has Lebesgue measure zero.
\end{defn}
The class of purely unrectifiable sets is widely regarded as the most exceptional in relation to differentiability of Lipschitz functions. Moreover, recently M\'athe has announced that, within the class of Borel sets, purely unrectifiable sets coincide with the formally smaller class of uniformly purely unrectifiable sets (see~\cite{maleva_preiss2018}, Definition~1.4 and Remark 1.7). Alberti, Cs\"ornyei and Preiss prove in~\cite{acp2010differentiability} that any uniformly purely unrectifiable set $P\subseteq \R^{d}$ admits a Lipschitz function $f\colon \R^{d}\to\R$ which fails to have any directional derivatives in the set $P$. A strengthening of this is proved by the second named author and Preiss in~\cite[Theorem~1.13]{maleva_preiss2018}: such a function $f$ may be constructed so that at all $x\in P$, the function $f$ is non-differentiable at $x$ in the strongest possible sense:
\begin{equation*}
\liminf_{r\to 0}\sup_{\norm{y}\leq r}\frac{\abs{f(x+y)-f(x)-\sk{e}{y}}}{r}=0
\end{equation*}
for \emph{every} $e\in \R^{d}$ with $\norm{e}\leq 1$. This condition expresses that \emph{every} linear mapping $\R^{d}\to\R$ of norm at most one behaves as the derivative of $f$ along a certain subsequence approaching $x$.
In Section~\ref{sec.proof} we show that the results of~\cite{maleva_preiss2018} are extremely relevant to typical non-differentiability; see Theorem~\ref{thm:main:nondiff}.

To find a characterisation of typical differentiability sets in higher dimensional Euclidean spaces, one might seek higher dimensional analogues of interval subsets not coverable by unions of countably many closed null sets. However, as explained earlier, the same notion cannot work, in particular because there are closed, null universal differentiability sets. We verify that countable unions of closed purely unrectifiable sets, which coincide with countable unions of closed null sets in the case $d=1$ are the fitting choice; see the characterisation given in Theorems~\ref{thm:typ_ds} and~\ref{thm:typ_nds}.
Merlo~\cite{merlo} also proposes that the correct higher dimensional analogues of typical non-differentiability sets for vector-valued Lipschitz mappings are those subsets of $[0,1]^{d}$ which can be covered by a union of countably many closed, purely unrectifiable sets.

\textbf{Acknowledgement} The authors would like to thank David Preiss for very helpful and valuable discussions.
We would also like to thank the organisers of the $48^\text{th}$ Winter School in Abstract Analysis, Czech Republic, and in particular Martin Rmoutil for stimulating conversations.

\section{Main Results}
\subsection{Statement of main results}
In the present section we set out the structure of the proof of our main results, Theorems~\ref{thm:typ_ds} and~\ref{thm:typ_nds}:
\begin{theorem}\label{thm:typ_ds}
	Let $d\geq 1$. The following are equivalent for an analytic set $\A\subseteq (0,1)^{d}$:
	\begin{enumerate}[(a)]
		\item\label{ncoverable} The set $\A$ cannot be covered by an $F_{\sigma}$, purely unrectifiable set.
		\item\label{diff_pts} A typical $f\in\lip_{1}([0,1]^{d})$ has points of differentiability in $\A$, \\ i.e.\ $\A$ is a typical differentiability subset of $(0,1)^d$.
	\end{enumerate}
\end{theorem}
\begin{theorem}\label{thm:typ_nds}
	Let $d\geq 1$. The following are equivalent for an analytic set $\A\subseteq (0,1)^{d}$:
	\begin{enumerate}[(i)]
		\item\label{coverable} The set $\A$ is contained in an $F_{\sigma}$, purely unrectifiable set.
		\item\label{nwh_diff} A typical $f\in\lip_{1}([0,1]^{d})$ is nowhere differentiable in $\A$, \\ i.e.\ $\A$ is a typical non-differentiability subset of $(0,1)^d$.
	\end{enumerate}
\end{theorem}
We caution again that Theorems~\ref{thm:typ_ds} and~\ref{thm:typ_nds}  are not formally equivalent statements, i.e.\ the negation of~\eqref{nwh_diff} is not formally the same as~\eqref{diff_pts}. Thus, the following dichotomy is also a new result which follows from Theorems~\ref{thm:typ_ds}~and~\ref{thm:typ_nds}.
\begin{cor} \label{cor:typ_dich}
	Let $d\geq 1$. Each analytic subset $\A\subseteq (0,1)^{d}$ belongs to exactly one of the following two classes: typical differentiability or typical non-differentiability sets.
\end{cor}
\begin{remark}
Note that a typical differentiability set $\A$ may be purely unrectifiable. As an example, we may take $\A$ as 
a $1$-dimensional, Lebesgue null, $G_\delta$ set dense in $[0,1]$, embedded in $[0,1]^d$. Although by~\cite[Theorem~1.13]{maleva_preiss2018},
there is a Lipschitz function non-differentiable in $\A$ in the strongest possible sense, Theorem~\ref{thm:typ_ds} guarantees that  a \emph{typical} Lipschitz function has differentiability points in $\A$.
\end{remark}

Theorems~\ref{thm:typ_ds} and~\ref{thm:typ_nds} in dimension $d=1$ coincide with the results proved by Preiss and Ti\v ser in~\cite{preiss_tiser94}, in this paper we provide a proof of the two statements for all dimensions $d\geq 1$. Also, as a corollary of the proof of Theorem~\ref{thm:typ_ds}, we obtain a strengthening of their typical differentiability result, see Remark~\ref{rem:first_typ_fn_typ_point}. 

Since conditions~\eqref{ncoverable} of Theorem~\ref{thm:typ_ds} and~\eqref{coverable} of Theorem~\ref{thm:typ_nds} are mutually exclusive, it is enough to prove only implications~\eqref{ncoverable}$\Rightarrow$\eqref{diff_pts} in Theorem~\ref{thm:typ_ds}, and~\eqref{coverable}$\Rightarrow$\eqref{nwh_diff} in Theorem~\ref{thm:typ_nds}. For convenience, we restate these as two new statements. Moreover, we include in these two statements additional details concerning special forms of differentiability and non-differentiability which, for simplicity, are omitted from Theorems~\ref{thm:typ_ds} and~\ref{thm:typ_nds}.
\begin{theorem}\label{thm:main_result}
	Let $d\geq 1$. If an analytic set $\A\subseteq (0,1)^{d}$ 
cannot be covered by an $F_{\sigma}$, purely unrectifiable set, then a typical $f\in\lip_{1}([0,1]^{d})$ has points of differentiability in $\A$. Such points $x\in\A$ may additionally be taken so that the gradient $\nabla f(x)$ of $f$ at $x$ has magnitude one.
\end{theorem}
In Theorem~\ref{thm:main:nondiff} we show that the non-differentiability of Theorem~\ref{thm:typ_nds} may be taken in a stronger sense. Namely, we prove that for each typical non-differentiability set $\A$ a typical function $f\in\lip_{1}([0,1]^{d})$ has no directional derivatives at every $x\in\A$ and, moreover, its derived set $\mathcal D f(x,v)$, defined below, coincides with $[-1,1]$, for each $\norm{v}=1$.
\begin{defn}\label{def:der_set}
Suppose that $f:[0,1]^d\to \R$ is a function and $x\in(0,1)^d$, $v\in\Sphere^{d-1}$ are two vectors.
The \emph{derived set} of $f$ at the point $x$  in the direction of $v$ is defined as the
set  $\mathcal D f(x,v)$ of all
existing limits $\lim_{n\to\infty} (f(x+t_nv)-f(x))/t_n$, where $t_n\searrow0$.
\end{defn}

\begin{restatable}{theorem}{mainnondiff}\label{thm:main:nondiff}
Let $d\ge1$.
If a set $\A\subseteq (0,1)^{d}$ can be covered by an $F_{\sigma}$, purely unrectifiable set, then a typical $f\in\lip_{1}([0,1]^{d})$ has no directional derivatives at every point of $\A$ and, moreover, for a typical $f\in\lip_{1}([0,1]^{d})$ it holds that $\mathcal Df(x,v)=[-1,1]$ for every $x\in\A$ and every $v\in\Sphere^{d-1}$.
\end{restatable}

To conclude, note that~\cite{merlo} provides a statement analogous to Theorem~\ref{thm:main_result} in spaces of vector-valued Lipschitz mappings $\R^d\to\R^m$,
with the restriction $m\ge d$, and with only directional differentiability instead of full differentiability. Although this statement might appear similar in spirit, we show in Section~\ref{section:proj} that projection arguments do not allow one to lower the codomain dimension to $1$, as we achieve in Theorem~\ref{thm:main_result}. On the other hand,
parts of the argument employed in \cite{merlo} apply to Lipschitz mappings without restriction on the dimension of the codomain and therefore Theorem~\ref{thm:main:nondiff} is proved implicitly there. 
However, in Section~\ref{sec.proof} of the present article we provide an independent shorter proof of Theorem~\ref{thm:main:nondiff}, using results of~\cite{maleva_preiss2018}. 

\subsection{Strategy of the proof of typical differentiability}
The proof of the `typical differentiability' Theorem~\ref{thm:main_result} roughly divides into two halves, proved in Sections~\ref{sec:typ_diff_curves} and~\ref{sec:construction}. In the first part, we prove the statement for the special case where $\A$ (or $\gamma(E)$ in the statement below) is a subset of a Lipschitz curve with unique tangents at all points in~$\A$.
\begin{restatable}{theorem}{bm}\label{thm:bm}
	Let $\emptyset\ne \Fonedim\subseteq[0,1]$ be a closed set with every portion of positive measure and let $E$ be a relatively residual subset of $\Fonedim$. Let $\gamma\colon [0,1]\to(0,1)^d$ be a Lipschitz curve with Lipschitz constant $1$, such that $\gamma$ is differentiable with derivative of magnitude one at each $t\in E$.
	Then the set $S$ of those functions $f\in\Lip_1([0,1]^{d})$ for which there exists $t\in E$ such that $f$ is differentiable with derivative of magnitude one at $\gamma(t)$ is residual in $\Lip_1([0,1]^{d})$.
\end{restatable}
\begin{remark}\label{rem:first_typ_fn_typ_point}
In the setting of Theorem~\ref{thm:bm}, it is possible to obtain the stronger conclusion that there is a residual subset $R$ of $\lip_{1}([0,1]^{d})$ for which every function $f\in R$ is differentiable at $\gamma(t)$ for residually many $t\in F$ (or, equivalently, for residually many $t\in E$). Loosely rephrased, a typical $f\in\lip_{1}([0,1]^{d})$ is differentiable at a typical point of $\gamma(F)$ (or a typical point of $\gamma(E)$). For further details see Remark~\ref{rem:additional-diff}.

Importantly, this is a new observation even in dimension $d=1$, where it asserts a stronger property of one-dimensional typical differentiability sets than that proved in~\cite{preiss_tiser94}; in particular it strengthens~\cite[Lemma~2]{preiss_tiser94}. Indeed, we may state the following extension of the results of \cite{preiss_tiser94}:

\textit{If an analytic set $\A\subseteq[0,1]$ cannot be covered by a one-dimensional Lebesgue null $F_\sigma$ set, then there exists a non-empty closed set $\Fonedim\subseteq[0,1]$  with every portion of positive measure 
	 and a residual set of functions $f\in\lip_1([0,1]^d)$
	 for which $\A\cap\Diff(f)$ is a residual subset of $\Fonedim$. The same conclusion holds for
	 any non-empty closed set $\Fonedim\subseteq [0,1]$ with every portion of positive measure such that $\A\cap\Fonedim$ is residual in $\Fonedim$.}
\end{remark}

Theorem~\ref{thm:bm} is proved in Section~\ref{sec:typ_diff_curves}. Then, in Section~\ref{sec:construction}, we show that the general statement of Theorem~\ref{thm:main_result} reduces to the special case of Theorem~\ref{thm:bm}. Put differently, we show that any set $\A\subseteq (0,1)^{d}$ satisfying the hypotheses of Theorem~\ref{thm:main_result} intersects some Lipschitz curve $\gamma$, with $\Lip(\gamma)\le1$, in the particular manner required by Theorem~\ref{thm:bm}. To achieve this, we make important use of the following proposition, which follows from~\cite{solecki1994}, cf.~\cite[Theorem~2.8]{merlo}.
It shows that analytic sets which cannot be covered by a countable union of closed purely unrectifiable sets, may be approximated by closed sets having every portion of positive cone width, see Definition~\ref{def:every_portion}~\eqref{ev_prt_pos_cw}.
\begin{prp}\label{th:sol}
If an analytic set $\A\subseteq(0,1)^d$ cannot be covered by a countable union of closed purely unrectifiable sets, then there exists a closed set $\Fhidim\subseteq[0,1]^d$, such that $\A\cap\Fhidim$ is residual in $\Fhidim$, and $\Fhidim$ has every portion of positive cone width.
\end{prp}
\begin{proof}
We apply~\cite[Remark~(2), p.~1024]{solecki1994} to the collection $I$ of all closed, purely unrectifiable sets
and set $\A$. We see that if $\A$ cannot be covered by a countable union of closed, purely unrectifiable sets, i.e.\ $\A\not\in I_{\text{ext}}$, then there is a non-empty closed set $\Fhidim=C$ such that $\A\cap\Fhidim$ contains a $G_\delta$ set, dense in $\Fhidim$ (implying that $\A\cap\Fhidim$ is residual in $\Fhidim$), and such that for any open set $V$ with $V\cap\Fhidim\ne\emptyset$ it holds that $\overline{V\cap\Fhidim}\not\in I$. In other words, $\overline{V\cap\Fhidim}$ is not a purely unrectifiable set, which implies that there exists a $C^1$-smooth curve $\gamma$, such that $\lebm{\gamma^{-1}(\overline{V\cap\Fhidim})}>0$, implying $\lebm{\gamma^{-1}(\overline V\cap\Fhidim)}>0$. Let $U$ be an open set with $U\cap\Fhidim\ne\emptyset$, let $x\in U\cap\Fhidim$. Choose $r>0$ such that $V=B(x,r)\subseteq \overline V\subseteq U$. If we take $\nu=\gamma$ as above, the condition of Definition~\ref{def:every_portion}~\eqref{ev_prt_pos_cw} is satisfied for $U$, and the statement follows. 
\end{proof}
With Proposition~\ref{th:sol} at hand, the reduction to the `special case' described above is completed by the next theorem.
\begin{restatable}{theorem}{key}\label{lemma:key}
	Let $d\ge1$ and $\Fhidim\subseteq [0,1]^{d}$ be a non-empty, closed set having every portion of positive cone width. Let $\A\subseteq(0,1)^{d}$ be an analytic set such that $\A\cap \Fhidim$ is relatively residual in $\Fhidim$. Then there exists a $1$-Lipschitz curve $\gamma\colon [0,1]\to(0,1)^{d}$ and sets $E\subseteq\Fonedim\subseteq [0,1]$ with the following properties: 
	\begin{enumerate}[(i)]
		\item\label{key1} $\Fonedim$ is non-empty, closed and has every portion of positive measure;
		\item\label{key2} $E$ is residual in $\Fonedim$;
		\item\label{key3} $\gamma$ is differentiable at every point $t\in E$ with $\norm{\gamma'(t)}=1$;
		\item\label{key3a} For every $t\in E$ we have
		\begin{equation*}
		\lim_{\delta\to 0}\osci{\gamma'}{[t-\delta,t+\delta]}=0;
		\end{equation*}
		\item\label{key4} $\gamma(E)\subseteq \A$.
	\end{enumerate}
\end{restatable}
The quantity $\osci{\gamma'}{[t-\delta,t+\delta]}$ of~\eqref{key3a} should be understood in the natural way; for a more precise definition see Section~\ref{sec:construction},~\eqref{eq:def-osc}. 
\begin{remark}
	We point out that Theorem~\ref{lemma:key} and Proposition~\ref{th:sol} combine to give the following statement, which may be viewed as a generalisation of the one-dimensional result of \cite{petruska93} to all higher dimensions:
\newline	
	\emph{An analytic set $\A\subseteq (0,1)^{d}$ cannot be covered by a countable union of closed, purely unrectifiable sets if and only if there exists a $1$-Lipschitz curve $\gamma\colon[0,1]\to (0,1)^{d}$ and a non-empty, closed set $F\subseteq[0,1]$ having every portion of positive measure such that $\gamma^{-1}(\A)\cap\Diff(\gamma)$ intersects $F$ in a relatively residual set.}
\end{remark}

To prove Theorem~\ref{lemma:key}, we construct a sequence $(\gamma_{k})_{k=1}^{\infty}$ of Lipschitz curves $\gamma_{k}$ converging uniformly to the desired curve $\gamma$. We postpone this construction until Section~\ref{sec:construction}. For now, let us present a proof of Theorem~\ref{thm:main_result} based on Theorems~\ref{thm:bm} and~\ref{lemma:key}, and Proposition~\ref{th:sol}. 
\begin{proof}[Proof of Theorem~\ref{thm:main_result}]
	By Proposition~\ref{th:sol}, there exists a closed set $\Fhidim\subseteq [0,1]^{d}$ such that $\A$ and $\Fhidim$ satisfy the conditions of Theorem~\ref{lemma:key}. Let $\gamma$, $E$ and $\Fonedim$ be given by the conclusion of Theorem~\ref{lemma:key}. Then $\gamma$, $E$ and $\Fonedim$ satisfy the conditions of Theorem~\ref{thm:bm}. Applying Theorem~\ref{thm:bm}, we conclude that a typical $f\in\lip_{1}([0,1]^{d})$ has points of differentiability where the derivative has magnitude one in $\gamma(E)\subseteq \A$.
\end{proof}

\subsection{Application in Universal Differentiability Set Theory}
Recall that purely unrectifiable sets fail badly to have the universal differentiability property. However, there are examples which show that such sets may provide surprisingly many differentiability points of some Lipschitz functions. Cs\"ornyei, Preiss and Ti\v ser construct in~\cite{CPT_2005} a universal differentiability set $E\subseteq \R^{2}$, a purely unrectifiable subset $P\subseteq E$ and a Lipschitz function $h\colon \R^{2}\to\R$ such that all differentiability points of $h$ in the universal differentiability set $E$ are captured by $P$, that is,
\begin{equation}\label{eq:all_diff_pts_pu}
\Diff(h)\cap E\subseteq P.
\end{equation}
In the new paper~\cite{dymond2019typical}, the first named author shows that by a modification of this construction, the set $P$ may additionally capture all differentiability points in $E$ of a typical Lipschitz function in the shifted $\lip_{1}$ space $X=h+\lip_{1}([0,1]^{2})$. In other words,~\eqref{eq:all_diff_pts_pu} holds not just for $h$, but for a typical $f\in X$.
This naturally invites the question of whether it is possible to find $E$ and $P$ so that~\eqref{eq:all_diff_pts_pu} holds for a typical $f$ in the natural space $\lip_{1}([0,1]^{d})$ without any shift. As an application of the dichotomy between typical differentiability and typical non-differentiability sets, see Theorems~\ref{thm:typ_ds} and~\ref{thm:typ_nds}, and Corollary~\ref{cor:typ_dich}, 
 we establish that this is not possible. Although Theorem~\ref{thm:app_uds} shows that purely unrectifiable sets cannot capture all points of differentiability of a typical Lipschitz function within a given universal differentiability set, the main result of~\cite{dymond2019typical} asserts that purely unrectifiable sets may nonetheless capture `equivalently' large sets of differentiability points of a typical Lipschitz function.
 
\begin{theorem}\label{thm:app_uds}
	Let $U\subseteq[0,1]^{d}$ be a universal differentiabiliity set and $V\subseteq U$ be a subset with the property that
	\begin{equation*}
	\Diff(f)\cap U\subseteq V
	\end{equation*}
	for a typical $f\in\lip_{1}([0,1]^{d})$. Then $V$ is  not a purely unrectifiable set.
\end{theorem}
\begin{proof}
By assumption, the set $U\setminus V$ is a typical non-differentiability set. Hence, Theorem~\ref{thm:typ_nds} implies that the set $U\setminus V$ is purely unrectifiable. If we assume that $V$ is also purely unrectifiable, we conclude that their union $U$ is purely unrectifiable, hence a cone unrectifiable set, see~\cite[Definition~1.7 and Remark~1.8]{maleva_preiss2018}. Applying~\cite[Theorem~1.1]{maleva_preiss2018} to the set $U$ we obtain a Lipschitz function $g$ which is non-differentiable everywhere in $U$, contrary to $U$ being a universal differentiability set.	
\end{proof}

\section{Typical differentiability inside Lipschitz curves}\label{sec:typ_diff_curves}
In this section we prove Theorem~\ref{thm:bm}.
\begin{defn}\label{def:affine_modulo_F}
	Let $\gamma\colon [0,1]\to (0,1)^{d}$ be a Lipschitz curve and $\Fonedim\subseteq [0,1]$ be a closed set. We say that $\gamma$ is \emph{affine modulo $\Fonedim$} if $\gamma$ is affine on each component of $[0,1]\setminus\Fonedim$.
\end{defn}
The next lemma allows us to assume that the Lipschitz curve given by the hypothesis of Theorem~\ref{thm:bm} is affine modulo $\Fonedim$.
\begin{lemma}\label{lem:cond-C}
If $\gamma\colon[0,1]\to (0,1)^d$ is a Lipschitz curve, $\Fonedim\subseteq[0,1]$ is a closed set, $E\subseteq \Fonedim$ is a relatively residual subset of $\Fonedim$ and $\gamma'(t)$ exists for every $t\in E$, then we may redefine $\gamma$ and $E$ as $\gamma_1$ and $E_1$ in such a way that $E_1\subseteq E$ is a relatively residual subset of $\Fonedim$, $\gamma_1\colon [0,1]\to (0,1)^{d}$ is a Lipschitz curve with $\lip(\gamma_1)\le\lip(\gamma)$, $\gamma_1(t)=\gamma(t)$ for $t\in E_1$,  $\gamma_1$ is differentiable at every $t\in E_1$ with $\gamma_1'(t)=\gamma'(t)$ and $\gamma_1$ is affine~modulo~$\Fonedim$.
\end{lemma}
\begin{proof}
Note that $(0,1)\setminus \Fonedim$ is an open set, hence it is equal to the union $\bigcup_{n=1}^{\infty}(a_n,b_n)$ of open, disjoint intervals. Let $E_{1}=E\setminus \bigcup_{n\ge1}\{a_n,b_n\}$; re-define $\gamma$ on each of $(a_n,b_n)$ in an affine way and call the new curve $\gamma_{1}$. Note that $E_{1}$ is a relatively residual subset of $\Fonedim$ and that $\gamma_{1}\colon [0,1]\to (0,1)^{d}$ is a Lipschitz curve with $\lip(\gamma_{1})\leq \lip(\gamma)$ and $\gamma_{1}(t)=\gamma(t)$ for all $t\in E_{1}$.

To check that $\gamma_1$ is differentiable on $E_1$, let us fix any $t\in E_1$ and $\eps>0$. As $\gamma$ is differentiable at $t$, let $v\in\R^d$ and $\delta>0$ be such that
$\norm{\gamma(t+h)-\gamma(t)-hv}\le\eps\abs{h}$ for all $\abs{h}<\delta$. Let $N=\{n\ge1\colon (b_n-a_n)\ge\delta/2\}$. Note that the set $N$ is finite, and $t$ has positive distance from the set $U=\bigcup_{n\in N}(a_n,b_n)$. Let $\delta_1=\min(\dist(t,U), \delta/2)$ and assume $\abs{h}<\delta_1$. If $t+h\not\in\bigcup_{n=1}^{\infty}(a_n,b_n)$, then $\gamma_1(t+h)=\gamma(t+h)$ and $\gamma_1(t)=\gamma(t)$, so that
\begin{equation}\label{eq:diff_gamma1}
\norm{\gamma_1(t+h)-\gamma_1(t)-hv}\le\eps\abs{h}.
\end{equation}
If $n\ge1$ is such that $t+h\in(a_n,b_n)$, then $n\not\in N$, i.e.\ $(b_n-a_n)<\delta/2$. Hence using $\abs{h}<\delta/2$, we get $\abs{a_n-t},\abs{b_n-t}<\delta$. We thus have, using $\gamma_1(a_n)=\gamma(a_n)$ and $\gamma_1(b_n)=\gamma(b_n)$, that
\[
\norm{\gamma_1(a_n)-\gamma_1(t)-(a_n-t)v}\le\eps\abs{a_n-t}
\text{ and }
\norm{\gamma_1(b_n)-\gamma_1(t)-(b_n-t)v}\le\eps\abs{b_n-t}.
\]
As $t\not\in[a_n,b_n]$, we either have that both $(a_n-t)$ and $(b_n-t)$ are positive, or both are negative.
Thus if $t+h=\alpha a_n+(1-\alpha)b_n$, for $\alpha\in(0,1)$, then
\begin{align*}
&\norm{\gamma_1(t+h)-\gamma_1(t)-hv}
=
\norm{\alpha\gamma_1(a_n)+(1-\alpha)\gamma_1(b_n)-\gamma_1(t)-hv}\\
&\le
\alpha \norm{\gamma_1(a_n)-\gamma_1(t)-(a_n-t)v}
+(1-\alpha)\norm{\gamma_1(b_n)-\gamma_1(t)-(b_n-t)v}\\
&\le
\eps\abs{\alpha(a_n-t)+(1-\alpha)(b_n-t)}
=\eps\abs{h},
\end{align*}
verifying~\eqref{eq:diff_gamma1}.
\end{proof}
\begin{defn}
Let $\gamma\colon [0,1]\to (0,1)^{d}$ be a Lipschitz curve, $I\subseteq[0,1]$ be an interval, $u\in \Sphere^{d-1}$ and $\myeps>0$. We say that $\gamma$ is \emph{$\myeps$-flat in direction~$u$ around $I$} if for all $t_1,t_2\in[0,1]$ with $\dist(t_i,I)<\lebmi{I}$ it holds that
\begin{equation}\label{eq:flat}
\norm{\gamma(t_1)-\gamma(t_2)-(t_1-t_2) u}\le\myeps\abs{t_1-t_2}.
\end{equation}
There are many cases when we do not need to keep information about the vector $u$. Thus we will often write simply that $\gamma$ is $\myeps$-flat around $I$ to signify that $\gamma$ is $\myeps$-flat around $I$ in some direction~$u\in \Sphere^{d-1}$. \end{defn}
\begin{remark}\label{rem:flat_equiv}
Condition~\eqref{eq:flat} is equivalent to the following: there exists $w_{t_1,t_2}\in\R^d$ with $\norm{w_{t_1,t_2}}\le1$ such that 
\begin{equation}\label{eq:flat_equiv}
\gamma(t_1)-\gamma(t_2)=(t_1-t_2)(u+\myeps w_{t_1,t_2}).
\end{equation}
\end{remark}
\begin{remark}\label{rem:open-closed}
It is not important whether the interval $I$ in the definition of $\myeps$-flatness is open or closed: for $I_1=(a,b)$ and $I_2=[a,b]$ the values of $\lebmi{I_{j}}$ 
and the sets of $t\in[0,1]$ such that $\dist(t,I_{j})<\lebmi{I_{j}}$ are the same.

Trivially, the flatness property passes to subintervals.
\end{remark}
\paragraph{Notation.} Given $t\in \R$ and $\delta>0$ we let 
\begin{equation*}
I_{\delta}(t):=(t-\delta,t+\delta).
\end{equation*}
\begin{defn}
Let $t\in\R$, $\Fonedim\subseteq \R$ and $\eps>0$. We say that $I_{\delta}(t)$ is an \emph{$\eps$-density interval for $\Fonedim$} if 
\begin{equation*}
\leb(I_{r}(t)\setminus \Fonedim)<2r\eps\qquad \text{for every $r\in (0,\delta]$.}
\end{equation*}
\end{defn}
\begin{remark}\label{rem:density-interval}
Suppose $Y\subseteq\R$ is open, $X\subseteq Y$ has positive measure and let $\eps>0$. Then for almost all $t\in X$ there is an $\eps$-density interval $I_\delta(t)$ for $X$ such that $I_\delta(t)\subseteq Y$. 
This follows from the Lebesgue Density theorem, see~\cite[Corollary~2.14~(1)]{mattila_1995}.
\end{remark}
\begin{lemma}\label{cond-C-ued}
\label{lem:flat}
Let $\gamma\colon [0,1]\to (0,1)^{d}$ and $E\subseteq \Fonedim\subseteq [0,1]$ satisfy the hypotheses of Theorem~\ref{thm:bm} and suppose that $\gamma$ is affine modulo $\Fonedim$. Then for every interval $(a,b)\subseteq [0,1]$ for which $(a,b)\cap \Fonedim\neq \emptyset$ and $\myeps\in(0,1)$ there exist $u\in \Sphere^{d-1}$ and an open interval $I\subseteq (a,b)$ such that $\gamma$ is $\myeps$-flat in direction~$u$ around $I$ and $I\cap \Fonedim\neq \emptyset$. 
\end{lemma}
\begin{proof}
	Let $(a,b)\subseteq [0,1]$ with $(a,b)\cap \Fonedim\neq \emptyset$. Choose a set $\{u_k\}$ of unit vectors, dense in the unit sphere $\Sphere^{d-1}$, and consider a family of sets
	\begin{equation}\label{eq:def_ekm}
	E_{k,m}=\{r\in [a,b]\colon
	\norm{\tfrac{\gamma(s)-\gamma(r)}{s-r}-u_k}\le\myeps \text{ for all }s\in[a,b] \text{ with }0<\abs{s-r}<1/m
	\}.
	\end{equation}
	Note that each $E_{k,m}$ is closed and $\bigcup_{k,m}E_{k,m}\supseteq E\cap[a,b]$. 
	
	Since $E\cap [a,b]$ is relatively residual in $\Fonedim\cap[a,b]$, there is a pair $(k,m)$ and a non-degenerate open interval $J\subseteq (a,b)$ such that $E_{k,m}\supseteq J\cap \Fonedim\ne\emptyset$. Let $u=u_k$, $t\in J\cap \Fonedim$ and choose $\Delta>0$ small enough so that $I_{\Delta}(t)\subset J$. Let $0<\delta<\min\bigl(1/(6m),\Delta/3\bigr)$.
	We show $I:=I_{\delta}(t)$ fulfils the assertions of the lemma. Since $t\in I\cap \Fonedim$, we have $I\cap \Fonedim\ne\emptyset$.
	We now verify the flatness of $\gamma$ around $I$ in direction~$u$. Let $t_1,t_2\in [0,1]$ be such that $\dist (t_i,I)<\lebmi{I}=2\delta$. Then $\abs{t_1-t_2}<6\delta<1/m$ and $t_{1},t_{2}\in I_{\Delta}(t)$. If $t_1\in \Fonedim$, then $t_1\in I_\Delta(t)\cap \Fonedim\subseteq J\cap \Fonedim\subseteq E_{k,m}$. Hence~\eqref{eq:flat} is satisfied. Assume now $t_1\not\in \Fonedim$ and consider the decomposition of $I_{3\delta}(t)\setminus \Fonedim$ into the union of countably many disjoint open intervals $V_n=(a_n,b_n)$. We therefore have that $t_1\in V_n$ for some $n\ge1$. If $t_2\in V_n$ too, then the affineness of $\gamma$ on $V_n$ and the fact that the endpoints of $V_{n}$ belong to $I_{3\delta}(t)\cap \Fonedim\subseteq J\cap \Fonedim\subseteq E_{k,m}$, imply that~\eqref{eq:flat} is satisfied. If $t_2\not\in V_n$, then as $t_2\in I_{3\delta}(t)$ and $V_n\subseteq  I_{3\delta}(t)$, we conclude that both $\abs{a_n-t_{i}}$ and $\abs{b_n-t_{i}}$ for $i=1,2$ are less than $6\delta\le1/m$. Hence, using $a_n,b_n\in I_{3\delta}(t)\cap \Fonedim\subseteq J\cap \Fonedim\subseteq E_{k,m}$, we may write inequality~\eqref{eq:def_ekm} with $t_2$ and endpoints of $V_n$, to get~\eqref{eq:flat} for $t_1$, $t_2$.
\end{proof}
\paragraph{Notation.}
Suppose $\gamma\colon[0,1]\to(0,1)^{d}$ is a $1$-Lipschitz curve, $I\subseteq [0,1]$ is an interval, $P\subseteq I$ is finite and $f\colon[0,1]^d\to\R$ is Lipschitz. Let $\sigma,\tau>0$ and consider the set $Y_{\sigma,I,P}=\{y\in[0,1]^{d}\colon\dist(y,\gamma(I))\ge\sigma\}\cup\gamma(P)$. Denote
\begin{equation}\label{eq:conical}
\Phi_{\gamma,f,I,P,\sigma,\tau}(x)=\inf_{y\in Y_{\sigma,I,P}}(f(y)+\tau\norm{x-y}),\qquad x\in[0,1]^{d},
\end{equation}
and call $\Phi_{\gamma,f,I,P,\sigma,\tau}\colon [0,1]^{d}\to\R$ a \emph{conical function}. If $\alpha\in(0,1)$ is a parameter and $\tau>1-\alpha$, we call $\Phi_{\gamma,f,I,P,\sigma,\tau}$ an $\alpha$-conical function.
 
\begin{lemma}\label{lem:conical:easy}
Let $f\colon [0,1]^{d}\to\R$ be a Lipschitz function, $\emptyset\ne Y\subseteq [0,1]^d$ and 
$\tau\ge\Lip(f)$.
Then the conical function $\Phi(x)=\inf_{y\in Y}(f(y)+\tau\norm{x-y})$ is $\tau$-Lipschitz and $\Phi(x)=f(x)$ for $x\in Y$.
\end{lemma}
\begin{proof}
For any $y\in Y$ and any $x\in[0,1]^d$ we have $f(y)-f(x)\ge-\Lip(f)\norm{x-y}$ implying
$f(y)+\tau\norm{x-y}\ge f(x)+(\tau-\Lip(f))\norm{x-y}\ge f(x)$ which means, for all $x\in[0,1]^d$,
\[\Phi(x)\ge f(x).
\]  
In particular, the values of $\Phi$ are finite.
As for each $y\in Y$, the function 
\begin{equation}\label{eq:phi}
\phi_y(x)=f(y)+\tau\norm{x-y}
\end{equation}
is $\tau$-Lipschitz, we conclude that $\Phi$ is $\tau$-Lipschitz too.
Note that additionally, for $x\in Y$ it trivially holds $\Phi(x)\le\phi_x(x)=f(x)$. Thus $\Phi=f$ on $Y$.
\end{proof}
\begin{lemma}\label{lem:conical:new}
Let $\gamma\colon [0,1]\to(0,1)^{d}$ be a $1$-Lipschitz curve which is $\myeps$-flat around an interval $I\subseteq [0,1]$ in direction~$u\in \Sphere^{d-1}$, where $\myeps\in(0,1/3)$.
	Let $\varepsilon>0$ and $f\colon [0,1]^{d}\to\R$ be a Lipschitz function with $\lip(f)<1$. Then 
for every 
$\alpha\in(0,1)$ there is an $\alpha$-conical function, which we denote by $f_{\varepsilon,I}$, and a closed, null set $N=N_{f,\varepsilon,I}\subseteq I$ with the following properties:
	\begin{enumerate}[(i)]
		\item\label{it:approx_lip} 
		$\lip(f_{\varepsilon,I})<1$ and
		 $\inorm{f_{\varepsilon,I}-f}<\varepsilon$;
		\item\label{it:component} 
There is $\tau\in(1-\alpha,1)$ such that
		for every component $J$ of $I\setminus N$ there is $p=p_{J}\in N$ such that
\[
		f_{\varepsilon,I}(x)=f(\gamma(p))+\tau\norm{x-\gamma(p)}
		\text{ for all }
		x\in\gamma(J)
\]
		and
		the function $f_{\varepsilon,I}$ is continuously differentiable on an open neighbourhood $U_f(\gamma(J))$ of $\gamma(J)$ with 
		\begin{equation}\label{eq:grad_on_UJ}
		\nabla f_{\eps,I}(x)=\tau\frac{x-\gamma(p)}{\norm{x-\gamma(p)}}\qquad \text{ for all }x\in U_f(\gamma(J)).
		\end{equation}
	\end{enumerate}
\begin{remark*}
	Note that the conical function $f_{\eps,I}$ and associated set $N_{f,\eps,I}$ given by the conclusion of Lemma~\ref{lem:conical:new} depend on the value of $\tau$ and the curve $\gamma$. Since we will only ever consider conical functions with respect to a single fixed curve $\gamma$, we suppress this dependency on $\gamma$ in the notation. The value of $\tau$ will eventually be important for us but we suppress it for now to keep the notation tidier.
\end{remark*}
\end{lemma}
\begin{proof}[Proof of Lemma~\ref{lem:conical:new}] 
	Set $\eta=\frac{\eps(1-\Lip(f))}{64}$, $\sigma=\frac{\eps}{8}$, fix any finite $\eta$-net $P$ of $I$, and let 
\begin{equation}\label{eq:tau}
\tau\in\Bigl(\max\set{\frac{\lip(f)+1}{2},1-\alpha},1\Bigr)
\end{equation}
	be arbitrary. We define $f_{\varepsilon,I}$ as the conical function $\Phi_{\gamma,f,I,P,\sigma,\tau}$ of~\eqref{eq:conical}. We will show that part~\eqref{it:approx_lip} holds without further restriction on $\tau$, and that part~\eqref{it:component} holds with a suitable additional condition on $\tau$.
		
	By Lemma~\ref{lem:conical:easy}, the function~$f_{\eps,I}$ has Lipschitz constant less than or equal to $\tau<1$. If $x\in \R^d$ is such that $\dist(x,\gamma(I))<\sigma$, find $y\in\gamma(P)$ with $\norm{x-y}<\sigma+\Lip(\gamma)\eta\le\sigma+\eta$, then by Lemma~\ref{lem:conical:easy} it follows that $f_{\varepsilon,I}(y)=f(y)$, so that
	\[
	\abs{f_{\varepsilon,I}(x)-f(x)}
	\le
	\abs{f_{\varepsilon,I}(x)-f_{\varepsilon,I}(y)}
	+\abs{f(y)-f(x)}
	\le(\tau+1)(\sigma+\eta)
	<2\frac\eps4
	=\frac\eps2.
	\]
	Hence, using again Lemma~\ref{lem:conical:easy}, we get  $\inorm{f_{\varepsilon,I}-f}<\eps$, completing~\eqref{it:approx_lip}, for all $\tau$ satisfying~\eqref{eq:tau}.

	We now determine an additional mild restriction on $\tau$ satisfying~\eqref{eq:tau}, under which part~\eqref{it:component} is valid.
	Note first that~\eqref{eq:tau} implies 
	$\eta<\frac{\eps(\tau-\Lip(f))}{32}$, from which it follows that $\sigma=\frac{\eps}{8}>\frac{4\tau\eta}{\tau-\Lip(f)}$.
	Consider any $x,y\in \R^d$ such that $\dist(x,\gamma(I))<\eta$ and $\dist(y,\gamma(I))\ge\sigma$. Find $z\in\gamma(P)$ such that 
	$\norm{x-z}\le\eta+\Lip(\gamma)\eta\le2\eta$. Then, using additionally $\Lip(f)-\tau<0$ and $\norm{y-z}\ge\sigma$, we get
	\begin{align*}
	f(z)-f(y)&+\tau\norm{x-z}-\tau\norm{x-y}
	\le
	\Lip(f)\norm{y-z}+2\tau\eta-\tau\Bigl(\norm{y-z}-2\eta\Bigr)\\
	&=(\Lip(f)-\tau)\norm{y-z}+4\tau\eta
	\le(\Lip(f)-\tau)\sigma+4\tau\eta
	<0,
	\end{align*}
	so that 
	$f(y)+\tau\norm{x-y}\ge f(z)+\tau\norm{x-z}$. Using the definition~\eqref{eq:conical} of the conical function $f_{\varepsilon,I}=\Phi_{\gamma,f,I,P,\sigma,\tau}$ we conclude that
	\begin{equation}\label{eq:near_gamma}
	f_{\varepsilon,I}(x)=\min_{y\in \gamma(P)}(f(y)+\tau\norm{x-y}), \qquad \text{for all }x\in B(\image(\gamma),\eta).
	\end{equation}
	Let $\Gamma=\{(y,z)\colon y,z\in\gamma(P)\text{ and }y\ne z\}$ (a finite set).
	Fix a pair $(y,z)\in\Gamma$, then $y=\gamma(p)\ne z=\gamma(q)$, implying $p\ne q$, and let 
	\[M_{y,z,\tau}
	=\{t\in I\colon f(y)+\tau\norm{\gamma(t)-y}=f(z)+\tau\norm{\gamma(t)-z}
	\}.\]
	Each $M_{y,z,\tau}$ is a closed subset of $I$.
	Note that the set $S_{y,z}$ of solutions $t\in I$ of $\norm{\gamma(t)-y}=\norm{\gamma(t)-z}$ cannot contain more than one point. Indeed, if $t_1,t_2\in S_{y,z}$ are distinct, then, as both $\gamma(t_i)$ are equidistant from $y$ and $z$, we get that $\gamma(t_1)-\gamma(t_2)$ is orthogonal to $y-z=\gamma(p)-\gamma(q)$. Hence,
	applying~\eqref{eq:flat_equiv} with $\norm{u}=1$ and $\myeps$ to $\gamma(t_1)-\gamma(t_2)$ and $\gamma(p)-\gamma(q)$ we get
	\[
	{(t_1-t_2)(p-q)}\langle u+\myeps w_{t_1,t_2},u+\myeps w_{p,q}\rangle=0,
	\]
	which is impossible as $t_1\ne t_2$, $p\ne q$ and $3\myeps\in(0,1)$.
	Finally, use that for $\tau_1\ne\tau_2$ the sets $M_{y,z,\tau_1}\setminus S_{y,z}$ and $M_{y,z,\tau_2}\setminus S_{y,z}$ are disjoint to conclude, as  $M_{y,z,\tau}\subseteq I$ for all $\tau$, that there is an at most countable set $T_{y,z}$
	of such $\tau$, satisfying~\eqref{eq:tau}, for which the Lebesgue measure of $M_{y,z,\tau}$ is positive.
	Let $T=\bigcup_{(y,z)\in\Gamma}T_{y,z}$. This is a countable set. 
	In addition to~\eqref{eq:tau}, we now prescribe that $\tau$ lies outside of the countable set $T$.
	 Let $N=N_{f,\varepsilon,I}:=P\cup\bigcup_{(y,z)\in\Gamma} M_{y,z,\tau}$. Then $N$ is a null, closed subset of $I$. Recall that the function $f_{\eps,I}$ is given on $\image(\gamma)$ by~\eqref{eq:near_gamma}.
By the Intermediate Value Theorem, for any two points $x_i=\gamma(t_{i})\in\gamma(I)$, $t_i\in I$, $i=1,2$ and $t_1<t_2$, for which the minimum in the formula~\eqref{eq:near_gamma} for $f_{\eps,I}(x_i)$ is attained at different $y=y_{i}\in\gamma(P)$, $i=1,2$, there has to be a point $t_{3}\in [t_{1},t_{2}]$ with $t_{3}\in M_{y_{1},y_{2},\tau}\subseteq N$. Therefore the first assertion of~\eqref{it:component} is valid. 

For the second assertion of~\eqref{it:component}, it remains to note that the set 
\begin{equation*}
C=\bigcup_{(y,z)\in\Gamma}\{x\in\R^d\colon f(y)+\tau\norm{x-y}=f(z)+\tau\norm{x-z}\}
\end{equation*}
is closed, and for each open component $J$ of $I\setminus N$ there exists an open component $U_f(\gamma(J))$ of $B\bigl(\image(\gamma),\eta\bigr)\setminus C$ which contains $\gamma(J)$. Thus, $f_{\varepsilon,I}\big|_{U_f(\gamma(J))}=\phi_{\gamma(p_J)}\big|_{U_f(\gamma(J))}$, and~\eqref{eq:grad_on_UJ} holds. 
\end{proof}	

\begin{lemma}\label{lem:conical-increase}
Let $\gamma\colon [0,1]\to(0,1)^{d}$ be a $1$-Lipschitz curve which is $\myeps$-flat around an interval $\myI_0\subseteq [0,1]$ in direction~$u\in \Sphere^{d-1}$, where $\myeps\in(0,1/3)$.
Let $(a,b)\subseteq \myI_0$, 
$q\in \myI_0\setminus(a,b)$, $r\in\R$, $\tau>0$,  and let 
$h\colon [0,1]^{d}\to\R$ 
be a Lipschitz function with
\begin{equation*}
h(x)
=
r+\tau\norm{x-\gamma(q)}\qquad \text{for 
$x\in\{\gamma(a),\gamma(b)\}$.}
\end{equation*}
Then
\begin{align*}
\abs{(h\circ\gamma)\big|_a^b-\tau(b-a)}&\le3\myeps\tau(b-a)\qquad \text{if $q\le a\leq b$, and}\\
\abs{(h\circ\gamma)\big|_a^b-\tau(a-b)}&\le3\myeps\tau(b-a)\qquad \text{if $a\leq b\le q$,}	
\end{align*}
where $(h\circ\gamma)\big|_a^b=(h\circ\gamma)(b)-(h\circ\gamma)(a)$.
\end{lemma}
\begin{proof}
In what follows we adopt the notation of Remark~\ref{rem:flat_equiv} and in particular make use of the identity~\eqref{eq:flat_equiv} for points $t_{1},t_{2}\in \myI_0$.
Observe that
\begin{align*}
\gamma(a)-\gamma(q)&=(a-q)u+\myeps(a-q)w_{a,q}
\quad\text{and}\\
\gamma(b)-\gamma(q)&=\gamma(a)-\gamma(q)+(b-a)u+\myeps(b-a)w_{b,a}\\
&=(b-q)u+\myeps(b-q)w_{a,q}+\myeps(b-a)(w_{b,a}-w_{a,q}).
\end{align*}
Hence 
\begin{align*}
\norm{\gamma(a)-\gamma(q)}&=\abs{a-q}\norm{u+\myeps w_{a,q}}
\quad\text{and}\\
\Bigl|\norm{\gamma(b)-\gamma(q)}&-
\abs{b-q}\norm{u+\myeps w_{a,q}}\Bigr|\le2\myeps(b-a).
\end{align*}
Hence, if $q\le a\leq b$, then 
\begin{align*}
&\bigl|(h\circ\gamma)\big|_a^b-\tau(b-a)\bigr|\\
&=\tau
\Bigl|\bigl(\norm{\gamma(b)-\gamma(q)}-\norm{\gamma(a)-\gamma(q)}\bigr)-(b-a)\Bigr|\\
&=\tau
\Bigl|\norm{\gamma(b)-\gamma(q)}-(a-q)\norm{u+\myeps w_{a,q}}-(b-a)\Bigr|\\
&=\tau
\Bigl|\norm{\gamma(b)-\gamma(q)}-(b-q)\norm{u+\myeps w_{a,q}}+(b-a)\norm{u+\myeps w_{a,q}}-(b-a)\Bigr|\\
&\le
2\myeps\tau(b-a)+\tau(b-a)\Bigl|\norm{u+\myeps w_{a,q}}-1\Bigr|
\le
3\myeps\tau(b-a)
\end{align*}
If $a\leq b\le q$, then 
\begin{align*}
\bigl|(h\circ\gamma)\big|_a^b&-\tau(a-b)\bigr|\\
=\tau
\Bigl|\bigl(\norm{\gamma(b)-\gamma(q)}&-\norm{\gamma(a)-\gamma(q)}\bigr)-(a-b)\Bigr|\\
=\tau
\Bigl|\norm{\gamma(b)-\gamma(q)}&-(q-a)\norm{u+\myeps w_{a,q}}-(a-b)\Bigr|\\
=\tau
\Bigl|\norm{\gamma(b)-\gamma(q)}&-(q-b)\norm{u+\myeps w_{a,q}}+(a-b)\norm{u+\myeps w_{a,q}}-(a-b)\Bigr|\\
\le
2\myeps\tau(b-a)&+\tau(b-a)\Bigl|\norm{u+\myeps w_{a,q}}-1\Bigr|
\le
3\myeps\tau(b-a).
\end{align*}
\end{proof}
\begin{lemma}\label{lem:Lip}
If $f\colon[a,b]\to\R$ is a Lipschitz function, $N\subseteq[a,b]$ is a closed null set and 
$(a,b)\setminus N=\bigcup_{n=1}^{\infty}(a_n,b_n)$ is a union of disjoint, open components, then
$f(b)-f(a)=\sum_{n\ge1}\bigl(f(b_n)-f(a_n)\bigr)$. 
\end{lemma}
\begin{proof}
Observe that 
\[
f(b)-f(a)
=\int_a^b f'(t)dt
=\sum_{n\ge1}\int_{a_n}^{b_n} f'(t)dt
=\sum_{n\ge1}\bigl(f(b_n)-f(a_n)\bigr).
\]
\end{proof}
\begin{lemma}\label{lem:slope} 
Let $\gamma\colon [0,1]\to(0,1)^{d}$ and $E\subseteq \Fonedim\subseteq [0,1]$ satisfy the hypotheses of Theorem~\ref{thm:bm} and suppose that $\gamma$ is affine modulo $\Fonedim$. Suppose $f\in\Lip_1([0,1]^{d})$ is such that $\Lip(f)<1$. 
Assume an open set $U\subseteq [0,1]$ such that $U\cap \Fonedim$ is dense in $\Fonedim$ is given, $0<\myepsone\leq \myepszero<1/\myconst^3$ and $\epszero\in(0,1)$. 
Suppose 
$\myI\subseteq\myI_0\subseteq[0,1]$ are open intervals such that 
$\gamma$ is $\myepszero$-flat around $\myI_0$ and $\myI\cap \Fonedim\neq \emptyset$.
Suppose further that 
$f_{\epszero,\myI_0}$ is a $\theta$-conical function  given by Lemma~\ref{lem:conical:new}. 
\\
Then 
there is 
an open interval $\myI_1\subseteq \myI\cap U$, such that
$\gamma$ is $\myepsone$-flat around $\myI_1$ and the following statement holds:
\\
\textbf{Approximation property~\ref{lem:slope}a:}
Let $g\in \Lip_1([0,1]^d)$ with 
\begin{equation}\label{eq:g_hyp}
\Lip(g)<1\qquad\text{ and }\qquad \inorm{g-f_{\epszero,\myI_0}}<\myepszero\lebmi{\myI_1}/4,
\end{equation}
$\epsone\in(0,\myepszero\lebmi{\myI_{1}}/4)$ and $g_{\epsone,\myI_1}$
be a $\theta'$-conical function given by Lemma~\ref{lem:conical:new}.
Then there exists an open interval $V$ such that
\begin{enumerate}[(i)]
		\item\label{V_trivial_properties} $\overline V\subseteq \myI_1$, $\lebmi{V}\le\lebmi{\myI_1}/2$ and  
		$V\cap \Fonedim\ne\emptyset$;
		\item\label{g_cts_diff_on_V} $g_{\epsone,\myI_{1}}$ is continuously differentiable on an open neighbourhood of $\gamma(V)$; for points $x$ from this neighbourhood its gradient $\nabla g_{\epsone,\myI_{1}}(x)$ is given by the formula~\eqref{eq:grad_on_UJ} with $\tau\in(1-\myepsone,1)$ and $p\in R_{1}$;
		\item\label{f_g_slope_difference}
		for every $t\in V$ and $s\in [0,1]$ it holds \begin{equation}\label{eq:est-slope}
		\abs{
			(f_{\epszero,\myI_0}(\gamma(s))-f_{\epszero,\myI_0}(\gamma(t)))-
			(g_{\epsone,\myI_1}(\gamma(s))-g_{\epsone,\myI_1}(\gamma(t)))
		}
		\le2\myepszero^{1/3}\abs{s-t}.
		\end{equation}		 
	\end{enumerate}
\end{lemma}
\begin{proof}
Consider the closed, null set $\Nzero=N_{f,\epszero,\myI_0}\subseteq \myI_0$ defined by Lemma~\ref{lem:conical:new} for the function $f_{\epszero,\myI_0}$. Since $\lebm{\Nzero}=0$, $\myI\cap \Fonedim\neq\emptyset$ and $\Fonedim$ has every portion of positive measure, we have $\myI\cap \Fonedim\not\subseteq \Nzero$. Hence, we may choose one open component $J_0$ of $\myI_0\setminus \Nzero$ for which $J_0\cap\myI\cap \Fonedim\ne\emptyset$. As $U\cap \Fonedim$ is dense in $\Fonedim$ and $J_0\cap\myI$ is open, we conclude $U\cap J_0\cap\myI\cap \Fonedim\ne\emptyset$. Find then an open interval $J'\subseteq J_0\cap R\cap U$ such that $J'\cap \Fonedim\ne\emptyset$. 
Apply Lemma~\ref{lem:flat} to get an open interval $J''\subseteq J'$ such that 
$\gamma$ is $\myepsone$-flat around $J''$ and $J''\cap \Fonedim\neq \emptyset$.
By Remark~\ref{rem:density-interval}, find a $\myepszero$-density interval $I_\Delta(t_0)$ for $\Fonedim$, such that $I_\Delta(t_0)\subseteq J''$. Let $\myI_1=I_\Delta(t_0)$. Then, using Remark~\ref{rem:open-closed} for the latter statement, we get that
\[
\myI_1
\subseteq J''
\subseteq J'
\subseteq J_0\cap\myI\cap U
\subseteq \myI\cap U
\quad\text{ and }\quad
\gamma
\text{ is $\myepsone$-flat around }
\myI_1.
\]
Note that all assertions of the lemma for the interval $\myI_{1}$, apart from those contained in the Approximation property~\ref{lem:slope}a,
are already verified. We turn our attention to proving~\ref{lem:slope}a~\eqref{V_trivial_properties}--\eqref{f_g_slope_difference}.

Let $g\in \Lip_1([0,1]^d)$ be given according to~\eqref{eq:g_hyp} and let 
\begin{equation}\label{eq:def:eps1}
\epsone\in(0,\myepszero\lebmi{\myI_1}/4).
\end{equation}
Let 
$g_{\epsone,\myI_1}$ be a $\myepsone$-conical function given by the hypothesis of~\ref{lem:slope}a 
and 
$\None=N_{g,\epsone,\myI_1}$ be the corresponding closed null set, as given by Lemma~\ref{lem:conical:new}.
For brevity, denote $\hat f=f_{\epszero,\myI_0}$ and $\hat g=g_{\epsone,\myI_1}$.

As $\myI_1\subseteq J_0\subseteq \myI_0\setminus\Nzero$,
there is, by Lemma~\ref{lem:conical:new}~\eqref{it:component}, a point $p\in \Nzero\subseteq \myI_0\setminus\myI_1$ and a constant~$\tauzero\in(1-\myepszero,1)$ satisfying the formula $\hat f(x)=f(\gamma(p))+\tauzero\norm{x-\gamma(p)}$ for each $x\in\gamma(\myI_1)\subseteq\gamma(J_0)$. 
Note that $p\notin\myI_1$ allows us,
without loss of generality, to assume that $p$ is to the left of the interval $\myI_1$. 
Let $U_f(\gamma(J_0))$ be the open neighbourhood of $\gamma(J_0)$ guaranteed by Lemma~\ref{lem:conical:new}~\eqref{it:component}, such that $\hat f$ is continuously differentiable on $U_f(\gamma(J_0))$.

Consider all open components $C$ of $\myI_1\setminus \None$ and enumerate them as $C_n=(a_n,b_n)$. We will assume the more complicated case when there are infinitely many such components, so that every natural number $n$ is assigned bijectively to a component $C_{n}$.
Then Lemma~\ref{lem:conical:new}~\eqref{it:component} similarly provides $p_n\in\None\subseteq\myI_1$ and $\tauone\in(1-\myepsone,1)$
with respect to which $\hat g(x)=g(\gamma(p_n))+\tauone\norm{x-\gamma(p_n)}$ for all $x\in\gamma( C_n)$. Let 
$U_g(\gamma(C_{n}))$ 
be the open neighbourhood of $\gamma(C_n)$ such that $\hat g$ is continuously differentiable on $U_g(\gamma(C_{n}))$.

Hence, for each $n\ge1$, 
\begin{equation}\label{eq:Wn}
W_n:=U_f(\gamma(J_0))\cap U_g(\gamma(C_n)) 
\end{equation}
is an open neighbourhood of $\gamma(C_n)$ such that both $\hat f|_{W_n}$ and $\hat g|_{W_n}$ are continuously differentiable, 
so that for every $t\in \myI_1\setminus \None$ the gradients $\nabla \hat f(\gamma(t))$ and $\nabla \hat g(\gamma(t))$ are well-defined. Moreover, for every $n\geq 1$ the functions $\hat f|_{\gamma(C_n)}$ 
and $\hat g|_{\gamma(C_n)}$ satisfy the conditions of Lemma~\ref{lem:conical-increase} for $h$. The only difference will be that for all $n\ge1$, the functions $\hat f\big|_{\gamma(C_n)}$ will use the same $q=p\in \Nzero$ whilst
the functions $\hat g\big|_{\gamma(C_n)}$ may use different $q=p_n\in \None$. Moreover, by our assumption we have that $p<a_n<b_n$ for any $n\ge1$, but we may have $p_n<a_n<b_n$ for some $n\ge1$, and $a_n<b_n<p_n$ for others. 
Let 
\begin{equation}\label{eq:def_GB}
\begin{aligned}	
\good&=\{n\ge1\colon p_n<a_n<b_n\}, 
&\ggood=\bigcup_{n\in \good}(a_n,b_n)
&\text{\qquad (good sets)},\\
\bad&=\{n\ge1\colon a_n<b_n<p_n\},
&\bbad=\bigcup_{n\in \bad}(a_n,b_n)
&\text{\qquad (bad sets)},
\end{aligned}
\end{equation}
and note for future reference that $\good\cup \bad=\mathbb N$ and $\ggood\cup\bbad=\bigcup_{n\ge1}C_n=\myI_1\setminus \None$.

Write $\myI_1=(a,b)$ and denote by $u\in \Sphere^{d-1}$ the vector such that $\gamma$ is $\myepszero$-flat around $\myI_0$ in direction~$u$.
Consider the following sets: 
\begin{align}
\Fset_{0}&=\None\cup\set{a,b},\notag\\
\Fset_1&=\{t\in\overline{\myI}_1\setminus \Fset_0\colon
\abs{\sk{\nabla\hat f(\gamma(t))}{u}-\sk{\nabla\hat g(\gamma(t))}{u}
}\ge\myepszero^{1/3}\},\notag\\
\Fset_2&=\{t\in\overline{\myI}_1\setminus(\Fset_0\cup \Fset_1)\colon
\exists s\in \overline{\myI}_1\setminus\{t\} \text{ such that }
\notag\\
\label{eq:est-slope-neg}
&\hphantom{AAAAA}\abs{
(\hat f(\gamma(s))-\hat f(\gamma(t)))-
(\hat g(\gamma(s))-\hat g(\gamma(t)))
}
\ge
2\myepszero^{1/3}\abs{s-t}
\}.
\end{align}
We now show that the union $\Fset_{0}\cup \Fset_{1}\cup \Fset_{2}$ is closed. As an intermediate step, we first prove that $\Fset_{0}\cup \Fset_{1}$ is closed. To see this, recall that for each 
$n\ge1$
we have that both $\hat{f}$ and $\hat{g}$ are continuously differentiable on $\gamma(C_n)\subseteq U_f(\gamma(J_0))\cap U_g(\gamma(C_n))$. 
Therefore $\Fset_{1}$ intersects each $C_n$ in a relatively closed set, that is, there is a closed set $K_{n}\subseteq \overline{\myI}_1$ such that $\Fset_{1}\cap C_n=K_{n}\cap C_n$. Hence, $\Fset_{1}=\bigcup_{n=1}^{\infty}\left(K_{n}\cap C_n\right)$. Let $(t_{i})_{i=1}^{\infty}$ be a sequence in $\Fset_{1}\cup \Fset_{0}=\Fset_{0}\cup\bigcup_{n=1}^{\infty}\left(K_{n}\cap C_n\right)$ such that $t_{i}\to t\in \overline{\myI}_1$. We need to to verify that $t\in \Fset_{1}\cup \Fset_{0}$. We distinguish two cases: If there exists $n_{0}\in\N$ such that $t\in C_{n_{0}}$ then there is $m_{0}\in \N$ such that $t_{i}\in (\Fset_{1}\cup \Fset_{0})\cap C_{n_{0}}=K_{n_{0}}\cap C_{n_{0}}$ for all $i\geq m_{0}$. Since $K_{n_{0}}$ is closed, we conclude that $t=\lim t_{i}\in K_{n_{0}}$. Hence $t\in K_{n_{0}}\cap C_{n_{0}}\subseteq \Fset_{1}$. In the remaining case we have that $t\in \overline{\myI}_1\setminus \bigcup_{n=1}^{\infty}C_n\subseteq \Fset_{0}$.

Now we proceed to show that $\Fset_{0}\cup \Fset_{1}\cup \Fset_{2}$ is closed. Given that $\Fset_{0}\cup \Fset_{1}$ is closed it suffices to check that the limit of any convergent sequence in $\Fset_{2}$ belongs to $\Fset_{0}\cup \Fset_{1}\cup \Fset_{2}$. Let $(t_{i})_{i=1}^{\infty}$ be a convergent sequence in $\Fset_{2}$ with limit $t\in \overline{\myI}_1$. For each $i\in\N$ we may choose $s_{i}\in \overline{\myI}_1$ witnessing that $t_{i}\in \Fset_{2}$ and, by passing to a subsequence if necessary, we may assume that the sequence $(s_{i})_{i=1}^{\infty}$ converges to a point $s\in \overline{\myI}_1$. 
We distinguish two cases: If $s\neq t$, then taking limits as $i\to\infty$ in~\eqref{eq:est-slope-neg} for $s_{i}$ and $t_{i}$ implies $t\in \Fset_{2}$.
Assume now $s=t\not\in \Fset_0$. Then there exists $n_0\ge1$ such that $s=t\in C_{n_{0}}$ and
$t_i,s_i\in C_{n_{0}}\subseteq \myI_1$  for all $i$ sufficiently large, say $i\ge m_0$. 
Recall that
$\gamma$ is $\myepszero$-flat around $\myI_0\supseteq \myI_1$ in direction~$u$. Thus,
\begin{align*}
&\abs{
\frac{\hat f(\gamma(s_i))-\hat f(\gamma(t_i))}{s_i-t_i}-
\frac{\hat f(\gamma(t_i)+u(s_i-t_i))-\hat f(\gamma(t_i))}{s_i-t_i}
}\\
\le
&\lip(\hat f)
\abs{
\frac{\gamma(s_i)-\left(\gamma(t_i)+u(s_i-t_i)\right)}{s_i-t_i}
}
\le
\myepszero 
\end{align*}
and, similarly,
\begin{align*}
&\abs{
\frac{\hat g(\gamma(s_i))-\hat g(\gamma(t_i))}{s_i-t_i}-
\frac{\hat g(\gamma(t_i)+u(s_i-t_i))-\hat g(\gamma(t_i))}{s_i-t_i}
}
\le
\myepszero.
\end{align*}
Hence from~\eqref{eq:est-slope-neg} we get, for all $i\ge m_0$,
\begin{align}\label{eq:slope_gap_u'}
\abs{
\frac{(\hat f(\gamma(t_i)+u(s_i-t_i))-\hat f(\gamma(t_i)))}{s_i-t_i}-
\frac{(\hat g(\gamma(t_i)+u(s_i-t_i)))-\hat g(\gamma(t_i)))}{s_i-t_i}
}
\ge2\myepszero^{1/3}-2\myepszero.
\end{align}
For each $i\in\N$ we let
\begin{equation*}
\nu_{i}(\hat f):=\hat f(\gamma(t_i)+u(s_i-t_i))-\hat f(\gamma(t_i))-(s_{i}-t_{i})\sk{\nabla\hat f(\gamma(t_{i}))}{u}
\end{equation*}
and define $\nu_{i}(\hat g)$ similarly. Note that $\displaystyle\lim_{i\to\infty}\frac{{\nu_{i}(h)}}{{s_{i}-t_{i}}}=0$ for $h=\hat f,\hat g$. 
To see this, denote
\[
D_h(r,\myr)=\begin{cases}
\frac{h(\gamma(r)+\myr u)-h(\gamma(r))}{\myr}
-\sk{\nabla h(\gamma(r))}{u},&\text{if }\myr\ne0;\\
0,&\text{if }\myr=0,
\end{cases}
\]
for $r\in C_{n_0}$ and $\myr\in\R$, where, for the purposes of this formula, we extend the functions $h=\hat f, \hat g$ arbitrarily outside of $[0,1]^{d}$. 
We now show that the two functions $D_{\hat f},D_{\hat g}\colon C_{n_0}\times\R\to\R$ are continuous at the points $(r,0)$. 
Let $r_0\in C_{n_0}$; choose positive
$\delta_{0}$ and $\myr_{0}$ small enough so that $I_{2\delta_{0}}(r_0)\subseteq C_{n_{0}}$ and $B(\gamma(I_{\delta_{0}}(r_0)),\myr_{0})\subseteq W_{n_0}$, where  $W_{n_0}\supseteq\gamma(C_{n_0})$ is the open set defined by~\eqref{eq:Wn} on which both $\hat f$ and $\hat g$ are continuously differentiable.
Then, given $r\in I_{\delta_{0}}(r_0)$ and $\abs{\myr}<\myr_0$, we have that the segment $[\gamma(r),\gamma(r)+\myr u]$ is contained in $W_{n_0}$. Therefore, $\nabla h$ is well-defined (and continuous) along this segment and we may apply the Mean Value Theorem to write
\begin{equation*}
D_h(r,\myr)=\sk{\nabla h(\gamma(r)+\eta_\myr\myr u )}{u }-\sk{\nabla h(\gamma(r))}{u }\qquad \text{for some $\eta_\myr\in(0,1)$.} 
\end{equation*}
Since $r\in I_{\delta_{0}}(r_0)$ and $\abs{\myr}<\myr_{0}$ were arbitrary, we may let $r\to r_0$ and $\myr\to 0$ in the formula above. Using the continuity of $\nabla h$ in $W_{n_0}$, we get $\lim_{r\to r_0,\myr\to0} D_h(r,\myr)=D_h(r_0,0)=0$, verifying the continuity of $D_h$ at $(r_0,0)$ and, in particular, $\frac{\nu_i(\hat f)}{s_i-t_i}=D_{\hat f}(t_i,s_i-t_i)\to0$ and $\frac{\nu_i(\hat g)}{s_i-t_i}=D_{\hat g}(t_i,s_i-t_i)\to0$. 

After substituting $\nu_{i}(\hat f)$ and $\nu_{i}(\hat g)$ into~\eqref{eq:slope_gap_u'} and choosing $m_1\ge m_0$ large enough so that $\abs{\frac{\nu_i(h)}{s_i-t_i}}<\myepszero/2$ for both $h=\hat f,\hat g$ and $i\ge m_1$, we derive
\begin{equation*}
\abs{\sk{\nabla\hat f(\gamma(t_{i}))}{u }-\sk{\nabla\hat g(\gamma(t_{i}))}{u }}
\geq 
2\myepszero^{1/3}-2\myepszero
-\frac{|{\nu_{i}(\hat f)}|}{\abs{s_{i}-t_{i}}}
-\frac{|{\nu_{i}(\hat g)}|}{\abs{s_{i}-t_{i}}}
\ge 
2\myepszero^{1/3}-3\myepszero>\myepszero^{1/3}
\end{equation*}
for all $i\geq m_{1}$. Letting $i\to\infty$ in the above and using that both $\hat f$ and $\hat g$ are continuously differentiable on $\gamma(C_{n_0})$, by Lemma~\ref{lem:conical:new}~\eqref{it:component}, we prove that
$t\in \Fset_1$.
This finishes the proof that $\Fset_0\cup \Fset_1\cup \Fset_2$ is closed.

We will now find an upper bound for the Lebesgue measure of $\Fset_0\cup \Fset_1\cup \Fset_2\subseteq \overline{\myI}_1$, showing that it is much smaller than $\lebmi{\myI_1}$; see~\eqref{eq:small_meas} for the precise bound. 
It is clear that $\lebm{\Fset_0}=0$; let us proceed to get estimates of the Lebesgue measure of $\Fset_1$ and $\Fset_2$.
Recall the definition~\eqref{eq:def_GB} of the sets $\good$ and $\bad$ and the notation introduced in  Lemma~\ref{lem:conical-increase}. We assert that
\begin{equation}\label{eq:change}
\begin{aligned}
\abs{(\hat f\circ\gamma-\hat g\circ\gamma)\big|_{a_n}^{b_n}}
&\le7\myepszero(b_n-a_n),
&\text{if }
n\in \good,
\\ 
(\hat f\circ\gamma-\hat g\circ\gamma)\big|_{a_n}^{b_n}&\ge(b_n-a_n),
&\text{if }
n\in \bad.
\end{aligned}
\end{equation}
Indeed, recall that $C_n=(a_n,b_n)$ is an open component of $\myI_1\setminus \None\subseteq R_{0}\setminus \Nzero$, $\gamma$ is $\myepszero$-flat around $\myI_0$, $p,p_n\in \myI_0\setminus C_n$ and that both $\hat f$ and $\hat g$ have the special form of Lemma~\ref{lem:conical:new}~\eqref{it:component} on $\gamma(C_n)$ with respect to the points $p$ and $p_{n}$ and scalars $\tauzero\in (1-\myepszero,1)$ and $\tauone\in (1-\myepsone,1)$ respectively.
Therefore, we may apply Lemma~\ref{lem:conical-increase} 
to get
\begin{align*}
\abs{(\hat f\circ\gamma)\big|_{a_n}^{b_n}-\tauzero(b_n-a_n)}&\le3\myepszero\tauzero(b_n-a_n)\le3\myepszero(b_n-a_n), 
&n\in\N,\\
\abs{(\hat g\circ\gamma)\big|_{a_n}^{b_n}-\tauone(b_n-a_n)}&\le3\myepszero\tauone(b_n-a_n)\le 3\myepszero(b_n-a_n), 
&n\in \good,\\
\abs{(\hat g\circ\gamma)\big|_{a_n}^{b_n}-\tauone(a_n-b_n)}&\le3\myepszero\tauone(b_n-a_n)\le3\myepszero(b_n-a_n), 
&n\in \bad.
\end{align*}
This immediately implies the first inequality of~\eqref{eq:change}:
As both $\hat f$ and $\hat g$ are $\myepszero$-conical,  we have
$\abs{\tauone-\tauzero}\le \myepszero$ and $\tauone+\tauzero\ge2-2\myepszero$. Hence for any $n\in\good$
\[
\abs{(\hat f\circ\gamma-\hat g\circ\gamma)\big|_{a_n}^{b_n}}
\le\abs{\tauone-\tauzero}(b_n-a_n)+6\myepszero(b_n-a_n)
\le
7\myepszero(b_n-a_n)
.\] 
To see the second inequality of~\eqref{eq:change}, we note that if $n\in \bad$, 
then 
\[\abs{(\hat f\circ\gamma-\hat g\circ\gamma)\big|_{a_n}^{b_n}-(\tauzero+\tauone)(b_n-a_n)}
\le
6\myepszero(b_n-a_n).
\] 
Hence $(\hat f\circ\gamma-\hat g\circ\gamma)\big|_{a_n}^{b_n}\ge (\tauzero+\tauone-6\myepszero)(b_n-a_n)>b_n-a_n$, using $\tauone+\tauzero-6\myepszero\ge2-8\myepszero$ and $\myepszero<1/10$.

Using Lemma~\ref{lem:Lip}, $\bad\cup \good=\mathbb N$ and~\eqref{eq:change} we deduce
\[
(\hat f\circ\gamma-\hat g\circ\gamma)\big|_{a}^{b}
=\sum_{n\in \good}(\hat f\circ\gamma-\hat g\circ\gamma)\big|_{a_n}^{b_n}
+
\sum_{n\in \bad}(\hat f\circ\gamma-\hat g\circ\gamma)\big|_{a_n}^{b_n}
\ge 
\sum_{n\in \good}(\hat f\circ\gamma-\hat g\circ\gamma)\big|_{a_n}^{b_n}
+
\lebm{\bbad},
\]
where $\bbad$ is defined along with $\ggood$ in~\eqref{eq:def_GB}. 
Note that the absolute value of the first summand can be estimated
using~\eqref{eq:change} as 
\[
\abs{
\sum_{n\in \good}(\hat f\circ\gamma-\hat g\circ\gamma)\big|_{a_n}^{b_n}}
\le
7\myepszero\sum_{n\in \good}
(b_n-a_n)
=
7\myepszero\lebm{\ggood}
\le
7\myepszero\lebmi{\myI_1}.
\]
In addition, using~$\hat g=g_{\epsone,\myI_1}$, Lemma~\ref{lem:conical:new}~\eqref{it:approx_lip}, \eqref{eq:g_hyp}
and~\eqref{eq:def:eps1}, we get  
\begin{equation}\label{eq:normf-g}
\inorm{\hat f -\hat g}
\leq \inorm{\hat g-g}+\inorm{g-\hat f}
<\epsone+\myepszero\lebmi{\myI_1}/4
\leq \myepszero\lebmi{\myI_1}/2.
\end{equation}
Hence $\abs{(\hat f\circ\gamma-\hat g\circ\gamma)\big|_{a}^{b}}\le\myepszero\lebmi{\myI_1}$, and we conclude that
\begin{equation}\label{eq:measure-bad}
\lebm{\bbad}
\le
\myepszero\lebmi{\myI_1}+7\myepszero\lebmi{\myI_1}
=8\myepszero\lebmi{\myI_1}.
\end{equation}

We now show that for $t\in C_n$ with $n\in \good$ the gradients $\nabla\hat f(\gamma(t))$ and $\nabla\hat g(\gamma(t))$ differ in norm by less than the threshold $\myepszero^{1/3}$ defining the set $X_{1}$; see~\eqref{eq:good}. This will imply 
$\Fset_1\subseteq\bbad\cup \None\cup\{a,b\}$ so that
\begin{equation}\label{eq:f1}
\lebm{\Fset_1}
\le 
\lebm{\bbad}
\le 
8\myepszero\lebmi{\myI_1}
<
\myepszero^{1/3}\lebmi{\myI_1}
.
\end{equation}
Indeed, to estimate the norm of the difference between $\nabla\hat f(\gamma(t))$ and $\nabla\hat g(\gamma(t))$ we use~\eqref{eq:grad_on_UJ} of Lemma~\ref{lem:conical:new}~\eqref{it:component}, to write, for $x=\gamma(t)\in\gamma(C_n)$ and $p'=p_n$
\begin{equation}
\label{eq:good-proof}
\begin{aligned}
\norm{\nabla\hat g(\gamma(t))-\nabla\hat f(\gamma(t))}
&=
\norm{\tauone\frac{x-\gamma(p')}{\norm{x-\gamma(p')}}-
\tauzero\frac{x-\gamma(p)}{\norm{x-\gamma(p)}}
}
\\
&\le
\abs{\tauone-\tauzero}
+\tauzero\norm{
\frac{x-\gamma(p')}{\norm{x-\gamma(p')}}-
\frac{x-\gamma(p)}{\norm{x-\gamma(p)}}
}\\
&\le
\myepszero+\norm{
\frac{x-\gamma(p')}{\norm{x-\gamma(p')}}-
\frac{x-\gamma(p)}{\norm{x-\gamma(p)}}
}.
\end{aligned}
\end{equation}
Let 
\begin{equation}\label{eq:v1v2}
v_1=x-\gamma(p')=\gamma(t)-\gamma(p')
\quad\text{and}\quad
v_2=x-\gamma(p)=\gamma(t)-\gamma(p).
\end{equation}
Note that as $n\in \good$ and $t\in(a_n,b_n)$, we have $t>p'$. Note also that $p<p'$ as $p$ is to the left of $\myI_1$ and $p'\in \None=N_{g,\epsone,\myI_1}\subseteq \myI_1$. 
As $\gamma$ is $\myepszero$-flat in direction~$u$ around $\myI_0$, we get, using the notation of Remark~\ref{rem:flat_equiv}, for $p<p'<t$,
\begin{align}
&v_1=(t-p')(u +\myepszero w_{t,p'}); 
\text{ hence }
\norm{v_1}=(t-p')q_{t,p'}
\text{ with }q_{t,p'}\in(1-\myepszero,1+\myepszero),\label{al-3}
\\
&v_2=(t-p)(u+\myepszero w_{t,p});
\text{ hence }
\norm{v_2}=(t-p)q_{t,p}
\text{ with }q_{t,p}\in(1-\myepszero,1+\myepszero).\label{al-4}
\end{align}
Therefore, we have
\[
\frac{v_1}{\norm{v_1}}=\frac{u +\myepszero w_{t,p'}}{q_{t,p'}},
\quad
\frac{v_2}{\norm{v_2}}=\frac{u+\myepszero w_{t,p}}{q_{t,p}}.
\]
Note that both $\frac1{q_{t,p}}$ and $\frac1{q_{t,p'}}$ are at least $\frac{1}{1+\myepszero}\ge1-\myepszero$ and are at most $\frac{1}{1-\myepszero}\le1+2\myepszero\le2$, as $\myepszero<1/2$. 
Hence $\abs{\frac1{q_{t,p'}}-\frac1{q_{t,p}}}\le3\myepszero$ and their sum is at most $4$, so that 
\begin{align*}
\norm{\frac{v_1}{\norm{v_1}}-\frac{v_2}{\norm{v_2}}}
&\le
\norm{u \left(\frac1{q_{t,p'}}-\frac1{q_{t,p}}\right)}+
\myepszero\left(\frac1{q_{t,p'}}+\frac1{q_{t,p}}\right)\\
&\le
\abs{\frac1{q_{t,p'}}-\frac1{q_{t,p}}}
+4\myepszero
\le
7\myepszero.
\end{align*}
Together with~\eqref{eq:good-proof}, this gives
\begin{equation}\label{eq:good}
\norm{\nabla\hat f(\gamma(t))-\nabla\hat g(\gamma(t))}
\le
\myepszero+\norm{\frac{v_1}{\norm{v_1}}-\frac{v_2}{\norm{v_2}}}
\le
8\myepszero<\myepszero^{1/3}.
\end{equation}
Having verified the bound~\eqref{eq:f1} on the measure of $\Fset_{1}$, we turn our attention to $\Fset_{2}$. Let $\barf(t):=(\hat f\circ\gamma)(t)$ and $\barg(t):=(\hat g\circ\gamma)(t)$. Then~\eqref{eq:good} and $\lip(\gamma)\leq 1$ imply
\begin{equation}\label{eq:good-der}
\int_{\ggood}\abs{\barf'(s)-\barg'(s)}\,ds
\le
8\myepszero\lebm{\ggood}
\le
8\myepszero(b-a).
\end{equation}
Consider the following variant of the
uncentred Hardy-Littlewood maximal function $\mathbb M\phi$, see~\cite{analysis-book}, defined for Lebesgue measurable $\phi\colon \R\to\R$ such that $\phi\in L^1_\text{loc}(\R)$:
\[
\mathbb M\phi(t)=
\sup_{s\in\R\setminus\{t\}}\frac1{\abs{s-t}}\int_{[s,t]}\abs{\phi(r)}\,dr.
\]
We will use that for any $q>1$ the maximal function satisfies the following inequality which follows from~\cite[Theorem~21.76]{analysis-book}:
\begin{equation}\label{eq:max-op}
\int_{\R}(\mathbb M\phi(t))^q\,dt
\le
2\left(\frac{q}{q-1}\right)^q\int_{\R}\abs{\phi(t)}^q\,dt.
\end{equation}
We will use this inequality with $q=2$
and
$\phi\in L^{1}_{\text{loc}}(\R)$ defined by $\phi:=(\barf'-\barg')\chi_{\myI_1}$,
which trivially satisfies
\begin{equation}\label{eq:bound2}
\abs{\phi(r)}\le2 
\text{ for almost all }
r\in\R.
\end{equation}
Let $t\in\Fset_2\subseteq\overline{\myI}_1$ and choose $s$ according to~\eqref{eq:est-slope-neg}.
Then, $[s,t]\subseteq\overline{\myI}_1$, so that the equality
$\phi\big|_{[s,t]}=(\barf'-\barg')\big|_{[s,t]}$  holds in $L^1$, implying 
\begin{multline*}
\mathbb M\phi(t)\ge
\frac1{\abs{s-t}}\int_{[s,t]}\abs{\phi(r)}\,dr
\ge
\frac1{\abs{s-t}}\abs{\int_{[s,t]}\phi(r)\,dr}
\\=
\frac1{\abs{s-t}}\abs{\int_{[s,t]}(\barf'(r)-\barg'(r))\,dr}
=
\frac{\abs{(\barf-\barg)\big|_s^t}}{\abs{s-t}}
\ge
2\myepszero^{1/3},
\end{multline*}
where the last inequality comes from~\eqref{eq:est-slope-neg} for $s$ and $t$.
Since $t\in X_{2}$ was arbitrary, we use~\eqref{eq:max-op} with $q=2$ to derive
\begin{equation}
\label{eq:F2_measure_1}
\begin{split}
\myepszero^{2/3}\lebm{\Fset_2}
&\le\tfrac1{4}
\int_{\R}
(\mathbb M\phi(s))^2\,ds
\le
2\int_{\R}
\abs{\phi(s)}^2\,ds
\le
4\int_{\R}
\abs{\phi(s)}\,ds 
=
4\int_{\myI_1}
\abs{\phi(s)}\,ds 
\\
&=
4\int_{\bbad}
\abs{\phi(s)}\,ds 
+
4\int_{\ggood}
\abs{\phi(s)}\,ds 
\le
8\lebm{\bbad}
+
32\myepszero(b-a)
\le 96\myepszero(b-a)
. 
\end{split}
\end{equation}
Here we also used~\eqref{eq:bound2} for all $s\in\R$, followed by~\eqref{eq:good-der} and~\eqref{eq:measure-bad}.
Hence
\[
\lebm{\Fset_2}\le 
96\myepszero^{1/3}\lebmi{\myI_1}.
\]
Together with~\eqref{eq:f1} this implies
\begin{equation}\label{eq:small_meas}
\lebm{\Fset_0\cup \Fset_1\cup \Fset_2}
<
100\myepszero^{1/3}\lebmi{\myI_1}
.
\end{equation}
Recall that $\myI_1=I_\Delta(t_0)=(t_0-\Delta,t_0+\Delta)$ is a $\myepszero$-density interval for $\Fonedim$, and that $\myepszero<1/\myconst^3$, which implies $1-\myepszero>\mycons\myepszero^{1/3}$. 
Then for $\myI_1'=I_{\Delta/2}(t_{0})=(t_{0}-\Delta/2,t_{0}+\Delta/2)$ it holds that
the open set $V'=\myI_1'\setminus(\Fset_0\cup \Fset_1\cup \Fset_2)$ is of measure bounded below by $\lebmi{\myI_1'}-100\myepszero^{1/3}\lebmi{\myI_1}$ whereas, $\lebm{\myI_1'\cap \Fonedim}
\ge (1-\myepszero)\lebmi{\myI_1'}
=\tfrac12(1-\myepszero)\lebmi{\myI_1}
> 100\myepszero^{1/3}\lebmi{\myI_1}$.
This implies that
 $V'\cap \Fonedim\ne\emptyset$. Choose an open interval $V$ such that $V\subseteq V'$ and $V\cap \Fonedim\ne\emptyset$. Using that $\Fonedim$ has every portion of positive measure, and $\None\subseteq \myI_{1}$ is a closed set of measure zero, we deduce that there is an open interval $V\subseteq V'\setminus \None$ with $V\cap \Fonedim\neq \emptyset$. Part~\eqref{g_cts_diff_on_V} of the Approximation property~\ref{lem:slope}a now follows from $V\subseteq \myI_{1}\setminus \None$ and Lemma~\ref{lem:conical:new}~\eqref{it:component}.

We also have $\overline V\subseteq \overline {\myI_1'}\subseteq\myI_1$ and
$\lebmi{V}\le\lebmi{\myI_1'}=\lebmi{\myI_1}/2$. Now all assertions of part~\eqref{V_trivial_properties} of the Approximation property~\ref{lem:slope}a  are established.
 
To check
its remaining part~\eqref{f_g_slope_difference} and~\eqref{eq:est-slope}, we can immediately see that 
for any $t\in V\subseteq \myI_1\setminus(\Fset_0\cup\Fset_1\cup\Fset_2)$ and $s\in \myI_1$ we have~\eqref{eq:est-slope}; see~\eqref{eq:est-slope-neg} and the definition of $\Fset_2$. If $t\in V$ and $s\in[0,1]\setminus \myI_1$, then $t\in V\subseteq\myI_1'$ implies $\abs{s-t}\ge\lebmi{\myI_1}/4$. Therefore, using $\inorm{\hat f-\hat g}\le\myepszero\lebmi{\myI_1}/2$ from~\eqref{eq:normf-g}, we get
\[
\abs{
(\hat f(\gamma(s))-\hat f(\gamma(t)))-
(\hat g(\gamma(s))-\hat g(\gamma(t)))
}
\le
\myepszero\lebmi{\myI_1}
\le
4\myepszero\abs{s-t}
<2\myepszero^{1/3}\abs{s-t}.
\]
This proves~\eqref{eq:est-slope} for all $t\in V$ and $s\in[0,1]$, and thus 
part~\eqref{f_g_slope_difference} of the Approximation property~\ref{lem:slope}a.
\end{proof}

We will prove Theorem~\ref{thm:bm} using the Banach-Mazur game. We presently state a short description of the Banach-Mazur game; for more details see~\cite{kechris2012classical}.
\begin{defn}\label{def:BM}
Let $X$ be a non-empty topological space and $S\subseteq X$ its subset which we refer to as a target set. We define the Banach-Mazur game $G_{BM}(S)$ on $X$ as follows. Players~I and~II choose alternatively non-empty open sets $\piset_i$ (choices of Player~I) and $\piiset_i$ (choices of Player~II), such that $\piset_k\supseteq \piiset_k\supseteq \piset_{k+1}$ for each $k\ge1$, and Player~II is declared the winner if $\bigcap \piiset_k\subseteq S$. 
\end{defn}

The main result about the Banach-Mazur game which will be useful to us is the following theorem; see \cite[Theorem 8.33]{kechris2012classical}.

\begin{theorem}\label{thm:bm_standard}
Let $X$ be a non-empty topological space. Then $S\subseteq X$ is residual in $X$ if and only if Player~II has a winning strategy in  $G_{BM}(S)$.
\end{theorem}

We may immediately observe that in the case of metric spaces, with topology defined by the metric, we may check the residuality of $S$ in a slightly easier way.

\begin{theorem}\label{thm:bm_balls}
Let $X$ be a non-empty metric space. 
If Player~II has a winning strategy in  $G_{BM,\text{balls}}(S)$, the Banach-Mazur game with the restriction that both players may supply only non-empty, open balls as their choices of open sets, then $S$ is residual in $X$.
\end{theorem}
\begin{proof}
We show that Player~II has a winning strategy in $G_{BM}(S)$.
Assume Player~I supplies non-empty open sets $\piset_k$. For each $k\ge1$, Player~II picks $\phi_k\in \piset_k$ and finds $r_k>0$ such that $B(\phi_k,r_k)\subseteq \piset_k$, then gives a response $\piiset_k=B(\psi_k,\myd_k)$, via their strategy in $G_{BM,\text{balls}}$ to $B(\phi_k,r_k)$. Note that $\piiset_k$ is an open set and $\piiset_k\subseteq \piset_k$, so the sequence of open sets $(\piset_k,\piiset_k)$ satisfies Definition~\ref{def:BM}. Moreover, since Player~II's winning strategy in $G_{BM,\text{balls}}(S)$ guarantees that $\bigcap \piiset_k\subseteq S$, it also provides a winning strategy for Player~II in $G_{BM}(S)$. By Theorem~\ref{thm:bm_standard}, this implies that $S$ is residual in~$X$.
\end{proof}

Another simple fact we will need is the following lemma, in which $C^{1}(H)$ denotes the set of continuous functions $\phi\colon [0,1]^{d}\to\R$ for which $\phi|_{\interior(H)}$ is $C^{1}$.
\begin{lemma}\label{lem:lip1}
Let $f\colon [0,1]^d\to\R$ be a Lipschitz function with $\Lip(f)\le1$. Then for every $\eps>0$ there exists $g\colon [0,1]^d\to\R$ such that $\Lip(g)<1$ and $\inorm{f-g}<\eps$.
If moreover $f\in C^1(H)$ for some $H\subseteq[0,1]^d$, then the function $g$ may also be chosen to be in $C^1(H)$.
\end{lemma}
\begin{proof}
If $\inorm{f}\ne0$, let $g=rf$, with $r\in\left(\max(0,1-\frac{\eps}{\norm{f}_\infty}),1\right)$. 
\end{proof}
We are now ready to give a proof of Theorem~\ref{thm:bm}, the statement of which we repeat here for the reader's convenience.
\bm*
\begin{proof}
We prove Theorem~\ref{thm:bm} by describing a winning strategy for Player~II in the Banach-Mazur game $G_{BM,\text{balls}}(S)$ in $\lip_{1}([0,1]^{d})$, in which Player~I's choices are balls $B(\phi_k,r_k)$ and Player~II's choices are balls $B(\psi_k,\myd_k)$.
	
By Lemma~\ref{lem:cond-C} we may assume that $\gamma$ is affine modulo $\Fonedim$. Let $(0,1)=U_0\supseteq U_1\supseteq U_2\supseteq\dots$ be a sequence of open sets, such that $U_{n}\cap \Fonedim$ is dense in $\Fonedim$ for each $n\geq 1$ and $(\bigcap_{n=0}^{\infty} U_n)\cap \Fonedim\subseteq E$. Fix a strictly decreasing sequence of positive numbers $\myeta_k$ such that $\myeta_1<1/\myconst^3$ and  $\sum_{k=1}^{\infty}\myeta_k^{1/3}$ converges; for example, let $\myeta_k=2^{-3k}/\myconst^3$. For the most of the proof, we will only use that $\myeta_k\downarrow0$; the convergence property of the series will be used only at the very end of the proof; see~\eqref{eq:alpha}. In addition to defining $\psi_k\in\Lip_1([0,1]^{d})$ and $\myd_k>0$ for each $k\ge1$, Player~II also 
defines the following additional objects:
sequences of positive numbers $\myseq_k$,
unit vectors $u_k$, 
open intervals $I_k,J_k\subseteq[0,1]$, and
functions $\phi_k^{(1)}\in\Lip_1([0,1]^{d})$.
These objects have the following properties, for each $k\ge1$: 
\begin{enumerate}[(A)]
\item \label{it:first}
\label{it:phi}
$\Lip(\phi_k^{(1)})<1$ and $\phi_k^{(1)}\in B(\phi_k,r_k/4)$;
\item \label{it:psi}
$\psi_k=(\phi_k^{(1)})_{\myseq_{k},J_{k}}\in\Lip_1([0,1]^d)$ is a $\myeta_k$-conical function given by Lemma~\ref{lem:conical:new};
\item \label{it:ik}
\begin{enumerate}[(i)]
	\item\label{IJ_inclusion1} $I_k\subseteq J_{k}\subseteq I_{k-1}\cap U_{k-1}$,
	\item\label{I_nested}\label{I_measure} $\overline I_k\subseteq I_{k-1}$ and $\lebmi{I_k}\le\lebmi{I_{k-1}}/2$,
	\item\label{I_F_intersection} $I_{k}\cap \Fonedim\ne\emptyset$, 
	\item\label{flat} $\gamma$ is $\myeta_k$-flat in direction~$u_k$ around both $J_k$ and $I_k$; 
\end{enumerate}
\item \label{it:diff}
for $k\geq 2$ the function $\psi_k$ is continuously differentiable on an open neighbourhood of $\gamma(I_k)$; for points $x$ from this neighbourhood its gradient $\nabla\psi_k(x)$ is given by the right-hand side of~\eqref{eq:grad_on_UJ} with $\mytau>1-\myeta_k$ and $p\in J_k$;
\item \label{it:lem:slope:ineq}
for $k\geq 2$, 
\begin{equation*}
\abs{
(\psi_k(\gamma(s))-\psi_k(\gamma(t)))-(\psi_{k-1}(\gamma(s))-\psi_{k-1}(\gamma(t)))}
\le2\myeta_{k-1}^{1/3}\abs{s-t}
\end{equation*}
for all $t\in I_{k}$ and $s\in[0,1]$;
\item
\label{it:last}
\begin{enumerate}[(i)]
	\item\label{xik} $\myseq_k\in\left(0,\min\set{\frac{r_k}{2},\frac{\theta_{k}\leb(J_{k})}{4}}\right)$,
	\item\label{it:dk} for $k\ge2$, $\myd_{k-1}<\myeta_{k-1}\lebmi{J_{k}}/4$ and
	\label{ball_inclusion} 
	 $\overline{B}(\psi_{k-1},\myd_{k-1})\subseteq B(\phi_{k-1},r_{k-1})$. 	 
\end{enumerate} 
\end{enumerate}

Consider Player~I's first move $B(\phi_1,r_1)$.
Use Lemma~\ref{lem:lip1} to find $\phi_1^{(1)}\in B(\phi_1,r_1/4)$ such that $\Lip(\phi_1^{(1)})<1$; this establishes~\eqref{it:phi} for $k=1$. Apply Lemma~\ref{lem:flat} with $\myeps=\myeta_1$ to find an open interval $J_1\subseteq[0,1]$ and $u_1\in \Sphere^{d-1}$ such that $\gamma$ is $\myeta_1$-flat in direction~$u_1$ around $J_1$ and $J_1\cap \Fonedim\neq \emptyset$. 
Let $\myseq_{1}$ 	
	be chosen arbitrarily subject to~\eqref{xik} 
	for $k=1$ and let $\psi_1:=(\phi_1^{(1)})_{\myseq_1,J_1}$ 
be a $\myeta_1$-conical function 
given by Lemma~\ref{lem:conical:new}, verifying~\eqref{it:psi}
for $k=1$. 
We declare $\psi_1$ as the first function played by Player~II. 

Let $I_1\subseteq J_{1}$ be an open interval such that $\bar I_1\subseteq (0,1)$, $I_1\cap \Fonedim\neq \emptyset$ and $\lebmi{I_1}\le1/2$. Setting $I_0=(0,1)$, we see that all parts of~\eqref{it:ik} are satisfied with $k=1$. 

We thus verified all properties~\eqref{it:first}--\eqref{it:last} for $k=1$, 
including~\eqref{it:diff}, \eqref{it:lem:slope:ineq} 
and~\eqref{ball_inclusion}, 
for which there is nothing to verify in the case $k=1$.

Let $n\ge2$.
On Step~$n$, Player~II does the following main actions: 
\begin{itemize}
\item[-] defines $\myd_{n-1}$ so that~\eqref{ball_inclusion} is satisfied with $k=n$;
\item[-] accepts Player~I's choice of $(\phi_n,r_n)$ such that
$B(\phi_n,r_n)\subseteq B(\psi_{n-1},\myd_{n-1})$;
\item[-] defines $\psi_n\in B(\phi_n,r_n)$.
\end{itemize} 

Let $f:=\phi_{n-1}^{(1)}$, $U:=U_{n-1}$, 
$\myepszero:=\myeta_{n-1}$, $\myepsone:=\myeta_n$,
$\epszero:=\myseq_{n-1}$, 
$\myI:=I_{n-1}$, $\myI_0:=J_{n-1}$ and $f_{\epszero,\myI_{0}}:=\psi_{n-1}$.
	
These objects satisfy the conditions of Lemma~\ref{lem:slope}, namely
\begin{itemize}
	\item[-] $\lip(f)<1$, by~\eqref{it:phi} for $k=n-1$,
	\item[-] $\myI\subseteq\myI_0$, by~\eqref{IJ_inclusion1} for $k=n-1$,
	\item[-] $\myI\cap \Fonedim\ne\emptyset$, by~\eqref{I_F_intersection} for $k=n-1$, 
	\item[-] $\gamma$ is $\myepszero$-flat around $\myI_0$, 	
	by~\eqref{flat} for $k=n-1$, and
	\item[-]$f_{\epszero,R_{0}}$ is a $\myeta$-conical function, given by Lemma~\ref{lem:conical:new}, due to~\eqref{it:psi} 
	with $k=n-1$.
\end{itemize}
Let 
\begin{equation}\label{eq:def-jn} 
J_n:=\myI_1
\subseteq 
\myI\cap U
=I_{n-1}\cap U_{n-1}
\end{equation}
be the open interval given by Lemma~\ref{lem:slope} applied with these settings. This verifies
the second inclusion of~\eqref{IJ_inclusion1} with $k=n$. 

From~\eqref{it:psi} with $k=n-1$ and Lemma~\ref{lem:conical:new}~\eqref{it:approx_lip} it follows that $\inorm{\psi_{n-1}-\phi_{n-1}^{(1)}}<\myseq_{n-1}$. Therefore, by~\eqref{it:first} and~\eqref{xik} with $k=n-1$, we have $\psi_{n-1}\in B(\phi_{n-1},r_{n-1})$.
Define now a positive number $\myd_{n-1}$ arbitrarily 
so as to establish~\eqref{ball_inclusion} with $k=n$.

Assume Player~I's $n$th move is an open ball $B(\phi_n,r_n)\subseteq B(\psi_{n-1},\myd_{n-1})$ and make a choice of $\myseq_{n}$ and $\phi_{n}^{(1)}\in\Lip_1([0,1]^{d})$ 	
	verifying ~\eqref{xik} and~\eqref{it:phi} for $k=n$, using Lemma~\ref{lem:lip1} for the second choice. We declare $\psi_{n}$, defined according to~\eqref{it:psi} for $k=n$, as the $n$-th function of Player~II.

We are now ready to apply the Approximation property~\ref{lem:slope}a of $f_{\epszero,\myI_{0}}$.
Let $g:=\phi_{n}^{(1)}$, $\epsone:=\myseq_{n}$ 
	and $g_{\epsone,\myI_{1}}:=\psi_{n}$. These objects fit the framework of Lemma~\ref{lem:slope} and satisfy the hypotheses of the Approximation property~\ref{lem:slope}a,
namely
\begin{itemize}
	\item[-] $\lip(g)=\lip(\phi_n^{(1)})<1$ by~\eqref{it:phi} for $k=n$,
	\item[-] $\inorm{g-f_{\epszero,\myI_0}}
	=\inorm{\phi_{n}^{(1)}-\psi_{n-1}}
	<\myd_{n-1}
	<\myeta_{n-1}\lebmi{J_n}/4
	=\myepszero\lebmi{\myI_1}/4$, \\
	which derives from 
	$\phi_n^{(1)}\in B(\phi_n,r_n)\subseteq B(\psi_{n-1},\myd_{n-1})$, and~\eqref{it:dk} with $k=n$,
	\item[-] $\epsone\in (0,\myepszero\leb(\myI_{1})/4)$, due to~\eqref{xik} for $k=n$, and
	\item[-] $g_{\epsone,\myI_{1}}=\psi_{n}$ is a $\myepsone$-conical function given by Lemma~\ref{lem:conical:new}.
\end{itemize}

Let 
\begin{equation}\label{eq:def-in}
I_n:=V\subseteq J_n 
\end{equation}
be the open interval given by the Approximation property~\ref{lem:slope}a, applied with the settings above. We then have that~\eqref{flat}, \eqref{I_F_intersection}, \eqref{it:diff}, the remaining inclusion of~\eqref{IJ_inclusion1}, \eqref{I_nested} and~\eqref{it:lem:slope:ineq} are satisfied with $k=n$.

This verifies all properties~\eqref{it:first}--\eqref{it:last} for $k=n$.

Note that~\eqref{ball_inclusion} implies that
$\overline B(\psi_{n},\myd_{n})\subseteq B(\phi_{n},r_{n})\subseteq B(\psi_{n-1},\myd_{n-1})$
for each $n\ge2$, whilst~\eqref{it:dk} and $\myeta_n\to0$ implies $\myd_{n}\to 0$ as $n\to\infty$.
Hence the intersection of balls $B(\psi_n,\myd_n)$ is a single function 
\[
f\in\bigcap_{n=1}^{\infty}B(\psi_{n},\myd_{n})\subseteq\lip_{1}([0,1]^{d}).\] 
From~\eqref{I_nested}
we derive that the intersection of all $I_{n}$ is a single point $t^*\in \bigcap_{n=1}^{\infty}I_{n}\subseteq[0,1]$. 
Moreover, from~\eqref{I_F_intersection} and~\eqref{I_measure} it follows that $t^*$ is a limit point of $\Fonedim$ and so $t^*\in \Fonedim$. By~\eqref{IJ_inclusion1} we have $t^*\in I_n\subseteq U_{n-1}$ for all $n$. Therefore $t^*\in \Fonedim\cap\bigcap_{n=1}^{\infty} U_n\subseteq E$, implying  that $\gamma'(t^*)$ exists and $\norm{\gamma'(t^*)}=1$. 

We now show that $f$ is differentiable at $\gamma(t^*)$
in the direction of $\gamma'(t^*)$ and this directional derivative is equal to $1$ or $-1$. Since $f$ is $1$-Lipschitz, this will imply that $f$ is (fully) differentiable at $\gamma(t^*)$; see \cite[Corollary~2.6]{fitzpatrick1984differentiation},
and $\norm{\nabla f(\gamma(t^*))}=1$.

Let $\eps\in(0,1/5)$.
Consider any $n\ge1$.
From~\eqref{flat} we find a sufficiently small $\delta_n>0$ such that for all $s\in (I_n\setminus\{t^*\})\cap(t^*-\delta_n,t^*+\delta_n)$ it holds
	\begin{equation}\label{eq:conv}
	\norm{u_{n}-\gamma'(t^{*})}\leq \frac{\norm{\gamma(s)-\gamma(t^{*})-(s-t^{*})u_{n}}+\norm{\gamma(s)-\gamma(t^{*})-(s-t^{*})\gamma'(t^{*})}}{\abs{s-t^{*}}}
	\le
	2\myeta_n.
	\end{equation}
Notice that the left- and right-hand sides of the above do not depend on $s$. Hence $\myeta_n\to0$ implies $u_n\to\gamma'(t^*)$.

By~\eqref{it:diff}, we have that $\psi_n$ is continuously differentiable on an open neighbourhood of $\gamma(t^*)$ with $\nabla\psi_n(\gamma(t^*))$ given by~\eqref{eq:grad_on_UJ} with $\mytau=\mytau_n>1-\myeta_n$ and $p=p_n\in J_n$. Thus,
there is a $\delta_n'>0$ such that if $0<\abs{s-t^*}<\delta_n'$, then
\begin{align*}
&\abs{
\frac{\psi_n\big(\gamma(t^*)+(s-t^*)u_n\big)-\psi_n(\gamma(t^*))}{s-t^*}-
\langle\nabla\psi_n(\gamma(t^*)),u_n\rangle}\\
=
&\abs{
\frac{\psi_n\big(\gamma(t^*)+(s-t^*)u_n)\big)-\psi_n(\gamma(t^*))}{s-t^*}-
\mytau_n\langle\frac{\gamma(t^*)-\gamma(p_n)}{\norm{\gamma(t^*)-\gamma(p_n)}},u_n\rangle
}<\eps/2.
\end{align*}
Also, using~\eqref{flat}, we get for all $s\in I_n\setminus\set{t^{*}}$, 
\[
\abs{
\frac{\psi_n(\gamma(s))-\psi_n\big(\gamma(t^*)+(s-t^*)u_n\big)}{s-t^*}
}\le\Lip(\psi_n)\myeta_n
\le\myeta_n.
\]
Let $n_1>1$ be such that $\myeta_{n_1}<\eps/4$ and
let $n\ge n_1$. Then
\begin{equation}\label{eq:psi_inner_prod}
\abs{
\frac{\psi_n(\gamma(s))-\psi_n(\gamma(t^*))}{s-t^*}-
\mytau_n\langle\frac{\gamma(t^*)-\gamma(p_n)}{\norm{\gamma(t^*)-\gamma(p_n)}},u_n\rangle
}<\eps
\end{equation}
for $s\in(t^*-\delta_n',t^*+\delta_n')\cap I_n\setminus\set{t^{*}}$ and $\mytau_n,p_n$ as above.

Recall that $u_n\to\gamma'(t^*)$ by~\eqref{eq:conv} and $\mytau_n\ge1-\myeta_n$, so $\mytau_n\to1$. Note also that $p_n,t^*\in J_n\subseteq I_{n-1}$ for every $n$, by~\eqref{IJ_inclusion1}, and $\lebmi{I_n}\to0$ from~\eqref{I_measure}. This implies $\abs{p_n-t^*}\to0$ and we deduce that
\[
\lim_{n\to\infty} \abs{
\mytau_n\langle\frac{\gamma(t^*)-\gamma(p_n)}{\norm{\gamma(t^*)-\gamma(p_n)}},u_n\rangle
}
=
\langle \gamma'(t^*),\gamma'(t^*)\rangle
=
\norm{\gamma'(t^*)}^2=1.
\]
Thus, there is $n_2\ge n_1$ such that for each $n\ge n_2$ there is $\sigma_n\in\{-1,+1\}$ with
\[
\abs{
\mytau_n\langle\frac{\gamma(t^*)-\gamma(p_n)}{\norm{\gamma(t^*)-\gamma(p_n)}},u_n\rangle
-\sigma_n
}<\eps.
\]
However,~\eqref{it:lem:slope:ineq} and~\eqref{eq:psi_inner_prod} imply that for $n> n_2$ 
\[\abs{
\mytau_{n-1}\langle\frac{\gamma(t^*)-\gamma(p_{n-1})}{\norm{\gamma(t^*)-\gamma(p_{n-1})}},u_{n-1}\rangle
-
\mytau_n\langle\frac{\gamma(t^*)-\gamma(p_n)}{\norm{\gamma(t^*)-\gamma(p_n)}},u_n\rangle
}<2\eps+2\myeta_{n-1}^{1/3}.
\]
Therefore, choosing $n_3>n_2$ so that $\myeta_{n_3}<1/1000$, we get that for all $n> n_3$
\[
\abs{\sigma_{n-1}-\sigma_n}\le
4\eps+2\myeta_{n-1}^{1/3}<1.
\]
Hence the sign $\sigma_n$ of $\mytau_n\langle\frac{\gamma(t^*)-\gamma(p_n)}{\norm{\gamma(t^*)-\gamma(p_n)}},u_n\rangle$ does not change for $n> n_3$, and so 
\[
\lim_{n\to\infty} \mytau_n\langle\frac{\gamma(t^*)-\gamma(p_n)}{\norm{\gamma(t^*)-\gamma(p_n)}},u_n\rangle
\]
exists and is equal to $1$ or $-1$. Assume, without loss of generality, that
\[
\lim_{n\to\infty} \mytau_n\langle\frac{\gamma(t^*)-\gamma(p_n)}{\norm{\gamma(t^*)-\gamma(p_n)}},u_n\rangle
=1,
\]
and choose $n_4> n_3$ so that 
\begin{equation}\label{eq:limit}
\abs{
\mytau_n\langle\frac{\gamma(t^*)-\gamma(p_n)}{\norm{\gamma(t^*)-\gamma(p_n)}},u_n\rangle
-1
}<\eps
\end{equation}
for all $n\ge n_4$ and $\displaystyle\sum_{n= n_4}^{\infty}\myeta_n^{1/3}<\eps/2$.
Let $n\ge n_4$ and and $s\in[0,1]\setminus\set{t^{*}}$. We claim that
\begin{equation}\label{eq:alpha}
\alpha_n:=
\abs{
\frac{f(\gamma(s))-f(\gamma(t^*))}{s-t^*}-
\frac{\psi_n(\gamma(s))-\psi_n(\gamma(t^*))}{s-t^*}
}
<\eps.
\end{equation}
Indeed, using $\inorm{\psi_k-f}\to0$, 
\eqref{it:lem:slope:ineq} and~\eqref{xik}, we get 
\begin{align*}
\alpha_n
\le
\sum_{k=n}^\infty
\abs{
\frac{\psi_{k+1}(\gamma(s))-\psi_{k+1}(\gamma(t^*))}{s-t^*}
-
\frac{\psi_k(\gamma(s))-\psi_k(\gamma(t^*))}{s-t^*}
}
\le
\sum_{k=n}^\infty2\myeta_{k}^{1/3}
<\eps.
\end{align*}
Hence, whenever $0<\abs{s-t^*}<\delta_n'$, we have, using~\eqref{eq:psi_inner_prod}, \eqref{eq:limit} and~\eqref{eq:alpha}, that
\[
\abs{
\frac{f(\gamma(s))-f(\gamma(t^*))}{s-t^*}-
1
}
<3\eps.
\]
Thus, $f\circ \gamma$ is differentiable at $t^{*}$ with $(f\circ \gamma)'(t^{*})=1$. Since $\gamma$ is also differentiable with derivative of magnitude one at $t^{*}\in E$, it follows that $f$ is differentiable with derivative of magnitude one at $\gamma(t^{*})$, thus $f\in S$, so Player~II wins.
\end{proof}
\begin{remark}\label{rem:additional-diff}
	The proof of Theorem~\ref{thm:bm} given above may be slightly modified to obtain a proof of the stronger statement referred to in Remark~\ref{rem:first_typ_fn_typ_point}, namely that, in the setting of Theorem~\ref{thm:bm}, a typical function $f\in\lip_{1}([0,1]^{d})$ is differentiable with derivative of magnitude one at $\gamma(t)$ for typical $t\in F$. We describe the necessary additional details:
	
	Firstly, we modify the proof of Theorem~\ref{thm:bm} to show that the set of pairs 
	\[
	S_{\times}=\{(f,t):f\in\lip_{1}([0,1]^{d}), t\in F, f\text{ is differentiable at }\gamma(t),
	\norm{\nabla f(\gamma(t))}=1
	\}
	\]
	 is residual in $\lip_{1}([0,1]^{d})\times\Fonedim$. For this we define the Banach-Mazur game in $\lip_{1}([0,1]^{d})\times \Fonedim$, where on each turn, each of the two players supplies a direct product of an open ball around a $1$-Lipschitz function and an open interval with non-empty intersection with $\Fonedim$. Assuming that Player~I supplies $B(f_n,r_n)\times G_n$ on their $n$th turn, define $J_n\subseteq I_{n-1}\cap U_{n-1}\cap G_n$ (compare with~\eqref{eq:def-jn}). Then the reply $B(\psi_n,\rho_n)\times I_n$ from Player~II  will guarantee that Player~II wins the game in $\lip_{1}([0,1]^{d})\times\Fonedim$ with target $S_{\times}$ (here $I_n$ is defined by~\eqref{eq:def-in}).

	Having established that $S_\times$ is residual in $\lip_{1}([0,1]^{d})\times F$, by Theorem~\ref{thm:bm_standard}, it only remains to apply 
		the Kuratowski-Ulam theorem (see, for example, \cite[Theorem~8.41~(iii)]{kechris2012classical}).
		As $E$ is a relatively residual subset of $F$, 
		a typical function $f\in\lip_{1}([0,1]^{d})$ is differentiable with derivative of magnitude one at $\gamma(t)$ for typical $t\in E$.
\end{remark}
\section{Curve detection of non-coverable sets}\label{sec:construction}
In this section we prove Theorem~\ref{lemma:key}.
\paragraph{Notation and convention.}
We introduce some notation designed for $L^{\infty}$ mappings $\phi\colon I\to \R^{d}$ where $I\subseteq\R$ is a closed interval. In what follows $\phi$ will either be a $C^{1}$-smooth or a Lipschitz mapping or the derivative of such. We use the notation $I(\phi)$ to denote its domain $I$ and $\image(\phi)$ to denote the set of all its values, $\image(\phi)=\phi(I(\phi))=\phi(I)$.

For a subset $U\subseteq I$ we consider the quantity
\begin{equation}\label{eq:def-osc}
\osci{\phi}{U}:=\ess\sup\set{\norm{\phi(s)-\phi(t)}\colon s,t\in U},
\end{equation}
which corresponds to the \emph{oscillation} of $\phi$ on the set $U$. 

Recall that we call a Lipschitz or a $C^{1}$-smooth mapping $\gamma\colon I\to\R^{d}$ a \emph{curve} if the magnitude of its derivative is bounded away from zero almost everywhere. Moreover, given a $C^{1}$-smooth mapping $\gamma\colon I\to \R^{d}$ defined on a closed interval $I$ we interpret the derivative $\gamma'$ at the endpoints of $I$ as the one sided derivative so that $\gamma'$ is a well-defined mapping $I\to \R^{d}$.

Given sets $\Fhidim,U\subseteq \R^{d}$ with $U$ open we define $\Gamma_{\Fhidim}(U)$ as the collection of all $C^{1}$-smooth curves $\gamma\colon I\to \R^{d}$ with $\leb(\gamma^{-1}(\Fhidim))>0$ and $\image(\gamma)\subseteq U$. 

We let $\theta\colon\R^{d}\setminus\set{0}\to \Sphere^{d-1}$ denote standard spherical projection
\begin{equation*}
\theta(x)=\frac{x}{\norm{x}},\qquad x\in\R^{d}\setminus \set{0}.
\end{equation*}
\begin{defn}\label{def:partition}
	For each $n\ge 1$, consider the set $D_n$ of $(n-1)$-tuples $\beta=(i_1,\dots,i_{n-1})$, where each $i_j\in\N$ satisfies $1\le i_j\le 2^{d}$. The set $D_{1}$ should be interpreted as a singleton set containing the empty sequence $\emptyset$. 
For $\beta=(i_1,\dots,i_{n-1})\in D_n$ let $|\beta|=n$ and, for each $m\le n$, let $\beta|_m=(i_1,\dots,i_{m-1})\in D_m$.
	Define the order on each $D_n$ according to the lexicographical order, and extend this to an ordering on $\bigcup_{n=1}^{\infty}D_{n}$ via the following rule: if $|\beta'|<|\beta|$, then let $\beta'<\beta$ if $\beta'\le\beta|_{|\beta'|}$ and $\beta'>\beta$ otherwise.
	
	For each $n\ge1$ let $\{Q_\beta\}_{\beta\in D_n}$ be the standard dyadic partition of $[-1,1]^d$ into 
	$2^{(n-1)d}$ closed 	
	cubes with side $2^{-n+2}$, such that  
	$Q_{\beta}\subseteq Q_{\beta'}$ iff $\abs{\beta'}\leq \abs{\beta}$ and $\beta|_{\abs{\beta'}}=\beta'$.
	For each $n\ge1$ and $\beta\in D_n$ let $S_\beta=Q_\beta\cap \Sphere^{d-1}$.
	Define 
	\[
	T_n=\set{\beta\in D_n\colon
		\interior S_\beta\ne\emptyset},
	\]
	where the interior is taken with respect to the subspace topology on $\Sphere^{d-1}$. Note that for any $\beta\in T_n$, we have 
	\[S_\beta=\bigcup\set{ S_{\beta'}\colon \beta'\in T_{n+1}, 
		\beta'|_{n}=\beta},\]
	and for any $n\ge1$
	\[
	\Sphere^{d-1}=\bigcup_{\beta\in T_n}S_\beta.\]
	In particular, note that $\Sphere^{d-1}=S_{\emptyset}=\bigcup_{\beta\in T_{1}}S_{\beta}$. For each $\rho>0$, $n\ge1$ and $\beta\in T_n$ we will also denote by $B(S_{\beta},\rho)$ the open $\rho$-neighbourhood of $S_{\beta}$, considered as a subset of $\Sphere^{d-1}$, with respect to the induced topology and Euclidean metric $\norm{\cdot-\cdot}_2$ from $\R^d$.
\end{defn}

For $1\le m\le n$ and $\beta'\in T_m$, $\beta\in T_n$ we let 
\begin{equation}\label{eq:metric}
\metric(\beta',\beta)=\max\{\norm{x-y}_2\colon x\in S_{\beta'},y\in S_\beta\}.
\end{equation}
In this way, $\metric(\beta,\beta)$ is the Euclidean diameter of $S_\beta$. We note for future reference that $\metric(\beta,\beta)\to0$ as $|\beta|\to\infty$.

For each $k\in\N$ we let $\P_{k}$ denote the collection of open intervals in $[0,1]$ with consecutive $(k-1)$-th level dyadic endpoints. That is,
\begin{equation*}
\P_{k}:=\left\{\br{\frac{i-1}{2^{k-1}},\frac{i}{2^{k-1}}}\colon i=1,2,\ldots,2^{k-1}\right\}.
\end{equation*}
Further we let $\D_{k}$ denote the set of $(k-1)$-th level dyadic numbers in $[0,1]$, that is,
\begin{equation*}
\D_{k}:=\left\{\frac{i}{2^{k-1}}\colon i=0,1,2,\ldots,2^{k-1}\right\}=\bigcup_{I\in\P_{k}}\partial I.
\end{equation*}
Finally, for a subset $Y$ of $[0,1]$ we will use the notation $Y\comp$ to denote its complement $[0,1]\setminus Y$.

For the reader's convenience we repeat the statement to be proved:
\key*
The proof of Theorem~\ref{lemma:key} occupies the entire remainder of the present section and contains several lemmata, the hypotheses of which should be understood as the current setting in which the statement appears in the proof. Thus, each such statement refers to objects previously constructed.

By hypothesis there are open sets $O_{1},O_{2},\ldots\subseteq \R^{d}$ such that for each $n\in\N$ the set $O_{n}\cap \Fhidim$ is a dense subset of $\Fhidim$ and 
\begin{equation*}
\Fhidim\cap\bigcap_{n=1}^{\infty}O_{n}\subseteq \A.
\end{equation*}
We may assume that $O_1=\R^{d}$ and $O_{n+1}\subseteq O_n$ for all $n\ge1$. 

\paragraph{Iterative Construction.}
Let $L_1=c_{1}=1$ and 
\begin{equation}\label{eq:sequences_Lk_ck}
L_{k}=L_{k-1}+2^{-k},\qquad c_{k}=c_{k-1}-2^{-k},\qquad k\geq 2.
\end{equation}
\begin{remark}
	Note that $1\leq L_{k}\leq 2$  and $1/2\leq c_{k}\leq 1$ for all $k\in\N$. 
In fact, we could have chosen any strictly decreasing sequence $c_n$ and strictly increasing sequence $L_n$ with $0<c_\infty=\inf_{n\ge1} c_n<c_1=1=L_1<\sup_{n\ge1} L_n=L_\infty<\infty$. This would change constants in estimates for derivatives of $\gamma_k$ in~\eqref{lip} and~\eqref{conv_der} below, hence also in estimates for derivatives of the limit curve $\gamma_\infty$, see Lemma~\ref{lemma:gamma_inf_der}. However, a particular choice of $c_\infty$ and $L_\infty$ does not affect the strength of the result we prove. 
\end{remark}

Below, we construct sequences of 
\begin{itemize}
	\item[-] piecewise $C^{1}$-smooth, Lipschitz curves $\gamma_k:[0,1]\to\R^{d}$, 
	\item[-]  numbers $\alpha_k,\lambda_{k},r_k, \rho_{k},\psi_{k}>0$, $p_{k}\in\N$,
	\item[-] sets
	\begin{equation}\label{eq:sets:gk_hk}
	G_{k}=\bigcup_{j=1}^{p_{k}}G_{k,j}\subseteq [0,1],\quad H_{k}=\bigcup_{j=1}^{p_{k}}H_{k,j}\subseteq [0,1]
	\end{equation}
as finite unions of closed intervals $G_{k,j}$ and $H_{k,j}$,
\item[-] open sets
\begin{equation}\label{eq:sets:uk}
\quad U_{k}=\bigcup_{j=1}^{p_{k}}U_{k,j}\subseteq (0,1)^{d}
	\end{equation}
	as finite unions of open sets $U_{k,j}$,
	\item[-]
	sets $\nullA_k,W_{k}\subseteq [0,1]$ with $\nullA_{k}$ finite and $W_{k}\supseteq \nullA_{k}$ being a finite union of closed intervals;
	\item[-] functions $\beta_{k}\colon H_{k}\to T_k$,
\end{itemize}
such that the following conditions are satisfied for each $k\ge1$:
\begin{enumerate}[(A)]
\item\label{lip}
$\frac{1}{2}\leq c_{k}\leq \norm{\gamma_{k}'(t)}\leq L_{k}\leq 2$ for all $t\in [0,1]\setminus \nullA_{k}$.
\item\label{n_portion_positive} For any interval 
$I\in\P_k$ either 
\begin{equation*}
\leb\br{I\cap\gamma_k^{-1}(\Fhidim\cap O_k)\cap \bigcap_{i=1}^{k-1} W_{i}\comp}\geq \alpha_k
\quad\text{ or }\quad
I\cap\gamma_k^{-1}(\Fhidim)\cap \bigcap_{i=1}^{k-1} W_{i}\comp=\emptyset.
\end{equation*}
\item\label{curve_changes}
If $k\geq 2$, $I\in \P_{k}$ and $I\cap\gamma_{k-1}^{-1}(\Fhidim)\cap\bigcap_{i=1}^{k-1}W_{i}\comp=\emptyset$ then $\gamma_{k}(t)=\gamma_{k-1}(t)$ for all $t\in I$.
\item\label{curve_conv_both}
If $k\geq 2$, then \begin{enumerate}[(i)]
	\item\label{curve_conv} $\norm{\gamma_k(t)-\gamma_{k-1}(t)}\le \psi_{k-1}/2$ for all $t\in[0,1]$, and
	\item\label{psi_decay} $\psi_{k}\in (0,\psi_{k-1}/2)$. 
\end{enumerate}
\item\label{measure_diff_both}
If $k\geq 2$, then
\begin{enumerate}[(i)]
	\item\label{eq:measure_differ} $\leb(\set{t\in[0,1]\colon \gamma_{k}(t)\neq \gamma_{k-1}(t)})<\frac{\alpha_{k-1}}{4}$, and
	\item\label{eq:alphak}
	$0<\alpha_{k}\leq2^{-k}\alpha_{k-1}$.
\end{enumerate}
\item\label{H_in_ball}
$H_{k}$ is the union of finitely many pairwise disjoint, closed intervals $H_{k,j}$, $j=1,\ldots,p_{k}$. These sets have the following properties:
\begin{enumerate}[(i)]
	\item\label{H_k_1} If $k\geq 2$, $I\in\P_{k}$ and $\gamma_{k-1}^{-1}(\Fhidim)\cap I\cap\bigcap_{i=1}^{k-1}W_{i}\comp\neq \emptyset$ then there exists an index $j\in\set{1,\ldots,p_{k}}$ such that $H_{k,j}\subseteq I$.
	\item\label{H_k_2} $H_{k}\subseteq \bigcap_{i=1}^{k-1}W_{i}\comp$ and $\leb(H_{k,j}\cap \gamma_{k}^{-1}(\Fhidim\cap O_{k}))\geq \alpha_{k}$ for each $j=1,\ldots,p_k$.
	\item\label{H_k_2a} For all $1\leq l<k$ the components of $H_{l}$ and $H_{k}$ are either nested or disjoint. More precisely, for all $l\in\set{1,\ldots,k-1}$, $j\in\set{1,\ldots,p_{k}}$ and $i\in\set{1,\ldots,p_{l}}$ we have
			\begin{equation*}
			H_{l,i}\cap H_{k,j}=\emptyset\qquad \text{or}\qquad H_{k,j}\subseteq \interior(H_{l,i}).
			\end{equation*}
	\item\label{H_k_3} $\overline{B}(\gamma_{k}(H_{k,j}),\psi_{k})\subseteq U_{k,j}\subseteq O_{k}$ for all $j\in\set{1,\ldots,p_{k}}$.
	\item\label{H_k_4} $\beta_{k}|_{H_{k,j}}\in T_k$ is constant with value
	\begin{equation*}
	\beta_{k,j}:=\min\set{\beta\in T_k\colon 
\exists\gamma\in	\Gamma_{\Fhidim}(U_{k,j})
\text{ s.t. }\image(\theta(\gamma'))\subset B(S_{\beta},2^{-k})
}.
	\end{equation*}
	\item\label{beta_increasing} If $t\in H_{k}\cap H_{l}$ with $1\le l< k$ then 
	$\beta_{k}(t)> \beta_{l}(t)$.
\end{enumerate}   
\item\label{throw_away}(Throw away sets.)
\begin{enumerate}[(i)]
	\item\label{finite_increasing}$\nullA_{k}= \nullA_{k-1}\cup\D_{k}\cup \bigcup_{j=1}^{p_{k}}(\partial H_{k,j}\cup\partial G_{k,j})\cup\bigcup_{i=1}^{k-1}\partial W_i$ is a finite set and the restriction of $\gamma_{k}$ to each component of $[0,1]\setminus \nullA_{k}$ is $C^{1}$-smooth.
	\item\label{W_k}
	$W_{k}$ is a finite union of closed	
	subintervals of $[0,1]$, \\
	$\nullA_{k}\subseteq \interior W_{k}\cup\{0,1\}\subseteq W_k$ and $\leb(W_{k})\leq 2^{-k}\alpha_{k}$.
	\end{enumerate} 
\item (Convergence of derivatives.)
\label{conv_der}
\begin{enumerate}[(i)]
	\item\label{deriv_change_easy} If $t\in [0,1]\setminus(H_{k}\cup \nullA_{k})$ then $\norm{\gamma_{k}'(t)-\gamma_{k-1}'(t)}\leq2^{-k}$.
	\item\label{deriv_change_tricky_same_size} If $k\geq 2$ and $t\in H_{k}\setminus \nullA_{k}$ then $\Bigg|{\norm{\gamma_{k}'(t)}-\norm{\gamma_{k-1}'(t)}}\Bigg|\leq 2^{-k}$.
	\item\label{deriv_mag_constant} The mapping $t\mapsto \norm{\gamma_{k}'(t)}$ is constant on each component of $[0,1]\setminus \nullA_{k}$. 
	\item\label{deriv_change_tricky_compass} If $t\in H_{k}\setminus \nullA_{k}$ then $\theta(\gamma_{k}'(t))\in B(S_{\beta_{k}(t)},2^{-k})$. 
\end{enumerate}
\end{enumerate}
Let $\nullA_{0}=W_{0}=\emptyset$. Use Remark~\ref{rem:portion} to find a $C^{1}$-smooth curve $\gamma_1\colon [0,1]\to(0,1)^{d}$ with $\norm{\gamma_1'(t)}=1$ for all $t\in[0,1]$ and 
\begin{equation*}
\alpha_{1}:=\leb\left(\gamma_{1}^{-1}(\Fhidim)\right)>0.
\end{equation*}
Choose $\psi_{1}>0$ small enough so that $\overline{B}(\image(\gamma_{1}),\psi_{1})\subseteq (0,1)^{d}$. Further, set 
\begin{align*}
&p_{1}=r_{1}=\rho_{1}=\lambda_{1}=1,\qquad 
\nullA_{1}=\set{0,1},\qquad 
&W_{1}=\left[0,\frac{\alpha_{1}}{4}\right]\cup\left[1-\frac{\alpha_{1}}{4},1\right],\\
&G_{1}=G_{1,1}=H_{1}=H_{1,1}=[0,1], \text{ and} 
&U_{1}=U_{1,1}=(0,1)^{d}\subseteq O_{1}.
\end{align*}
Define $\beta_{1}\colon H_{1}\to T_{1}=\set{\emptyset}$ as the (only possible) constant function and set $\beta_{1,1}=\emptyset\in T_{1}$. Then for $k=1$ all conditions~\eqref{lip}--\eqref{conv_der} are either trivially satisfied or are void.

Assume now that $n\ge2$ and the conditions~\eqref{lip}--\eqref{conv_der}  are satisfied for $k=1,\ldots,n-1$. The $n$-th step of the construction proceeds as follows: Let $I_{n,1},\ldots,I_{n,p_{n}}$ be an enumeration of those intervals $I\in\P_{n}$ for which $\gamma_{n-1}^{-1}(\Fhidim)\cap I\cap\bigcap_{i=1}^{n-1}W_{i}\comp\neq\emptyset$. For each $j=1,\ldots,p_{n}$ we nominate a point $t_{n,j}\in \gamma_{n-1}^{-1}(\Fhidim)\cap I_{n,j}\cap\bigcap_{i=1}^{n-1}W_{i}\comp$. As $t_{n,j}\in I_{n,j}\cap \bigcap_{i=1}^{n-1}W_{i}\comp
$, and the latter is an open set, 
we may choose $\lambda_{n}>0$ sufficiently small so that for all $j=1,\ldots,p_{n}$
\begin{equation}\label{eq:osc_gnj}
G_{n,j}:=[t_{n,j}-\lambda_{n},t_{n,j}+\lambda_{n}]\subseteq I_{n,j}\cap\bigcap_{i=1}^{n-1}W_{i}\comp
\quad \text{and}\quad 
\osci{\gamma_{n-1}'}{G_{n,j}}\leq 2^{-(n+1)}.
\end{equation}
The second condition of~\eqref{eq:osc_gnj} can be achieved due to the fact, coming from~\eqref{finite_increasing} for $k=n-1$, that $\gamma_{n-1}$ restricted to each component of $[0,1]\setminus \nullA_{n-1}$ is $C^{1}$-smooth. We also impose a further condition on $\lambda_n$, as follows:
\begin{equation}
\label{eq:lambda_n}
\lambda_{n}\in
\Bigl(0,\tfrac18\min(\alpha_{n-1}/p_{n},\psi_{n-1})\Bigr).
\end{equation}

Observe that
\begin{equation}\label{eq:stay_inside}
\gamma_{n-1}(H_{l,i})\subseteq U_{l,i}\qquad \text{ whenever $l\in\set{1,\ldots,n-1}$ and $i\in\set{1,\ldots,p_{l}}$.}
\end{equation}
If $n=2$ this is clear. For $n>2$ we argue as follows: Given indices $l\in\set{1,\ldots,n-2}$, $i\in\set{1,\ldots,p_{l}}$ and $t\in H_{l,i}$ we may use~\eqref{curve_conv} for $l+1\le k\le n-1$ and $\psi_{k}\in(0,\psi_{k-1}/2)$ from~\eqref{H_k_3} to deduce that  
\begin{equation*}
\norm{\gamma_{n-1}(t)-\gamma_{l}(t)}\leq \sum_{k=l+1}^{n-1}\frac{\psi_{k-1}}{2}<\psi_{l}.
\end{equation*}
In case $l=n-1$ the above inequality is trivially satisfied.
Together with~\eqref{H_k_3} for $k\le n-1$ this verifies~\eqref{eq:stay_inside}. Now, let $r_{n}>0$ be chosen sufficiently small so that
\begin{equation}\label{eq:r_n:1}
B(\gamma_{n-1}(H_{l,i}),r_{n})\subseteq U_{l,i}\qquad \text{ whenever $l\in\set{1,\ldots,n-1}$ and $i\in\set{1,\ldots,p_{l}}$}
\end{equation}
and
\begin{equation}
\label{eq:r_n:2}
r_n\in(0,
2^{-(n+3)}\lambda_n
).
\end{equation}
For each $j=1,\ldots,p_{n}$ we set 
\begin{equation}\label{eq:def_unj}
U_{n,j}:=B(\gamma_{n-1}(t_{n,j}),r_{n})\cap O_{n}.
\end{equation}
Note that $U_{n,j}$ is open and has non-empty intersection with $\Fhidim$ due to the density of $\Fhidim\cap O_{n}$ in $\Fhidim$, and $\gamma_{n-1}(t_{n,j})\in \Fhidim$. 
Let
\begin{equation}\label{eq:def_beta}
\beta_{n,j}:=\min\set{\beta\in T_n\colon 
	\exists\gamma\in	\Gamma_{\Fhidim}(U_{n,j})
	\text{ s.t. }\image(\theta(\gamma'))\subset B(S_{\beta},2^{-n})
}.
\end{equation}
The hypothesis that $\Fhidim$ has every portion of positive cone width guarantees that the set for which the minimum in the definition of $\beta_{n,j}$ is considered is non-empty. 
For each $j=1,\ldots,p_{n}$ we choose, using Remark~\ref{rem:portion}, a $C^{1}$-curve $\nu_{n,j}\in\Gamma_\Fhidim(U_{n,j})$ such that 
\begin{equation}\label{eq:nu_norm}
\image(\theta(\nu_{n,j}'))\subseteq B(S_{\beta_{n,j}},2^{-n})\quad
\text{ and }\quad
\norm{\nu_{n,j}'(t)}=\norm{\gamma_{n-1}'(t_{n,j})}
\end{equation}
for all $t\in I(\nu_{n,j})$. By choosing $\rho_{n}>0$ sufficiently small, in particular,
\begin{equation}\label{eq:rho_n}
\rho_n\in(0,2^{-(n+4)}\lambda_n)
\end{equation}
and restricting each $\nu_{n,j}$ to a smaller and shifted interval and reparameterising if necessary, we may assume that for $j=1,\ldots,p_{n}$ each $\nu_{n,j}$ is defined on the interval 
\begin{equation*}
I(\nu_{n,j})=
H_{n,j}:=[t_{n,j}-\rho_{n},t_{n,j}+\rho_{n}]\subseteq \interior(G_{n,j}).
\end{equation*}
Note, for future reference, that for each $j=1,\ldots,p_n$
\begin{equation}\label{eq:nu}
\lebm{\nu_{n,j}^{-1}(\Fhidim\cap O_{n})}>0.
\end{equation}

We now verify properties~\eqref{lip}--\eqref{conv_der} for $k=n$.
We start by checking various parts of~\eqref{H_in_ball}. 
By definition of $t_{n,j}$ and $H_{n,j}$ we have that~\eqref{H_k_1} with $k=n$ is satisfied. Moreover,~\eqref{H_k_2a} with $k=n$ is readily verified: We note that $H_{n,j}$ is a subset of a connected component of $\bigcap_{i=1}^{n-1}W_{i}\comp\subseteq[0,1]\setminus \nullA_{n-1}$, whereas $\partial H_{l,i}\subseteq \nullA_l\subseteq \nullA_{n-1}$, by~\eqref{finite_increasing} with $k\le n-1$. Thus, it is clear that if $H_{l,i}\cap H_{n,j}\ne\emptyset$, then $\interior(H_{l,i})\supseteq H_{n,j}$, establishing~\eqref{H_k_2a} with $k=n$.

Let $G_{n}$ and $H_{n}$ be defined according to~\eqref{eq:sets:gk_hk}. Then
the first condition of~\eqref{H_k_2} with $k=n$ is satisfied. 
Define $\beta_{n}\colon H_{n}\to T_n$ by 
\begin{equation}\label{eq:beta_nj}
\beta_{n}(t)=\beta_{n,j},\qquad t\in H_{n,j},
\end{equation}
in accordance with~\eqref{H_k_4}, $k=n$.

We are now ready to verify~\eqref{beta_increasing} with $k=n$.
Suppose $t\in H_{n,j}\cap H_{l,i}\neq \emptyset$ for some $l\in\set{1,\ldots,n-1}$, $j\in\set{1,\ldots,p_{n}}$ and $i\in\set{1,\ldots,p_{l}}$. Then $H_{n,j}\subseteq H_{l,i}$ by~\eqref{H_k_2a}, which we already verified for $k=n$. In particular we have $t_{n,j}\in H_{l,i}$ and therefore $U_{n,j}\subseteq B(\gamma_{n-1}(t_{n,j}),r_{n})\subseteq U_{l,i}$, by~\eqref{eq:def_unj} and~\eqref{eq:r_n:1}. This trivially implies 
\begin{equation}\label{eq:Gamma_inclusion}
\Gamma_{\Fhidim}(U_{n,j})\subseteq \Gamma_{\Fhidim}(U_{l,i}).
\end{equation}
We will use this inclusion together with the following basic facts, readily verifiable from Definition~\ref{def:partition}:
\begin{align}
B(S_{\beta},2^{-n})&\subseteq B(S_{\beta|_{l}},2^{-l}), 
&\beta\in T_{n}, \label{eq:Sbeta_ball_inclusion}\\
\beta&>\beta|_{l}, &\beta\in T_{n}, \label{eq:beta_restriction}\\
\set{\beta|_{l}\colon \beta\in T_{n}}&=T_{l}.& \label{eq:k_res_Tn}
\end{align}
With these properties at hand, together with~\eqref{eq:def_beta} and~\eqref{eq:beta_nj}, we observe 
\begin{align*}
\beta_{n}(t)
=\beta_{n,j}
&=\min\set{\beta\in T_n,\colon 
	\exists\gamma\in	\Gamma_{\Fhidim}(U_{n,j})
	\text{ s.t. }\image(\theta(\gamma'))\subset B(S_{\beta},2^{-n})
}\\
&\geq \min\set{\beta\in T_n\colon 
	\exists\gamma\in	\Gamma_{\Fhidim}(U_{l,i})
	\text{ s.t. }\image(\theta(\gamma'))\subset B(S_{\beta|_{l}},2^{-l})
}\\
&> \min\set{\beta|_{l}\colon \beta\in T_n,\,
	\exists\gamma\in	\Gamma_{\Fhidim}(U_{l,i})
	\text{ s.t. }\image(\theta(\gamma'))\subset B(S_{\beta|_{l}},2^{-l})
}\\
&=\min\set{\beta\in T_l\colon 
	\exists\gamma\in	\Gamma_{\Fhidim}(U_{l,i})
	\text{ s.t. }\image(\theta(\gamma'))\subset B(S_{\beta},2^{-l})
}=\beta_{l,i}=\beta_{l}(t).
\end{align*}
The first inequality above follows from~\eqref{eq:Gamma_inclusion} and~\eqref{eq:Sbeta_ball_inclusion}, the second from~\eqref{eq:beta_restriction} and the subsequent equality from~\eqref{eq:k_res_Tn}. This completes the verification of~\eqref{beta_increasing}. 

We define the new curve $\gamma_{n}\colon[0,1]\to \R^{d}$ by
\begin{equation}\label{eq:def_gamma_n}
\gamma_{n}(t)=\begin{cases}\gamma_{n-1}(t) & \text{if }t\in[0,1]\setminus \bigcup_{j=1}^{p_{n}}\interior(G_{n,j}),\\
\nu_{n,j}(t) & \text{if }t\in H_{n,j},\, j=1,\ldots,p_{n},
\end{cases}
\end{equation}
and the condition that on each of the components of $\bigcup_{j=1}^{p_{n}}(G_{n,j}\setminus H_{n,j})$
the curve $\gamma_{n}$ is affine and hence $\norm{\gamma_n'(t)}$ is constant. Condition~\eqref{curve_changes} with $k=n$ is clearly satisfied.

Since, for each $j=1,\ldots,p_{n}$, $\gamma_{n}(H_{n,j})=\nu_{n,j}(H_{n,j})$ is a compact subset of the open set $U_{n,j}\subseteq O_{n}$, we may choose $\psi_{n}\in(0,\psi_{n-1}/2)$ establishing~\eqref{psi_decay} and~\eqref{H_k_3} for $k=n$, 

Note that 
\begin{equation}\label{eq:Gnj_cover_difference}
\set{t\in[0,1]\colon \gamma_{n}(t)\neq \gamma_{n-1}(t)}\subseteq \bigcup_{j=1}^{p_{n}}\interior(G_{n,j}),
\end{equation}
and the latter set has measure precisely $2p_{n}\lambda_{n}$. Therefore, we get~\eqref{eq:measure_differ} with $k=n$ by~\eqref{eq:lambda_n}. 
From the fact~\eqref{eq:Gnj_cover_difference} that $\gamma_{n}$ and $\gamma_{n-1}$ differ only on the pairwise disjoint intervals $G_{n,j}$ of length $2\lambda_{n}$, it also follows, using again~\eqref{eq:lambda_n} and~\eqref{lip}, that
\begin{equation*}
\lnorm{\infty}{\gamma_{n}-\gamma_{n-1}}
\leq 
(L_{n}+L_{n-1})\lambda_{n}<4\lambda_n\leq \psi_{n-1}/2.
\end{equation*}
This verifies~\eqref{curve_conv} with $k=n$. 

Recall~\eqref{eq:nu} and~\eqref{eq:def_gamma_n}, and set
\begin{equation*}
\alpha_{n}:=\min\set{2^{-n}\alpha_{n-1},\min_{1\leq j\leq p_{n}}\leb(\nu_{n,j}^{-1}(\Fhidim))}>0,
\end{equation*}
to obtain the remaining part of~\eqref{H_k_2}, and~\eqref{eq:alphak} for $k=n$. In particular, 
all parts of~\eqref{H_in_ball} are now established. 
From~\eqref{H_k_2} for $k=n$, the choice of $I_{n,j}\supseteq G_{n,j}\supseteq H_{n,j}$ and~\eqref{eq:Gnj_cover_difference} we derive~\eqref{n_portion_positive} for $k=n$.

Define $\nullA_{n}$ as in~\eqref{finite_increasing} with $k=n$. Then we see that the second condition of~\eqref{finite_increasing} with $k=n$ is satisfied, using~\eqref{finite_increasing} for $k=n-1$, \eqref{eq:Gnj_cover_difference} and the way that $\gamma|_{G_{n,j}}$ is defined for each $j\in\set{1,\dots,p_n}$. For each point in $\nullA_{n}$ we now nominate a small,  relatively open interval around this point so that the total measure of the union of all such intervals is at most $2^{-n}\alpha_{n}$. We define $W_{n}$ as the union of closures of these intervals so that~\eqref{W_k} with $k=n$ is satisfied.

The conditions~\eqref{deriv_change_tricky_same_size}, \eqref{deriv_mag_constant} and~\eqref{deriv_change_tricky_compass} are now easily verified via~\eqref{eq:def_gamma_n} and~\eqref{eq:nu_norm}. For~\eqref{deriv_change_tricky_same_size} we additionally use $t_{n,j}\in H_{n,j}\subseteq G_{n,j}$ and~\eqref{eq:osc_gnj}, 
whilst for~\eqref{deriv_mag_constant} we additionally recall~\eqref{finite_increasing} with $k=n$, \eqref{deriv_mag_constant} for $k=n-1$ and~\eqref{eq:Gnj_cover_difference}. 

If $t\in [0,1]\setminus (\nullA_{n}\cup G_{n})$ then by~\eqref{eq:Gnj_cover_difference} and~\eqref{finite_increasing} for $k=n$ we have $\gamma_{n}(t)=\gamma_{n-1}(t)$ and $\gamma_{n}'(t)=\gamma_{n-1}'(t)$. Therefore, both~\eqref{lip} and~\eqref{deriv_change_easy} are satisfied for $t$. If $t\in G_{n}\setminus (\nullA_{n}\cup H_{n})$ then without loss of generality $t$ belongs to an interval of the form $\sqbr{t_{n,j}-\lambda_{n},t_{n,j}-\rho_{n}}$, restricted to which $\gamma_{n}$ is affine. Hence, 
\begin{align*}
\gamma_{n}'(t)
&=\frac{\gamma_{n}(t_{n,j}-\rho_{n})-\gamma_{n}(t_{n,j}-\lambda_{n})}{\lambda_{n}-\rho_{n}}\\
&=\frac{\gamma_{n-1}(t_{n,j}-\rho_{n})-\gamma_{n-1}(t_{n,j}-\lambda_{n})}{\lambda_{n}-\rho_{n}}+\frac{\nu_{n,j}(t_{n,j}-\rho_{n})-\gamma_{n-1}(t_{n,j}-\rho_{n})}{\lambda_{n}-\rho_{n}}.
\end{align*}
Further, since $\image(\nu_{n,j})\subseteq U_{n,j}\subseteq B(\gamma_{n-1}(t_{n,j}),r_{n})$, by~\eqref{eq:def_unj}, we have
\begin{align*}
\norm{\nu_{n,j}(t_{n,j}-\rho_{n})-\gamma_{n-1}(t_{n,j}-\rho_{n})}&\leq r_{n}+\norm{\gamma_{n-1}(t_{n,j})-\gamma_{n-1}(t_{n,j}-\rho_{n})}\\
&\leq r_{n}+L_{n-1}\rho_{n}
\leq r_{n}+2\rho_{n}.
\end{align*}
We conclude that
\begin{equation}\label{eq:der_change_bound}
\norm{\gamma_{n}'(t)-\gamma_{n-1}'(t)}
\leq 
\osci{\gamma_{n-1}'}{G_{n,j}}+\frac{r_{n}+2\rho_{n}}{\lambda_{n}-\rho_{n}}
\leq
2^{-(n+1)}+2\frac{r_{n}+2\rho_{n}}{\lambda_{n}}
\leq 2^{-n},
\end{equation}
using~\eqref{eq:osc_gnj}, \eqref{eq:r_n:2} and~\eqref{eq:rho_n}. This verifies~\eqref{deriv_change_easy} for $k=n$. Moreover,~\eqref{eq:der_change_bound}, \eqref{lip} for $k=n-1$ and~\eqref{eq:sequences_Lk_ck} imply~\eqref{lip} for $t$ and $k=n$. To  complete the verification of~\eqref{lip}, note that for $t\in H_n\setminus\nullA_n$ we can find $j\in\set{1,\dots,p_n}$ such that $t\in H_{n,j}$, implying $\norm{\gamma_n'(t)}=\norm{\nu_{n,j}'(t)}=\norm{\gamma_{n-1}'(t_{n,j})}$ and finally apply~\eqref{lip} for $k=n-1$. Thus, all conditions~\eqref{lip}--\eqref{conv_der} hold for the objects of step $k=n$.

\paragraph{The limit curve $\gamma_{\infty}$.} 
By~\eqref{curve_conv_both} 
the sequence of mappings $(\gamma_{k})_{k=1}^{\infty}$ converges in the supremum norm to a 
mapping $\gamma_{\infty}\colon [0,1]\to\R^{d}$.
\begin{lemma}\label{lemma:gamma_inf_der}
		The limit curve $\gamma_{\infty}$ has the following properties:
		\begin{enumerate}[(i)]
			\item\label{gamma_inf_lip} The mapping $\gamma_{\infty}$ is Lipschitz with $\lip(\gamma_{\infty})\leq 2$.
			\item\label{gamma_inf_codomain} The mapping $\gamma_{\infty}$ may be viewed as a mapping $[0,1]\to (0,1)^{d}$, that is, with codomain $(0,1)^{d}$. 				
		\item\label{gamma_inf_deriv} For almost every $t\in[0,1],$ all mappings $\gamma_{k}$ with $k\in\N\cup\set{\infty}$ are differentiable at $t$ and there exists $m=m(t)\in\N$ such that $\gamma_{\infty}'(t)=\gamma_{k}'(t)$ for all $k\geq m$.
			\item\label{gamma_inf_lip_curve} For almost every $t\in [0,1]$, $\gamma_{\infty}$ is differentiable at $t$ with $\norm{\gamma_{\infty}'(t)}\geq 1/2$. Consequently $\gamma_{\infty}$ is a Lipschitz curve.			
		\end{enumerate}
\end{lemma}	
	\begin{proof}
		Part~\eqref{gamma_inf_lip} is trivial, since $\gamma_{\infty}$ is the uniform limit of mappings $\gamma_{k}$, all of which satisfy $\lip(\gamma_{k})\leq 2$, by~\eqref{lip}. For~\eqref{gamma_inf_codomain}, observe that~\eqref{curve_conv_both} implies
		\begin{equation*}
		\lnorm{\infty}{\gamma_{\infty}-\gamma_{1}}\leq \sum_{k=1}^{\infty}\frac{\psi_{k}}{2}<\psi_{1}.
		\end{equation*} 
		Recall that $\psi_{1}>0$ was chosen sufficiently small so that $\overline{B}(\image(\gamma_{1}),\psi_{1})\subseteq (0,1)^{d}$. We conclude that $\image(\gamma_{\infty})\subseteq (0,1)^{d}$, that is, we may view $\gamma_{\infty}$ as a mapping $[0,1]\to (0,1)^{d}$.
		Moving on to~\eqref{gamma_inf_deriv}, we use~\eqref{measure_diff_both}
		to infer
		\begin{equation}\label{eq:gamma_neq_gamma_k}
	\leb\br{\bigcup_{n=m}^{\infty}B_{n}}
	\leq 
	\sum^{\infty}_{n=m}\frac{\alpha_{n}}{4}
\leq\frac{\alpha_m}{2},
	\qquad\text{with }B_{n}:=\set{t\in[0,1]\colon \gamma_{\infty}(t)\neq \gamma_{n}(t)},
	\end{equation}
		for all $m\geq 1$. Letting $C_{m}:=\bigcap_{n=m}^{\infty}B_{n}\comp\subseteq [0,1]$, we conclude that $\bigcup_{m=1}^{\infty}C_{m}$ has full measure in $[0,1]$. Moreover, for each $m\geq 1$ and almost every density point $t$ of $C_{m}$ we have that all mappings $\gamma_{k}$ with $k\in\N\cup\set{\infty}$ are differentiable at $t$ and $\gamma_{\infty}'(t)=\gamma_{n}'(t)$ for all $n\geq m$. The statement of~\eqref{gamma_inf_deriv} follows. Finally, note that part~\eqref{gamma_inf_lip_curve} follows immediately from~\eqref{gamma_inf_deriv} and~\eqref{lip}.
\end{proof}

Let
\begin{equation}\label{eq:def_finf_h}
\Finf:=\gamma_{\infty}^{-1}(\Fhidim)\cap \bigcap_{i=1}^{\infty}W_{i}\comp, \qquad H:=\bigcap_{n=1}^{\infty}\bigcup_{k=n}^{\infty}\interior(H_{k}).
\end{equation}
\begin{lemma}\label{lemma:Finf_H_properties}
	The sets $\Finf$ and $H$ have the following properties:
	\begin{enumerate}[(i)]
		\item\label{H_Gdelta} The set $H$ is $G_{\delta}$.
		\item\label{H_diff} The derivative $\gamma_k'(t)$ exists for every $t\in H$ and every $k\in \N$.
		\item\label{Finf_closed} The set $\Finf$ is closed.		
		\item\label{HcapF_preim} $H\cap \Finf\subseteq \gamma_{\infty}^{-1}(\A)$.	
		\item\label{every_portion_pos} The set $\Finf\subseteq[0,1]$ is non-empty and has every portion of positive Lebesgue measure.
		\item\label{Finf_in_Hkj}
		For every $k\in\N$ and every component $H_{k,j}$ of $H_{k}$ we have
		\begin{equation*}
		\interior(H_{k,j})\cap \Finf\neq \emptyset.
		\end{equation*}
		\item\label{relatively_res} 
		The set $H\cap\Finf$ is a relatively residual subset of $\Finf$.			
	\end{enumerate}
	\end{lemma}
\begin{proof}
	The assertion~\eqref{H_Gdelta} for $H$ is obvious, and existence of $\gamma_k'(t)$ in~\eqref{H_diff} follows from~\eqref{finite_increasing},
		as $H\cap \bigcup_{i=1}^{\infty}\nullA_{i}\neq \emptyset$, by~\eqref{throw_away} and~\eqref{H_k_2}.
	To see that
	$\Finf$ is a closed subset of $[0,1]$ 
	we argue that $\bigcup_{i=1}^\infty W_i$ is a relatively open subset of $[0,1]$. Indeed, 
	by~\eqref{throw_away} we have that $\partial W_i\subseteq \nullA_{i+1}\subseteq \interior W_{i+1}\cup\{0,1\}$ for each $i\ge1$.
	Hence, as $0,1\in W_1$,
	\[
	\{0,1\}\cup\bigcup_{i=1}^\infty \interior W_i
	\subseteq
	\bigcup_{i=1}^\infty W_i
	\subseteq
	\{0,1\}\cup\bigcup_{i=1}^\infty \interior W_i.
	\]
	It remains to note that 
	\[
	\{0,1\}\cup\bigcup_{i=1}^\infty\interior W_i
	=\left(\{0,1\}\cup\interior W_1\right)
	\cup
	\bigcup_{i=2}^\infty \interior W_i
	\]
	is a union of relatively open sets in $[0,1]$. This proves~\eqref{Finf_closed}.
	
	For~\eqref{HcapF_preim}, it suffices to show that $H\cap\Finf\subseteq \gamma_{\infty}^{-1}\br{\bigcap_{n=1}^{\infty}O_{n}\cap \Fhidim}$. Fix $t\in H\cap\Finf$ and $n\in\N$. Since $t\in\Finf$ we have $\gamma_{\infty}(t)\in \Fhidim$. Since $t\in H$, we may choose $k\geq n$ such that $t\in\interior(H_{k})$. Now conditions~\eqref{curve_conv_both} and~\eqref{H_k_3} guarantee that 
	\begin{equation*}
	\gamma_{\infty}(t)=\lim_{l\to\infty}\gamma_{l}(t)\in \overline B(\gamma_{k}(t),\psi_{k})\subseteq O_{k}\subseteq O_{n}.
	\end{equation*}
	Hence $\gamma_{\infty}(t)\in O_{n}\cap \Fhidim$.
	
	Finally, we prove~\eqref{every_portion_pos}, \eqref{Finf_in_Hkj} and~\eqref{relatively_res} simultaneuously. By~\eqref{throw_away}, the set $\Finf$ contains no dyadic numbers. Therefore, it suffices to verify the `every portion of positive measure' condition of~\eqref{every_portion_pos} on all intervals $I\in\P_{k}$ for all $k\geq 2$. Further, to prove~\eqref{Finf_in_Hkj} we may assume that $k\geq 2$, since $H_{1,1}=[0,1]$ is the only component of $H_{1}$ and contains all other $H_{k,j}$. Let $k\geq 2$ and $I\in\P_{k}$ be such that $I\cap \Finf\neq\emptyset$. We claim that
	\begin{equation}\label{eq:claim}
	I\cap\gamma_{k-1}^{-1}(\Fhidim)\cap\bigcap_{i=1}^{k-1}W_{i}\comp\neq\emptyset.
	\end{equation}
	Otherwise, applying~\eqref{curve_changes} inductively for $k'\geq k$ yields that $\gamma_{\infty}|_{I}=\gamma_{k'}|_{I}=\gamma_{k-1}|_{I}$ for all $k'\geq k$. But this implies
	\begin{equation*}
	I\cap \Finf=I\cap \gamma_{\infty}^{-1}(\Fhidim)\cap \bigcap_{i=1}^{\infty}W_{i}\comp\subseteq I\cap\gamma_{k-1}^{-1}(\Fhidim)\cap\bigcap_{i=1}^{k-1}W_{i}\comp=\emptyset,
	\end{equation*}
	contrary to our assumption. This proves~\eqref{eq:claim}. By~\eqref{H_k_1} there exists $j_{0}\in\set{1,\ldots,p_{k}}$ with $H_{k,j_{0}}\subseteq I$. For the proof of~\eqref{Finf_in_Hkj} we write the next part of the argument for an arbitrary, fixed $j\in\set{1,\ldots,p_{k}}$. By~\eqref{H_k_2} we have
	\begin{equation*}
	H_{k,j}\subseteq \bigcap_{i=1}^{k-1}W_{i}\comp
	\text{ and }
	\lebm{H_{k,j}\cap\gamma_{k}^{-1}(\Fhidim\cap O_{k})}\geq \alpha_{k}.
	\end{equation*}
	Applying~\eqref{eq:gamma_neq_gamma_k}, we infer
	\begin{equation*}
	\lebm{H_{k,j}\cap\gamma_{\infty}^{-1}(\Fhidim\cap O_{k})}\geq \frac{\alpha_{k}}{2}.
	\end{equation*}
	Finally we apply~\eqref{W_k} and~\eqref{eq:alphak} to derive
	\begin{equation*}
	\lebm{H_{k,j}\cap\gamma_{\infty}^{-1}(\Fhidim\cap O_{k})\cap \bigcap_{i=1}^{\infty}W_{i}\comp}\geq \frac{\alpha_{k}}{2}-\sum_{i=k}^{\infty}\leb(W_{i})\geq \frac{\alpha_{k}}{8}>0,
	\end{equation*}
	which implies
	\begin{equation}\label{eq:pos_measure}
	\leb\left(\Finf\cap \interior(H_{k,j})\right)>0.
	\end{equation}
	This proves~\eqref{Finf_in_Hkj}. Since $k\geq 2$ and $I\in \P_{k}$ were arbitrary and $H_{k,j_{0}}\subseteq I$, taking $j=j_{0}$ in~\eqref{eq:pos_measure} verifies~\eqref{every_portion_pos} and further proves that the sets $\bigcup_{i=k}^{\infty}\interior(H_{i})\cap \Finf$ are dense in $\Finf$ for all $k\in\N$. Hence~\eqref{relatively_res} is also verified.
\end{proof}	
	
For each $t\in H$, let $(k_n(t))_{n=1}^{\infty}$ be the increasing sequence of positive integers such that $t\in\interior(H_{k})$ if and only if $k\in\{k_n(t)\colon n\ge1\}$. In other words, setting $k_0(t)=0$, we let 
	\begin{equation}\label{def-kn}
	k_{n}(t):=\min\set{k>k_{n-1}(t)\colon t\in \interior(H_{k})}, \qquad t\in H, n\geq 1.
	\end{equation} 
In places where the relevant point $t\in H$ is clear, we often shorten $k_n(t)$ to $k_n$. 
\begin{remark}\label{rem:kn}
Recall, from~\eqref{H_k_2a}, that any two components $H_{k,j}$, $H_{l,i}$ of $H_{k}$ and $H_{l}$ respectively with $k\neq l$ are either pairwise disjoint or strongly nested in the sense that one is contained in the interior of the other.  
	This implies the following additional property, which we will use later on:
	if $t\in H$, $k_{n}:=k_{n}(t)$ for $n\geq 1$ and $s\in\interior(H_{k_{m},j_{m}})$ for some $m\ge1$, then	
	\begin{equation}
	\label{rem:kn:s} 
k_n(s)=k_n(t)=k_n
\text{ for all }
1\le n\le m.
	\end{equation}
\end{remark}
Let $t\in H$ and $k_n=k_n(t)$. By~\eqref{beta_increasing} we have that $\beta_{k_{n}}(t)>\beta_{k_{n-1}}(t)$ for each $n\ge2$. This implies that for each fixed $m\ge1$, the sequence $\beta_{k_{n}}(t)|_m$ eventually becomes  constant. Define the infinite sequence $\beta(t)=(i_m)_{m=1}^{\infty}$ by the condition 
\begin{equation}\label{def-B}
\beta(t)|_m=\lim_{n\to\infty}\beta_{k_{n}}(t)|_m
\qquad 
\text{ for each }m\ge1,
\text{ where }t\in H.	
\end{equation}
Note for future reference that 
\begin{equation}\label{eq:beta-incr}
\beta_{k_{n}}(t)|_m\le \beta(t)|_m
\qquad\text{ for all }n,m\ge1,
\text{ where }t\in H.
\end{equation}	

Recall from 
Lemma~\ref{lemma:Finf_H_properties}, part~\eqref{H_diff}, that for each $t\in H$ and $k\in\N$ the derivative $\gamma_k'(t)$ exists.
The next lemma gives an estimate of how close the derivatives of $\gamma_k$ on $H$ are in terms of the function $\metric$ defined in~\eqref{eq:metric}. 
\begin{lemma}\label{lemma:cauchy_bound_der}
	Let $t\in H$, see~\eqref{eq:def_finf_h},  and $k_{n}:=k_{n}(t)$ be defined according to~\eqref{def-kn}. Let $k_1\le k\leq l$ and let $p,q\ge1$ be maximal such that $k_{p}\leq k$ and $k_{q}\leq l$.
	 Then
	\begin{equation*}
	\norm{\gamma_{l}'(t)-\gamma_{k}'(t)}\leq 2\metric(\beta_{k_{q}}(t),\beta_{k_{p}}(t))+7\cdot 2^{-k_{p}}.
	\end{equation*}
\end{lemma} 
\begin{proof}
Clearly $1\le k_p\le k_q$.
	By~\eqref{deriv_change_easy} we have
	\begin{equation*}
	\norm{\gamma_{l}'(t)-\gamma_{k_{q}}'(t)}\leq \sum_{m=k_{q}+1}^{l}2^{-m}\leq 2^{-k_{q}}\leq 2^{-k_{p}},
	\end{equation*}
	and similarly
	\begin{equation*}
	\norm{\gamma_{k}'(t)-\gamma_{k_{p}}'(t)}\leq 2^{-k_{p}}.
	\end{equation*}
	To obtain an estimate for $\norm{\gamma_{k_{q}}'(t)-\gamma_{k_{p}}'(t)}$, we compare separately the magnitudes and directions of these vectors. By~\eqref{deriv_change_easy} and~\eqref{deriv_change_tricky_same_size} the magnitudes differ by
	\begin{equation*}
	\Biggl|{\norm{\gamma_{k_{q}}'(t)}-\norm{\gamma_{k_{p}}'(t)}}\Biggr|
	\leq \sum_{m=k_{p}+1}^{k_{q}}2^{-m}\leq2^{-k_{p}},
	\end{equation*}
	and with~\eqref{deriv_change_tricky_compass} we can bound the difference of directions by
	\begin{equation*}
	\norm{\theta(\gamma_{k_{q}}'(t))-\theta(\gamma_{k_{p}}'(t))}\leq \metric(\beta_{k_{q}}(t),\beta_{k_{p}}(t))+2\cdot 2^{-k_{p}}.
	\end{equation*}
	Combining the last two inequalities and using that $\norm{\gamma_{n}'(t)}\leq 2$, from~\eqref{lip}, for all $n\ge1$ we deduce
	\begin{equation*}
	\norm{\gamma_{k_{q}}'(t)-\gamma_{k_{p}}'(t)}\leq 2^{-k_{p}}+2\left(\metric(\beta_{k_{q}},\beta_{k_{p}}(t))+2\cdot 2^{-k_{p}}\right)
	\end{equation*}
	The inequality of Lemma~\ref{lemma:cauchy_bound_der} now follows by the triangle inequality.
\end{proof}
The previous lemma enables us to establish convergence of the derivatives $\gamma_{k}'(t)$ at points $t\in H$.
\begin{lemma}\label{lemma:convergence_derivatives} 
Let $t\in H$. Then the sequence $(\gamma_{k}'(t))_{k=1}^{\infty}$ converges and 
\begin{equation}
\theta\left(\lim_{k\to\infty}\gamma_{k}'(t)\right)=\bigcap_{n\ge1} S_{\beta(t)|_n},
\end{equation}
where $\displaystyle\beta(t)=\lim_{n\to\infty} \beta_{k_{n}}(t)$ and
$k_n=k_n(t)$ are defined in~\eqref{def-B} and~\eqref{def-kn} respectively.
\end{lemma}
\begin{proof}
Given $\varepsilon>0$ choose $M\in\N$ such that $2^{-M+2}\sqrt{d}<\eps$, 
i.e.\ the
diameter $\metric(\beta,\beta)$ of any $S_\beta$ with $\beta\in T_{n}$, $n\ge M$, is less than $\eps$ (see Definition~\ref{def:partition}). Let $N>M$ be such that
for any $n\ge N$, it holds that
$\beta_{k_n}(t)|_M=\beta(t)|_M$.

Given $l> k\geq k_{N}$ we choose $p,q\in\N$ maximal so that $k_{p}\leq k$ and $k_{q}\leq l$. Then, by Lemma~\ref{lemma:cauchy_bound_der}, we have
\begin{equation*}
\norm{\gamma_{l}'(t)-\gamma_{k}'(t)}\leq 2\metric(\beta_{k_{q}}(t),\beta_{k_{p}}(t))+7\cdot 2^{-k_{p}}<2\varepsilon+7\varepsilon.
\end{equation*}	
Here we used that $p,q\ge N$ to deduce
$\beta_{k_p}(t)|_M=\beta_{k_q}(t)|_M=\beta(t)|_M$, and subsequently
$S_{\beta_{k_p}(t)},S_{\beta_{k_q}(t)}\subseteq S_{\beta(t)|_M}$. Hence
$\metric(\beta_{k_p}(t),\beta_{k_q}(t))\le \metric(S_{\beta(t)|_M},S_{\beta(t)|_M})<\eps$. We also used $2^{-k_{p}}\leq 2^{-k_{N}}\leq 2^{-N}<\varepsilon$.

We thus conclude that $(\gamma_{k}'(t))_{k=1}^{\infty}$ is a Cauchy sequence and hence converges. Moreover, for any $p\geq N$ we have, by~\eqref{deriv_change_tricky_compass}, that $\theta(\gamma_{k_{p}}'(t))\in B(S_{\beta_{k_{p}}(t)},2^{-k_{p}})\subseteq B(S_{\beta(t)|_{M}},2^{-k_{p}})$. Letting $p\to\infty$ we conclude that the vector $x:=\theta(\lim_{k\to\infty}\gamma_{k}'(t))$ belongs to $S_{\beta(t)|_{M}}$. Since $M\in\N$ could have been chosen arbitrarily large, this proves $x\in \bigcap_{n=1}^{\infty} S_{\beta(t)|_{n}}$. It is clear that the latter has diameter $0$, thus the statement of the lemma follows.
\end{proof}
For each $k\ge1$, let us recall~\eqref{eq:def_finf_h} and define
	\begin{equation}\label{def:Omegak}
	\Omega_{k}:=\set{t\in H\colon \exists\sigma=\sigma(t,k)>0\text{ s.t. }\beta(s)|_{k}\leq \beta(t)|_{k}\text{ for all $s\in [t-\sigma,t+\sigma]\cap H$}}
	\end{equation} 
and
\begin{equation}\label{eq:Omega_Einf}
\Einf:=\bigcap_{k=1}^{\infty}\Omega_{k}.
\end{equation}
We now show that each of the sets $\Omega_k$ is non-empty and moreover, that each $\Omega_{k}\cap\Finf$ contains a relatively open and dense subset of $H\cap\Finf$. 
Together with Lemma~\ref{lemma:Finf_H_properties}\eqref{relatively_res}
this will imply that $\Einf\cap \Finf$ is relatively residual in $\Finf$.

\begin{lemma}\label{lemma:Omega_k_Einf_prop}	
	The sets $\Omega_{k}$, $\Einf$ and $\Finf$ defined in~\eqref{def:Omegak}, \eqref{eq:Omega_Einf} and \eqref{eq:def_finf_h} have the following properties:
		\begin{enumerate}[(i)]
			\item\label{Omegak_rel_open_dense} For each $k\ge1$, the set $\Omega_{k}\cap \Finf$ contains a relatively open and dense subset of $H\cap \Finf$.
			\item\label{Einf_rel_res} The set $\Einf\cap \Finf$ is a non-empty, relatively residual subset of $\Finf$.
		\end{enumerate}		
	\end{lemma}
	\begin{proof}
		A subset $R$ of a topological space $X$ contains an open, dense set if and only if $R$ intersects every non-empty, open set in a set of non-empty interior. We prove part~\eqref{Omegak_rel_open_dense} by verifying this equivalent condition for the sets $R=\Omega_{k}\cap \Finf$ and topological space $X=H\cap \Finf$ with the subspace topology inherited from $[0,1]$. Thus, fixing $k\geq 1$ and an open interval $U\subseteq\R$ with $U\cap H\cap \Finf\neq\emptyset$, our task is to find an open interval $V\subseteq U$ such that 
		\begin{equation}\label{eq:V_condition}
		\emptyset\neq V\cap H\cap\Finf\subseteq \Omega_{k}\cap \Finf.
		\end{equation}
		Since $U\cap H\ne\emptyset$, the
		set
		\begin{equation*}
		\set{\beta(r)|_{k}\colon r\in U\cap H}
		\end{equation*}
		is a finite, non-empty set. Therefore, there exists $t\in U\cap H$ such that
		\begin{equation}\label{eq:beta_max}
		\beta(t)|_{k}=\max\set{\beta(r)|_{k}\colon r\in U\cap H}.
		\end{equation}
		Note that a priori we do not know whether $t$ belongs to $\Finf$. Let $k_{n}=k_n(t)$ be defined by~\eqref{def-kn}.
		We then have $\beta(t)=\lim_{n\to\infty}\beta_{k_{n}}(t)$, see~\eqref{def-B}. Therefore, we may choose $n_{0}\in\N$ large enough so that $\beta_{k_{n}}(t)|_{k}=\beta(t)|_k$ is constant for all $n\geq n_{0}$. Fix $n\geq n_{0}$ and consider the component $H_{k_{n},j_{n}}$ of $H_{k_{n}}$ containing $t$. We additionally take $n$ sufficiently large so that $H_{k_{n},j_{n}}\subseteq U$. Now we seek to verify~\eqref{eq:V_condition} for $V:=\interior(H_{k_{n},j_{n}})\subseteq U$.
		First note that the set $V\cap H\cap \Finf$ is non-empty: By Lemma~\ref{lemma:Finf_H_properties}, part~\eqref{Finf_in_Hkj}, the set $V\cap \Finf$ is a non-empty, relatively open subset of $\Finf$. Therefore, by Lemma~\ref{lemma:Finf_H_properties}, part~\eqref{relatively_res}, it has non-empty intersection with $H$.	
		Let $s\in V\cap H\cap \Finf$. Then, by Remark~\ref{rem:kn}, we have $k_{i}(s)=k_{i}(t)=k_{i}$ for $1\leq i\leq n$. 
		Hence, using	
		\eqref{eq:beta-incr} and the choice~\eqref{eq:beta_max} of $t$ we get
		\begin{equation*}
		\beta(s)|_{k}\geq \beta_{k_{n}}(s)|_{k}=\beta_{k_{n}}(t)|_{k}=\beta(t)|_{k}\geq \beta(s)|_{k}.
		\end{equation*}
		We conclude that $\beta(s)|_{k}=\beta(t)|_{k}$. Taking $\sigma=\sigma(s,k)>0$ sufficiently small so that $[s-\sigma,s+\sigma]\subseteq U$ and using~\eqref{eq:beta_max}, we verify that $s\in \Omega_{k}$. Hence $s\in \Omega_{k}\cap \Finf$.	
		
		We turn our attention now to part~\eqref{Einf_rel_res}. From part~\eqref{Omegak_rel_open_dense} it follows that $\Einf\cap\Finf(:=Z)$ is a relatively residual subset of $H\cap \Finf(:=Y)$. Recall in addition, that $H\cap \Finf$ is a relatively residual subset of $\Finf(:=X)$ and that $\Finf$ is closed (Lemma~\ref{lemma:Finf_H_properties}, parts~\eqref{relatively_res} and~\eqref{Finf_closed}), thus a Baire space in its own right. Therefore, to prove~\eqref{Einf_rel_res}, it suffices to recall the following general topological statement, which may be verified easily using \cite[\S10~IV~Theorem~1]{MR1296876}. \newline \emph{Let $X$ be a topological space, $Y\subseteq X$ be a residual subset of $X$ and $Z\subseteq Y$ be a relatively residual subset of $Y$. Then $Z$ is a residual subset of $X$.}
\end{proof}
We are now ready to make an important step and verify that the limit curve $\gamma_\infty=\lim\gamma_k$ is differentiable everywhere in $E_\infty$~\eqref{eq:Omega_Einf}, and its derivative is the limit of derivatives of $\gamma_k$.
\begin{lemma}\label{lemma:differentiable}
 	Let $t\in\Einf$. Then the Lipschitz curve $\gamma_{\infty}$ is differentiable at $t$ with
 	\begin{equation*}
 	\gamma_{\infty}'(t)=\lim_{k\to\infty}\gamma_{k}'(t).
 	\end{equation*}
 	Moreover, we have
 	\begin{equation*}
 	\lim_{\delta\to 0}\osci{\gamma_\infty'}{[t-\delta,t+\delta]}=0.
 	\end{equation*}
 \end{lemma}
\begin{proof}
	Fix $\varepsilon>0$. Let $N\in\N$ be sufficiently large such that $2^{-N+2}\sqrt{d}<\eps$, i.e.\ the diameter of any $S_\beta$ with $\beta\in T_n$, $n\ge N$, is less than $\eps$.
	As $t\in\Einf\subseteq\Omega_N$, let $\sigma(t,N)>0$ be given by the definition~\eqref{def:Omegak} of $\Omega_{N}$. 

	Recall $\Omega_N\subseteq H$, so $t\in H$. Let $\kappa_{n}=k_n(t)$ be the sequence of indices defined by~\eqref{def-kn}. For each $n\in\N$ let $j_{n}\in\set{1,\ldots,p_{n}}$ be the index with $t\in \interior(H_{\kappa_{n},j_{n}})$. By Lemma~\ref{lemma:convergence_derivatives}, there exists $L(t)=\lim_{n\to\infty} \gamma_{\kappa_n}'(t)$. Choose $M\ge N$ sufficiently large so that
	\begin{itemize}
		\item $\beta_{\kappa_{m}}(t)|_{N}=\beta(t)|_N$ is constant for all $m\geq M$, 
		\item $H_{\kappa_{M},j_{M}}\subseteq [t-\sigma(t,N),t+\sigma(t,N)]$,
		\item $\norm{\gamma_{\kappa_{M}}'(t)-L(t)}\leq\varepsilon$.
	\end{itemize}
		Choose $\eta$ sufficiently small so that $[t-\eta,t+\eta]\subseteq H_{{\kappa_{M},j_{M}}}$. Then, by~\eqref{H_k_4}, we have $\beta_{\kappa_{M}}(s)= \beta_{\kappa_{M}}(t)=:\beta$ for all $s\in [t-\eta,t+\eta]$. 
By~\eqref{rem:kn:s} of Remark~\ref{rem:kn}
	we conclude that $k_n(s)=k_n(t)=\kappa_n$ for all $s\in[t-\eta,t+\eta]$ and $1\le n\le M$.

	Let $s\in[t-\eta,t+\eta]$, $l\geq \kappa_{M}$ and choose $q\in\N$ maximal so that $k_{q}(s)\leq l$.
	As $l\ge \kappa_M=k_M(s)$, we conclude that $q\ge M$. Using, in addition,~\eqref{beta_increasing} with $s\in H_{\kappa_M}\cap H_{k_q}$, $t\in\Omega_N$ and~\eqref{eq:beta-incr}, we get
	\begin{equation*}
	\beta_{k_q}(s)|_{N}
	\geq \beta_{\kappa_{M}}(s)|_{N}
	=\beta_{\kappa_{M}}(t)|_{N}
	=\beta(t)|_{N}
	\geq \beta(s)|_{N}
	\geq \beta_{k_q}(s)|_{N}.
	\end{equation*}
	Therefore, 
	$\beta_{k_q}(s)|_{N}
	=\beta_{\kappa_{M}}(s)|_{N}
	=\beta(s)|_N$ 
	so that 
	\[\metric(\beta_{k_q}(s),\beta_{\kappa_{M}}(s))
	\leq \metric(\beta(s)|_N,\beta(s)|_N)<\varepsilon.\]
	Then, applying Lemma~\ref{lemma:cauchy_bound_der} we get
	\begin{equation*}
	\norm{\gamma_{l}'(s)-\gamma_{\kappa_{M}}'(s)}
	\leq 2\metric\left(\beta_{k_q}(s),\beta_{\kappa_{M}}(s)\right)+7\cdot 2^{-\kappa_{M}}<9\varepsilon.
	\end{equation*}
	From this we conclude that $\Lip\left((\gamma_{l}-\gamma_{\kappa_{M}})|_{[t-\eta,t+\eta]}\right)\leq 9\varepsilon$. Since $\gamma_{l}$ converges uniformly to $\gamma_\infty$ we deduce that $\Lip\left((\gamma_\infty-\gamma_{\kappa_{M}})|_{[t-\eta,t+\eta]}\right)\leq 9\varepsilon$. Further, by~\eqref{deriv_change_tricky_compass}, \eqref{deriv_mag_constant} and~\eqref{lip} we have 
	\begin{equation*}
	\osci{\gamma_{\kappa_{M}}'}{[t-\eta,t+\eta]}
	\leq2\diam\left(B(S_{\beta_{\kappa_{M}}},2^{-\kappa_{M}})\right)=2\left(2\cdot 2^{-\kappa_{M}}+\metric(\beta_{\kappa_{M}},\beta_{\kappa_{M}})\right)\leq 4\varepsilon.
	\end{equation*}
	It follows that for all $h\in [-\eta,\eta]$
	\begin{equation*}
	\norm{\gamma_{\kappa_{M}}(t+h)-\gamma_{\kappa_{M}}(t)-h\gamma_{\kappa_{M}}'(t)}\leq 4\varepsilon\abs{h}.
	\end{equation*}
	Using
	\begin{align*}
	\norm{\tfrac{\gamma_\infty(t+h)-\gamma_\infty(t)}{h}-\gamma_{\kappa_{M}}'(t)}
	\leq
	\norm{\tfrac{(\gamma_\infty-\gamma_{\kappa_{M}})(t+h)-(\gamma_\infty-\gamma_{\kappa_{M}})(t)}{h}}+
	\norm{\tfrac{\gamma_{\kappa_{M}}(t+h)-\gamma_{\kappa_{M}}(t)}{h}-\gamma_{\kappa_{M}}'(t)},
	\end{align*}	
	we now derive, for all $h\in [-\eta,\eta]\setminus \set{0}$,
\[
\norm{\tfrac{\gamma_\infty(t+h)-\gamma_\infty(t)}{h}-L(t)}
	\leq
9\eps+4\eps+
\norm{\gamma_{\kappa_{M}}'(t)-L(t)}\le14\eps.
\]	
	
Since $\varepsilon>0$ was arbitrary, this verifies the differentiability of $\gamma_\infty$ at $t$ with $\gamma_\infty'(t)=L(t)$. For the `moreover' part of the lemma, we observe that
	\begin{equation*}
	\osci{\gamma_\infty'}{[t-\eta,t+\eta]}
	\leq 2\lip\left((\gamma_\infty-\gamma_{\kappa_{M}})|_{[t-\eta,t+\eta]}\right)+
	\osci{\gamma_{\kappa_{M}}'}{[t-\eta,t+\eta]}
	\leq 18\varepsilon+4\varepsilon.
	\end{equation*}
\end{proof}
We now reparameterise the curve $\gamma_{\infty}$ to obtain a curve $\gamma\colon I(\gamma)\to(0,1)^{d}$ satisfying the conclusions of Theorem~\ref{lemma:key}. Let 
\begin{equation*}
\ell(\gamma_{\infty}):=\int_{0}^{1}\norm{\gamma_{\infty}'(s)}\,ds
\end{equation*}
denote the length of the curve $\gamma_{\infty}$. Define a mapping $\varphi\colon [0,1]\to[0,\ell(\gamma_{\infty})]$ by
\begin{equation*}
\varphi(t)=\int_{0}^{t}\norm{\gamma_{\infty}'(s)}\,ds,\qquad t\in[0,1].
\end{equation*}
\begin{lemma}\label{lemma:varphi_der}
	The function $\varphi\colon [0,1]\to[0,\ell(\gamma_{\infty})]$ has the following properties:
	\begin{enumerate}[(i)]
		\item\label{vphi_bilip} The function $\varphi$ is bilipschitz with $\lip(\varphi),\lip(\varphi^{-1})\leq 2$.
		\item\label{vphi_x_diff} There is a set $X\subseteq[0,1]$ of full measure with $\Einf\subseteq X$ such that for every $t\in X$ both $\gamma_{\infty}$ and $\varphi$ are differentiable at $t$ and 
		\begin{equation*}
		\varphi'(t)=\norm{\gamma_{\infty}'(t)}\geq\frac{1}{2}.
		\end{equation*}
	\end{enumerate}
	\end{lemma}
	\begin{proof}
		Part~\eqref{vphi_bilip} follows from Lemma~\ref{lemma:gamma_inf_der}\eqref{gamma_inf_deriv} and~\eqref{lip}. For part~\eqref{vphi_x_diff}, let $X$ be defined as the set of points $s\in[0,1]\setminus \bigcup_{i=1}^{\infty}\nullA_{i}$ at which all curves $\gamma_{k}$ with $k\in\N\cup\set{\infty}$ are differentiable and $\gamma_{\infty}'(s)=\lim_{k\to\infty}\gamma_{k}'(s)$. The inequality $\geq \frac{1}{2}$ in the statement is now a consequence of~\eqref{lip}. Recalling that the sets $\nullA_{i}$ are finite, it follows immediately from Lemma~\ref{lemma:gamma_inf_der}\eqref{gamma_inf_deriv} that $X$ has full measure. Further, from Lemma~\ref{lemma:differentiable} and $\Einf\subseteq H\subseteq [0,1]\setminus \bigcup_{i=1}^{\infty}\nullA_{i}$, we derive that $X$ contains $\Einf$. 		
		Fix $t\in X$, $\varepsilon\in(0,1/4)$ and let $k\in\N$ be large enough so that $2^{-k}\leq\varepsilon$ and $\norm{\gamma_{k}'(t)-\gamma_{\infty}'(t)}\leq \varepsilon$. Next choose $\delta>0$ small enough so that $[t-\delta,t+\delta]$ is contained in a single component of $[0,1]\setminus \nullA_{k}$. From~\eqref{deriv_change_easy}--\eqref{deriv_mag_constant} it follows that 
		\begin{equation*}
		\Biggl|{\norm{\gamma_{l}'(s)}-\norm{\gamma_{k}'(t)}}\Biggr|\leq 2^{-k}\le\eps
		\end{equation*}
		for all $l\geq k$ and all $s\in [t-\delta,t+\delta]\cap X$, implying
		\begin{equation*}
		\Biggl|{\norm{\gamma_{\infty}'(s)}-\norm{\gamma_{k}'(t)}}\Biggr|\le\eps
		\end{equation*}
		for all such $s$.
		Hence, for almost all $s\in[t-\delta,t+\delta]$ we have
		\begin{equation*}
		\Biggl|{\norm{\gamma_{\infty}'(s)}-\norm{\gamma_{\infty}'(t)}}\Biggr|
		\le
		\Biggl|{\norm{\gamma_{\infty}'(s)}-\norm{\gamma_{k}'(t)}}\Biggr|
		+\eps
		\leq 2\eps,
		\end{equation*}
		and therefore, for all $h\in[-\delta,\delta]$ we have
		\begin{equation*}
		\Bigl|
		\varphi(t+h)-\varphi(t)-h\cdot \norm{\gamma_{\infty}'(t)}\Bigr|\leq 
		\int_{t}^{t+h}\Bigg|{\norm{\gamma_{\infty}'(s)}-\norm{\gamma_{\infty}'(t)}}\Bigg|\,ds\leq 2\varepsilon\abs{h}.
		\end{equation*}
\end{proof}
We now use results and constructions of Section~\ref{sec:construction} to finish the proof Theorem~\ref{lemma:key}.
\begin{proof}[Proof of Theorem~\ref{lemma:key}]
We find a curve $\gamma$ satisfying all assertions Theorem~\ref{lemma:key}, except that its domain is an interval $I(\gamma)$ and not necessarily $[0,1]$. It is then a trivial matter to adjust $\gamma$ so that its domain is $[0,1]$ and all assertions of the theorem remain valid. We comment briefly on the required modification at the very end.

From Lemma~\ref{lemma:varphi_der}\eqref{vphi_x_diff}, 
and an appropriate form of the inverse function theorem it follows that 
\begin{equation}\label{eq:varphi_inv_der}
(\varphi^{-1})'(\varphi(r))=\frac{1}{\norm{\gamma_{\infty}'(r)}}
\end{equation}
for all $r\in X$, where $X$ and $\varphi$ are given by Lemma~\ref{lemma:varphi_der}. 
More precisely,~\eqref{eq:varphi_inv_der} is obtained by an application of \cite[Theorem~1.2]{RR_inverse_fn_thm} to $U=(0,1)$, $n=1$, $x_{0}=r\in X$ and $f=\varphi$. 
Note that the condition $f'(x_{0})=\varphi'(r)\in \operatorname{Isom}(\R,\R)$ is satisfied due to Lemma~\ref{lemma:varphi_der}\eqref{vphi_x_diff}. Since in this case $f=\varphi$ is invertible, the function $h$ given by the conclusion of \cite[Theorem~1.2]{RR_inverse_fn_thm} necessarily coincides with $\varphi^{-1}$ on its domain. 

We recall sets $E_\infty$ and $F_\infty$ from~\eqref{eq:Omega_Einf} and~\eqref{eq:def_finf_h} to define 
\begin{equation*}
\Fonedim:=\phi(\Finf),\qquad
E:=\phi(\Einf\cap \Finf)
\end{equation*}
and $\gamma\colon [0,\ell(\gamma_{\infty})]\to (0,1)^{d}$ by
\begin{equation*}
\gamma(t)=\gamma_{\infty}(\phi^{-1}(t)).
\end{equation*}
By 
Lemmata~\ref{lemma:Finf_H_properties}~\eqref{every_portion_pos} and~\ref{lemma:Omega_k_Einf_prop}~\eqref{Einf_rel_res}, the sets $E$ and $\Fonedim$ are  non-empty.
We verify the assertions~\eqref{key1}--\eqref{key4} of Theorem~\ref{lemma:key} for $\Fonedim$, $E$ and $\gamma$. The properties~\eqref{key1} 
and~\eqref{key2} are invariant under bilipschitz transformations. Therefore $\Fonedim$ and $E$ inherit these properties from $\Finf$ and $\Einf\cap \Finf$; see 
Lemmata~\ref{lemma:Finf_H_properties}~\eqref{every_portion_pos} and~\ref{lemma:Omega_k_Einf_prop}~\eqref{Einf_rel_res}. Moreover,~\eqref{key4} is immediate from the definitions of $\gamma$, $E$, $\Fonedim$ and 
Lemma~\ref{lemma:Finf_H_properties}\eqref{HcapF_preim}.
To complete the proof, we verify~\eqref{key3} and~\eqref{key3a}. Fix $t\in \varphi(X)$. Then $t=\varphi(r)$ for some $r\in X$. Applying~\eqref{eq:varphi_inv_der} we conclude that $\varphi^{-1}$ is differentiable at $t$ with derivative $(\varphi^{-1})'(t)=\frac{1}{\norm{\gamma_{\infty}'(r)}}$. Moreover, $\gamma_{\infty}$ is differentiable at $\varphi^{-1}(t)$ by Lemma~\ref{lemma:varphi_der}. It follows that $\gamma$ is differentiable at $t$ with
\begin{equation*}
\gamma'(t)
=\gamma_{\infty}'(\varphi^{-1}(t))\cdot (\varphi^{-1})'(t)
=\gamma_{\infty}'(r)\cdot\frac{1}{\norm{\gamma_{\infty}'(r)}}.
\end{equation*}
Clearly, from the above, we also have $\norm{\gamma'(t)}=1$. Since $E\subseteq \varphi(X)$, part~\eqref{key3} is satisfied. For $t_{0}=\varphi(r_{0})\in E$ and any $t,s\in[t_0-\delta,t_0+\delta]\cap \varphi(X)$ 
Lemma~\ref{lemma:varphi_der}\eqref{vphi_bilip} implies that the preimages $r_t:=\varphi^{-1}(t)$ and $r_s:=\varphi^{-1}(s)$ belong to $\varphi^{-1}[t_0-\delta,t_0+\delta]\cap X\subseteq[r_0-2\delta,r_0+2\delta]$,
and then~\eqref{eq:varphi_inv_der}, together with Lemma~\ref{lemma:gamma_inf_der}\eqref{vphi_x_diff}, implies $\abs{(\phi^{-1})'(t)-(\phi^{-1})'(s)}\le 4\osci{\gamma_\infty'}{[r_0-2\delta,r_0+2\delta]}$. Therefore, we obtain
\begin{align*}
\norm{\gamma'(t)-\gamma'(s)}
&=
\norm{\gamma_{\infty}'(r_t)\cdot (\varphi^{-1})'(t)-
\gamma_{\infty}'(r_s)\cdot (\varphi^{-1})'(s)}
\\
&\le
\norm{\gamma_\infty'(r_t)}
\abs{(\phi^{-1})'(t)-(\phi^{-1})'(s)}
+
\abs{(\varphi^{-1})'(s)}\osci{\gamma_{\infty}'}{[r_{0}-2\delta,r_{0}+2\delta]}
\\
&\le 10\osci{\gamma_{\infty}'}{[r_{0}-2\delta,r_{0}+2\delta]},
\end{align*}
where for the last inequality we used that both
$\Lip(\phi^{-1})$ and $\Lip(\gamma_\infty)$ are bounded from above by $2$; see Lemmata~\ref{lemma:gamma_inf_der}\eqref{gamma_inf_lip} and~\ref{lemma:varphi_der}\eqref{vphi_bilip}. 

The proof of part~\eqref{key3a} is now completed by the `moreover' conclusion of Lemma~\ref{lemma:differentiable}.

Let us now comment on why we may assume that the domain $I(\gamma)$ of $\gamma$ is the interval $[0,1]$, as in the statement of Theorem~\ref{lemma:key}. Note that $I(\gamma)$ has the form $[0,a]$ for some $a:=\ell(\gamma_{\infty})>0$. If $a\geq 1$ then we choose a closed interval $J\subseteq (0,a)$ of length strictly less than one such that the endpoints of $J$ are density points of $\Fonedim$. We then redefine the sets $\Fonedim$ and $E$ by intersecting with $J$. Finally, we choose a closed interval $J'\subseteq [0,a]$ of length one with $J\subseteq \interior(J')$ and redefine $\gamma$ by restricting to $J'$ and then shifting so that $\gamma$ is defined on $[0,1]$. If $a<1$ then we extend the curve $\gamma$ arbitrarily to $[0,1]$ and leave the sets $\Fonedim\subseteq [0,a]$ and $E\subseteq [0,a]$ unchanged. In both cases all assertions~\eqref{key1}--\eqref{key4} of Theorem~\ref{lemma:key} are preserved.
\end{proof}
\section{Typical non-differentiability on coverable sets}\label{sec.proof}
In this section we prove Theorem~\ref{thm:main:nondiff}, that is, we show that any set in $(0,1)^{d}$ which may be covered by a countable union of closed, purely unrectifiable sets avoids, for the typical function $f\in\lip_{1}([0,1]^{d})$, the set of points where $f$ has a directional derivative.

\paragraph{Notation.} We will write $\lip([0,1]^{d})$ for the set of all Lipschitz functions $[0,1]^{d}\to\R$. Further, recall that for a subset $U\subseteq [0,1]^{d}$, we let $C^{1}(U)$ denote the set of continuous functions $f\colon[0,1]^{d}\to \R$ with the property that $f|_{\interior(U)}$ is $C^{1}$.

The following lemma is a simplification of~\cite[Lemma~2.3]{maleva_preiss2018}, in the case when $P\subseteq[0,1]^{d}$ is 
a closed set.
We also only state it in the case when the function $\omega_0(t)$ of \cite[Lemma~2.3]{maleva_preiss2018} is constant.
\begin{lemma}\label{3x}
Suppose that $P\subseteq H\subseteq(0,1)^d$, where $P$ is closed and $H$ is open, 
the function $g\colon (0,1)^d\to\R$ belongs to $C^1(H)$ and $\omega_0,\eta\in(0,1)$.
Then there are $\xi_0,r_0\in(0,\omega_0/2]$ such that
if $h\colon [0,1]^d\to\R$ satisfies
\begin{equation}\label{eq:lm_cnd_h}
\text{$\abs{h(x)-g(x)}\le2\xi_0$ for all $x\in[0,1]^d$,}
\end{equation}
then 
for all $x\in P$ and 
$\norm{y}\le r_0$, it holds
\begin{equation}\label{3x.2}
\abs{h(x+y)-h(x)-\spr{\nabla g(x)}{y}}\le\eta r_0.
\end{equation}
\end{lemma}
\begin{proof}
Denote $\rho_H(x):=\dist(x,[0,1]^d\setminus H)$;
let $\Psi$ be the set of functions $\psi\in\Lip_1([0,1]^d)$ satisfying $0\le \psi(x)\le\tfrac12\min(\rho_H(x),\omega_0)$
and such that
\begin{equation}\label{eq:g'_eta_condition}
\text{$\norm{\nabla g(y)-\nabla g(z)}\le\tfrac12\eta$
whenever $x\in H$ and $\max(\norm{y-x},\norm{z-x})<\psi(x)$}.
\end{equation}
Since $0\in\Psi$, the function $\phi(x):=\sup\{\psi(x)\colon  \psi\in\Psi\}$
is well-defined. We also have $\phi\in\Psi$ since for any
$x,y,z$ satisfying $x\in H$ and $\max(\norm{y-x},\norm{z-x})<\phi(x)$
there is $\psi\in\Psi$ such that
$\max(\norm{y-x},\norm{z-x})<\psi(x)$ and hence
$\norm{\nabla g(y)-\nabla g(z)}\le\tfrac12\eta$.

Let $w\in H$ be arbitrary. Choose $\eps_w\in(0,\omega_0/2)$ such that $B(w,3\eps_w)\subseteq H$ and the bound $\norm{\nabla g(y)-\nabla g(z)}\le\frac12\eta$ holds for $y,z\in B(w,2\eps_w)$.
Then the function defined by $\psi_{w}(x):=\max(0,\eps_w-\norm{x-w})$
satisfies $\psi_{w}=0$ outside of the ball $B(w,\eps_w)$ and
$0\le\psi_{w}(x)\le \eps_w \le \tfrac12\min(\rho_H(x),\omega_0)$
for all $x\in B(w,\eps_w)$. This, together with the choice of $\eps_{w}$, clearly ensures that~\eqref{eq:g'_eta_condition} is satisfied for $\psi=\psi_{w}$.
Hence $\psi_{w}\in\Psi$ and
we infer that
$\phi(w)\ge\psi_{w}(w)=\eps_w>0$.
Consequently, $\phi$ is strictly positive on $H$.
Let $\phi_0=\inf\{\phi(x)\colon x\in P\}$; as $P$ is compact we have that $0<\phi_0\le\tfrac12\omega_0$. Furthermore,
whenever $x\in P$ and $\norm{y}< \phi_0$, it holds
\[\abs{g(x+y)-g(x)-\spr{\nabla g(x)}{y}}
\le\norm{y}\sup_{z\in B(x,\norm{y})} \norm{\nabla g(z)-\nabla g(x)}
\le \tfrac12\eta\norm{y}.
\]

To prove~\eqref{3x.2},  we let
$r_0:=\phi_0/2\in(0,\omega_0/2]$ and $\xi_0:=\phi_0\eta/16=r_0\eta/8\in (0,\omega_{0}/2]$ and consider an arbitrary function $h\colon[0,1]^{d}\to\R$ satisfying~\eqref{eq:lm_cnd_h}. 
Then,
whenever $x\in P$ and $\norm{y}\le r_0<\phi_0\le\phi(x)$, we have 
\begin{align*}
\abs{h(x+y)-h(x)-\spr{\nabla g(x)}{y}}
&\le
4\xi_0 +\abs{g(x+y)-g(x)-\spr{\nabla g(x)}{y}}\\
&\le
4\xi_0 +  \tfrac12\eta\norm{y}
\le \eta r_0.
\end{align*}
\end{proof}

Hence,~\cite[Lemma~2.9]{maleva_preiss2018} may be restated in the following way, in the case of a compact purely unrectifiable set $P$: note that such sets are automatically uniformly purely unrectifiable; see~\cite{maleva_preiss2018,acp2010differentiability}.
\begin{lemma}\label{U}
Suppose $P\subseteq H\subseteq (0,1)^d$,
$P$ is a closed, uniformly~purely~unrectifiable~set, 
$H$~is open,
$\omega_0\in(0,1)$
and $f\in\Lip([0,1]^d)\cap C^1(H)$.
Then for every
$e\in\R^d$ and
$\eta>0$
there is $g\colon[0,1]^d\to\R$,
$\xi_{0},r\in(0,\omega_0)$
and an open set $U\subseteq(0,1)^d$ 
such that
\begin{enumerate}[(i)]
\item\label{pU.1}
$P\subseteq U\subseteq H$,
\item\label{pU.2}
$g \in \lip([0,1]^{d})\cap C^1(U)$,
$\Lip(g)\le\max(\Lip(f),\norm{e})+\eta$
and $\inorm{g-f}\le\omega_0$,
\item\label{pU.4}
if a function $h\colon[0,1]^d\to\R$ satisfies
$\abs{h(x)-g(x)}\le 2\xi_0$ for all $x\in[0,1]^d$, then 
$\sup_{\norm{y}\le r}\abs{h(x+y)-h(x)-\spr{e}{y}}\le\eta r$ for all $x\in P$.
\end{enumerate}
\end{lemma}

We are now ready to prove Theorem~\ref{thm:main:nondiff}, which we restate here, in a slightly different form, for the reader's convenience.
\begin{theorem}[restatement of Theorem~\ref{thm:main:nondiff}]\label{thm:nondiff-cov}
Let $P\subseteq (0,1)^{d}$ be an $F_{\sigma}$, purely unrectifiable set. Then a typical $f\in\lip_{1}([0,1]^{d})$ has no directional derivatives at every point of $P$ and, moreover, for a typical $f\in\lip_{1}([0,1]^{d})$ it holds that $\mathcal Df(x,v)=[-1,1]$ for every $x\in P$ and every $v\in\Sphere^{d-1}$.
\end{theorem}
\begin{proof} We may assume that $P$ is closed. Indeed, if the statement holds for $P$ closed, it extends immediately to countable unions of closed $P_{n}$ as follows: Letting $S_n=\text{NonD}(P_n)$ denote the collection of functions $f\in\lip_1([0,1]^d)$ which are non-differentiable at every point of $P_{n}$ in the very strong sense described in the statement of the theorem, we get that each $S_n$ is residual. Hence,
\[
\set{f\in\lip_1([0,1]^d)\colon
\mathcal Df(x,v)=[-1,1]
\text{ for any }
x\in \displaystyle \bigcup_{n=1}^{\infty}P_{n}
\text{ and }v\in\Sphere^{d-1}
}
\supseteq
\bigcap_{n\ge1}S_n
\] 
is residual too.

Let $P\subseteq(0,1)^d$ be a closed purely unrectifiable set and $S:=\text{NonD}(P)$. We now consider a Banach-Mazur game $G_{BM,\text{balls}}$ in $\Lip_1([0,1]^d)$ with the target set $S$ and show that Player~II has a winning strategy; by Theorem~\ref{thm:bm_balls} this will imply that $S$ is residual in $\Lip_1([0,1]^d)$. 

Assume $H_0=(0,1)^d$.
Fix a sequence $(e_n)$ of vectors with $\norm{e_n}<1$ such that the collection $(e_n)$ is dense in the unit ball $B(0,1)$. 
Let $g_0(x)=0$ for all $x\in[0,1]^d$ and $\omega_0=1$.

On reaching step $n$ in the Banach-Mazur game the two players would have constructed a nested sequence of open balls
and Player~II would have additionally defined a nested sequence of open sets $H_0\supseteq\dots\supseteq H_{n-1}\supseteq P$.

Assume $B(f_n,r_n)$ is the $n$th choice of Player~I. Using that smooth functions are dense in $C([0,1]^{d})$ followed by Lemma~\ref{lem:lip1}, we choose $f_n^{(1)}\in C^1([0,1]^{d})$ such that $\Lip(f_{n}^{(1)})<1$ and $\inorm{f_n-f_n^{(1)}}<r_n/2$. Choose $\eta_n\in(0,2^{-n})$ s.t. $\max(\Lip(f_n^{(1)}),\norm{e_n})+\eta_n<1$.
Let $\omega_{n}= \min(r_{n}/2,2^{-n})$.

Apply now Lemma~\ref{U} to $P$ and $H:=H_{n-1}$,
$\omega_0:=\omega_{n}$, $f:=f_n^{(1)}$, $e:=e_n$ and $\eta:=\eta_n$ to get function $g_n:=g\colon[0,1]^d\to\R$, $\xi_n:=\xi_0$, $\eps_n:=r\in(0,\omega_n)$
and an open set $H_n:=U$.

From Lemma~\ref{U}~\eqref{pU.2}, 
we have that $g_n\in\Lip_1([0,1]^d)$ and $\inorm{g_n-f_n^{(1)}}\le\omega_n\le r_n/2$,  hence $\inorm{g_n-f_n}<r_n$. 
Choose $\rho_n\in\left(0,\min\left(\xi_n,2^{-n}\right)\right)$ such that $\overline B(g_n,\rho_n)\subseteq B(f_n,r_n)$. Let Player~II's response be $B(g_n,\rho_n)$.

Since $\overline B(g_n,\rho_n)\subseteq B(g_{n-1},\rho_{n-1})$ and $\rho_n\to0$, we conclude that the intersection of balls $B(g_n,\rho_n)$ is a single function $h\in\Lip_1([0,1]^d)$. We now show that $h$ has no directional derivatives at any $x\in P$ and, moreover, $\mathcal Dh(x,v)\supseteq[-1,1]$ for every $x\in P$ and every $v\in\Sphere^{d-1}$. As it is clear that $\mathcal Dh(x,v)\subseteq[-1,1]$ from $\Lip(h)\le1$, this will imply the required equality.

 Indeed, fix any $x\in P$, $v\in\Sphere^{d-1}$ and $n\ge1$. Recall the application of Lemma~\ref{U} which provided $g_n=g$ and $\xi_n=\xi_0$.
 Since $\inorm{h-g}=\inorm{h-g_n}\le\rho_n\le\xi_n=\xi_0$, we see that $h$ satisfies condition~\eqref{pU.4} of
Lemma~\ref{U}. Hence
$\abs{h(x+y)-h(x)-\spr{e_n}{y}}
\le \eta_n \eps_n$ whenever $\norm{y}\le\eps_n$. 
In particular, letting $y=\eps_n v$, we get
\[\abs{\frac{h(x+\eps_n v)-h(x)}{\eps_n}-\spr{e_n}{v}}
\le \eta_n.\]
As the vectors $e_n$ form a dense subset of the closed ball $\overline B(0,1)$, $0<\eps_n\le \omega_n\le2^{-n}\to0$ and $0<\eta_n\le2^{-n}\to0$, we
get that $\mathcal Dh(x,v)\supseteq[-1,1]$, hence $\mathcal Dh(x,v)=[-1,1]$.
\end{proof}

\section{Comparison with vector-valued mappings}\label{section:proj}
For $d,l\in\N$ we denote by $\lipdl$ the space of Lipschitz mappings $f\colon [0,1]^{d}\to\R^{l}$ with $\lip(f)\leq 1$, viewed as a complete metric space with the supremum metric. In most of the paper, we have $l=1$ and abbreviate $\lipd$ to $\lip_{1}([0,1]^{d})$. Merlo~\cite{merlo} shows that whenever $d\leq l$ and $\A\subseteq (0,1)^{d}$ is a non-coverable set in the sense of Theorem~\ref{thm:main_result}, there is a residual set $S\subseteq \lipdl$ for which every mapping $f=(f_{1},\ldots,f_{l})\in S$ has a directional derivative in $\A$; see \cite{merlo} Proposition~3.3 and Theorem~2.8. At first glance, it may appear that this statement is closely related to Theorem~\ref{thm:main_result}. Indeed, for such non-coverable $\A\subseteq (0,1)^{d}$ and residual $S\subseteq \lipdl$, the natural projection mappings
\begin{equation*}
\rho_{j}\colon \lipdl\to\lipd,\qquad f=(f_{1},\ldots,f_{l})\mapsto f_{j},
\end{equation*}
for $j=1,\ldots,l$, give rise to sets $\rho_{1}(S),\ldots,\rho_{l}(S)\subseteq \lipd$ in which all functions have a directional derivative in $\A$. Since $S$ is residual in $\lipdl$, we might hope that the projections $\rho_{j}(S)$ are also large in some sense in $\lipd$ and therefore hope to obtain via \cite{merlo} a statement of the form of Theorem~\ref{thm:main_result} with full differentiability weakened to existence of a directional derivative. However, the next theorem demonstrates that this argument fails badly: even very large residual sets in $\lipdl$ may project to negligible sets in $\lipd$. Thus, Theorem~\ref{thm:main_result} and its implications in Theorems~\ref{thm:typ_ds} and~\ref{thm:typ_nds} are completely independent of \cite{merlo} for all dimensions $d\geq 2$.
\begin{restatable}{theorem}{proj}\label{thm:proj_category}
	Let $d,l\in\N$ with $l\geq 2$ and $\rho\colon \lipdl\to\lipd$ be the standard projection defined by
	\begin{equation*}
	\rho(f)=f_{1},\qquad f=(f_{1},\ldots,f_{l})\in\lipdl.
	\end{equation*}
	Then there exists an open, dense subset $U$ of $\lipdl$ for which the set $\rho(U)$ is of the first Baire category in $\lipd$.
\end{restatable}
Note that Theorem~\ref{thm:proj_category} also provides an example of a residual subset $S$ of $\lipd$ whose preimage $\rho^{-1}(S)$ under the projection $\rho$ is nowhere dense in $\lipdl$; we may take $S=\lipd\setminus \rho(U)$. For the proof of Theorem~\ref{thm:proj_category}, we require two simple lemmata:
\paragraph{Notation.} In what follows we use again the notation $I_{\eta}(t)$, introduced in Section~\ref{sec:typ_diff_curves}, to denote the open interval $(t-\eta,t+\eta)$.
\begin{lemma}\label{lem:easy_density}
	Let $d,l\in\N$, $\gamma\colon[0,1]\to(0,1)^{d}$ be the length parameterisation of a line segment, $\P$ be a dense subset of $\lipint$, $t_{0}\in (0,1)$, $f=(f_{1},\ldots,f_{l})\in\lipdl$ be mapping with $\lip(f)<1$, $\eps\in(0,1)$ and $j\in\set{1,\ldots,l}$. Then there exist $p\in \P$, $\eta>0$ and $g=(g_{1},\ldots,g_{l})\in\lipdl$ such that
	\begin{enumerate}[(i)]
		\item\label{cond_g-f_supnorm} $\norm{g(x)-f(x)}\leq \eps$ for all $x\in [0,1]$,
		\item\label{cond_g_rest_gamma} $g_{j}\circ \gamma|_{I_{\eta}(t_{0})}=p|_{I_{\eta}(t_{0})}$,
		\item\label{cond_g_res_gamma_1} $g_{1}\circ \gamma=p$ if $l=1$.
	\end{enumerate}	
\end{lemma}
\begin{proof}
	Let $\eta,\sigma>0$ be defined by
	\begin{equation}\label{eq:eta_sigma}
\eta:=\frac{(1-\lip(f)^{2})\eps^{2}}{128\sqrt{d}}, \qquad  \sigma:=\left(\frac{8\sqrt{d}\cdot\eta}{1-\lip(f)^{2}}\right)^{1/2}=\frac{\eps}{4},
	\end{equation}
	choose $p\in\P$ such that 
	\begin{equation}\label{eq:p_approx_f}
	\abs{p(t)-f_{j}(\gamma(t))}\leq \eta \qquad \text{for all }t\in [0,1]
	\end{equation}
	and set
	\begin{equation*}
	J_{l}:=\begin{cases}
	I_{\eta}(t_{0}) & \text{ if }l>1,\\
	[0,1] & \text{ if }l=1.
	\end{cases}
	\end{equation*}
	We define $g=(g_{1},\ldots,g_{l})$ initially on a subset of $[0,1]^{d}$ co-ordinatewise by
	\begin{align}
	g_{j}(x)&=\begin{cases}
	p(t) & \text{if }x=\gamma(t),\, t\in J_{l},\\
	f_{j}(x) & \text{if }x\in [0,1]^{d}\setminus B(\gamma(J_{l}),\sigma),\text{ and}
	\end{cases} \label{eq:g_initial}\\
	g_{i}(x)&=\begin{cases}
	f_{i}(\gamma(t_{0})) & \text{if }x=\gamma(t),\, t\in J_{l},\\
	f_{i}(x) & \text{if }x\in [0,1]^{d}\setminus B(\gamma(J_{l}),\sigma).
	\end{cases}\nonumber
	\end{align}
	for $i\in\set{1,\ldots,l}\setminus\set{j}$. The remainder of the proof is designed primarily for the more complicated case $l>1$. However, it also applies to the case $l=1$; observe that in this case we necessarily have $j=1$ and all sums over $i\neq j$ disappear.

Note that $g|_{[0,1]^{d}\setminus B(\gamma(J_{l}),\sigma)}$ and $g|_{\gamma(J_{l})}$ are $1$-Lipschitz, where the latter case relies heavily on the fact that $\gamma$ is a length parameterisation of a line segment. To verify that this initially defined mapping is globally $1$-Lipschitz on its entire domain, we observe, for $x=\gamma(t)$, $t\in J_{l}$ and $y\in [0,1]^{d}\setminus B(\gamma(J_{l}),\sigma)$,
\begin{multline*}
\norm{g(y)-g(x)}^{2}\leq\sum_{i\neq j}(\abs{f_{i}(y)-f_{i}(x)}+\eta)^{2}+(\abs{f_{j}(y)-f_{j}(x)}+\eta)^{2}\\
\leq \norm{f(y)-f(x)}^{2}+4\sqrt{d}\cdot\eta+2\eta^{2}\leq\left(\lip(f)^{2}+\frac{8\sqrt{d}\cdot\eta}{\sigma^{2}}\right)\norm{y-x}^{2}
= \norm{y-x}^{2},
\end{multline*}
using~\eqref{eq:p_approx_f}, $t\in J_{l}$ and~\eqref{eq:eta_sigma}. By Kirszbraun's Theorem~\cite[Hauptsatz~I]{Kirszbraun1934}, \cite[2.10.43]{Fed}, we may now extend $g$ to the whole of $[0,1]^{d}$ without increasing its Lipschitz constant. Thus, we obtain a mapping $g\in\lipdl$. Note that this mapping $g$ satisfies conclusions~\eqref{cond_g_rest_gamma} and~\eqref{cond_g_res_gamma_1} of the lemma due to~\eqref{eq:g_initial}. To verify conclusion~\eqref{cond_g-f_supnorm}, we first note that the inequality of~\eqref{cond_g-f_supnorm} is trivially valid for all $x\in [0,1]^{d}\setminus B(\gamma(J_{l}),\sigma)$, where we have $f(x)=g(x)$. In the remaining case, $x\in B(\gamma(J_{l}),\sigma)$, we may choose $t\in J_{l}$ with $\norm{x-\gamma(t)}\leq\sigma$. We then derive
\begin{multline*}
\norm{g(x)-f(x)}\leq 2\sigma+\norm{g(\gamma(t))-f(\gamma(t))}\\=2\sigma+\left(\sum_{i\neq j}\abs{f_{i}(\gamma(t_{0}))-f_{i}(\gamma(t))}^{2}+\abs{p(t)-f_{j}(\gamma(t))}^{2}\right)^{1/2}\leq 2\sigma+\sqrt{d}\cdot\eta\leq \eps,
\end{multline*}
using~\eqref{eq:p_approx_f}, $t\in J_{l}$ and~\eqref{eq:eta_sigma}. This verifies~\eqref{cond_g-f_supnorm} and completes the proof of the lemma.
\end{proof}

\begin{lemma}\label{lem:easy_steep}
	Let $f\in\lipint$, $s<t\in [0,1]$, $\tau,\eps\in(0,1)$ and suppose that
	\begin{equation*}
	\norm{f(t)-f(s)}=t-s.
	\end{equation*}
	Then there exists $\delta>0$ such that for every $g\in\lipint$ with $\inorm{g-f}\leq \delta$ the set
	\begin{equation*}
	C:=C_{g,\tau,s,t}=\set{r\in [s,t]\colon g'(r)\geq\tau}
	\end{equation*}
	has positive Lebesgue measure $\leb(C)\geq (1-\eps)(t-s)$.
\end{lemma}
\begin{proof}
	We verify that the assertion of the lemma holds with
	\begin{equation*}
	\delta:=\frac{(1-\tau)(t-s)\eps}{2}.
	\end{equation*}
	Let $g\in\lipint$ with $\inorm{g-f}\leq\delta$. Then
	\begin{multline*}
	t-s-2\delta\leq g(t)-g(s)=\int_{s}^{t}g'(r)\,dr\\
	\leq \int_{[s,t]\setminus C}g'(r)\,dr+\int_{C}g'(r)\,dr\leq \tau(t-s-\leb(C))+\leb(C).
	\end{multline*}
	Rearranging, we obtain
	\begin{equation*}
	\leb(C)\geq t-s-\frac{2\delta}{1-\tau}=(1-\eps)(t-s).
	\end{equation*}
\end{proof}
We are now ready to prove Theorem~\ref{thm:proj_category}.
\begin{proof}[Proof of Theorem~\ref{thm:proj_category}]
	Let $\P$ denote the set of piecewise isometric functions $[0,1]\to\R$ with only finitely many points of non-differentiability. Recall that $\P$ is a dense subset of $\lip_{1}([0,1])$; see \cite{preiss_tiser94}. 
	Let $\Omega$ denote the set of all mappings $f=(f_{1},\ldots,f_{l})\in\lip_{1}([0,1]^{d})$ for which there exist $t_{0}\in (0,1)$, $\eta>0$ and $p\in\P$ such that $f_{2}\circ\gamma|_{I_{\eta}(t_{0})}=p|_{I_{\eta}(t_{0})}$. By Lemma~\ref{lem:easy_density}, the set $\Omega$ is dense in $\lipdl$. We additionally fix a countable, dense subset $\Gamma$ of $\Omega$ and emphasise that $\Gamma$ is trivially also dense in $\lipdl$.

	Let $f\in\Gamma$ and let $t_{0}\in(0,1)$, $\eta>0$ and $p\in\P$ witness that $f\in\Omega$. Since $f_{2}\circ\gamma|_{I_{\eta}(t_{0})}=p|_{I_{\eta}(t_{0})}$ and $p\in\P$, there exist points $s_{f}<t_{f}\in I_{\eta}(t_{0})$ such that $\abs{f_{2}\circ\gamma(t_{f})-f_{2}\circ\gamma(s_{f})}=t_{f}-s_{f}$. Let $\delta_{f}>0$ be given by the conclusion of Lemma~\ref{lem:easy_steep} applied to $f_{2}\circ \gamma\in\lipint$ $s_{f}<t_{f}$, $\tau=3/4$ and $\eps=1/4$. The required open dense subset of $\lipdl$ is now defined by 
	\begin{equation*}
	U=\bigcup_{f\in \Gamma}B(f,\delta_{f}).
	\end{equation*}
To verify that $\rho(U)$ is of the first Baire category in $\lipd$, it suffices to show that each set $\overline{\rho(B(f,\delta_{f}))}$ with $f\in\Gamma$ has empty interior. We fix $f\in\Gamma$. First, observe that
	\begin{equation}\label{eq:closure_claim}
	\overline{\rho(B(f,\delta_{f}))}=\rho(\overline{B}(f,\delta_{f})).
	\end{equation}
	This follows immediately from the continuity of $\rho$ and the fact that $\overline{B}(f,\delta_{f})$ is compact in $\lipdl$, where the latter is a consequence of the Arzel\`{a}-Ascoli Theorem.
	
	Assume that the set given in~\eqref{eq:closure_claim} has non-empty interior. We complete the proof by deriving a contradiction. Fix a function $\widetilde{f}\in \interior\rho(\overline{B}(f,\delta_{f}))$ with $\lip(\widetilde{f})<1$. By Lemma~\ref{lem:easy_density} applied to $\widetilde{f}\in\lipd$ and $l=1$, there exist $q\in\P$ and a function $g_{1}\in \rho(\overline{B}(f,\delta_{f}))$ such that $g_{1}\circ \gamma=q$. Let $(g_{2},\ldots,g_{l})\in\lip_{1}([0,1]^{d},\R^{l-1})$ be such that $(g_{1},g_{2},\ldots,g_{l})\in \overline{B}(f,\delta_{f})$. Then $\inorm{g_{2}\circ \gamma-f_{2}\circ \gamma}\leq \delta_{f}$. Therefore, by the choice of $\delta_{f}$ and Lemma~\ref{lem:easy_steep}, we obtain a set 
	\begin{equation*}
	C=C_{g_{2},3/4,s_{f},t_{f}}\subseteq [s_{f},t_{f}],
	\end{equation*}
	of Lebesuge measure at least $(1-\eps)(t_{f}-s_{f})=3(t_{f}-s_{f})/4>0$, on which $g_{2}\circ \gamma$ is differentiable with $\abs{(g_{2}\circ \gamma)'(t)}\geq 3/4$ for all $t\in C$. However, at all but finitely many points $t\in [0,1]$ we have $\abs{(g_{1}\circ\gamma)'(t)}=\abs{q'(t)}=1$. Therefore, all but finitely many $t\in C$ satisfy
	\begin{equation*}
	\abs{(g_{1}\circ\gamma)'(t)}^{2}+\abs{(g_{2}\circ\gamma)'(t)}^{2}\geq 1+(3/4)^{2}>1.
	\end{equation*}
	Recalling that $\gamma$ is the length parametrisation of a line segment, we see that this is clearly incompatible with $g$ being $1$-Lipschitz.
\end{proof}

\bibliographystyle{plain}
\bibliography{biblio}
\enddocument